\documentclass[10pt]{amsart}
\usepackage{graphicx,amssymb,amsfonts,amsmath,amsthm,newlfont}
\usepackage{epsfig}
\usepackage[all,2cell]{xy} \UseAllTwocells \SilentMatrices

\newtheorem{thm}{Theorem}
\newtheorem{cor}[thm]{Corollary}
\newtheorem{lem}[thm]{Lemma}
\newtheorem{prop}[thm]{Proposition}
\newtheorem{caveat}[thm]{Caveat}
\newtheorem{propdef}[thm]{Proposition-Definition}
\theoremstyle{definition}
\newtheorem{defn}[thm]{Definition}

\newtheorem{example}[thm]{Example}
\newtheorem{ex}[thm]{Examples}
\theoremstyle{remark}
\newtheorem{rem}[thm]{Remark}
\numberwithin{equation}{section}

\newcommand{\Z}{\mathbb Z}
\newcommand{\C}{\mathbb C}

\newcommand{\R}{\mathbb R}
\newcommand{\N}{\mathbb N}
\newcommand{\Hp}{\mathbb H}
\newcommand{\Pro}{\mathbb P}

\newcommand{\sign}{\varepsilon}

\newcommand{\gr}{\mathrm {gr}}

\font \rus= wncyr10
\newcommand{\sha}{\, \hbox{\rus x} \,}

\newcommand{\Mod}{\mathfrak{M}}

\newcommand{\Q}{\mathbb Q}
\newcommand{\Li}{\mathrm{Li}}

\newcommand{\To}{\longrightarrow}

\newcommand{\G}{\mathbb{G}}
\newcommand{\E}{\mathcal{E}}

\newcommand{\e}{\mathbf{e}}

\newcommand{\Fo}{\mathcal{F}}

\newcommand{\Ss}{\mathcal{S}}

\newcommand{\Sym}{\mathfrak{S}}
\newcommand{\Aut}{\mathrm{Aut}}

\newcommand{\reg}{\mathrm{reg}}
\newcommand{\EE}{\mathtt{\,E}}

\newcommand{\Spec}{\mathrm{Spec}\,}
\newcommand{\Image}{\mathrm{Im}\,}
\newcommand{\sdots}{.\,.\,}

\newcommand{\En}{\E^{(n)}}

\newcommand{\Ao}{\mathcal{A}}
\newcommand{\Mo}{\mathcal{M}}

\newcommand{\dt}{\cdot}

\newcommand{\B}{\mathbb{B}}

\newcommand{\x}{\mathsf{x}}
\newcommand{\y}{\mathsf{y}}

\newcommand{\sing}{\mathrm{sing}}

\newcommand{\basepnt}{\varrho}

\def\0skip{\vskip 2pt}
\def\1skip{\vskip 4pt}
\def\2skip{\vskip 10pt}
\def\3skip{\vskip 0.15in}
\def\4skip{\vskip 0.20in}
\def\5skip{\vskip 0.25in}
\def\6skip{\vskip 0.30in}

\title[Multiple Elliptic Polylogarithms]{Multiple Elliptic Polylogarithms}

\author{Francis Brown and Andrey Levin}

\begin{document}

\begin{abstract}
We  study the de Rham fundamental group of  the configuration space $\E^{(n)}$ of $n+1$ marked  points on a complex  elliptic curve $\E$, and define
multiple elliptic polylogarithms. These are    multivalued functions on $\E^{(n)}$ with unipotent monodromy, and are constructed by a general averaging procedure.
 We show that all iterated integrals on $\E^{(n)}$, and in particular the periods of the unipotent fundamental group of the punctured curve $\E \backslash \{0\}$, can be expressed in terms of these functions.
\end{abstract} 

\maketitle

\section{Introduction}

\subsection{Motivation} Iterated integrals on the moduli space $\Mod_{0,n}$ of curves of genus $0$ with $n$ ordered marked points can be expressed in terms of  multiple
polylogarithms. These are   defined for  $n_1,\ldots, n_r\in \N$ by 
\begin{equation}\label{intromultpoly} 
\Li_{n_1,\ldots, n_r} (x_1,\ldots, x_r) = \sum_{0<k_1<\ldots< k_r} {x_1^{k_1}\ldots x_r^{k_r} \over k_1^{n_1}\ldots k_r^{n_r}} \quad  \hbox{where }  |x_i|<1\ ,
\end{equation}
 and   have many applications from arithmetic geometry to quantum field theory. 
 By specializing $(\ref{intromultpoly})$ at  $x_i=1$, one obtains 
 the multiple zeta values 
 \begin{equation} \label{intromzv}
\zeta(n_1,\ldots, n_r) =  \sum_{0<k_1<\ldots< k_r} {1 \over k_1^{n_1}\ldots k_r^{n_r}}\ , 
\end{equation}
which converges for $n_r\geq 2$. These  are of particular interest   since they are the periods of the  fundamental group of 
$\Pro^1\backslash \{0,1,\infty\}$, and generate  the periods of  all mixed Tate motives  over $\Z$.

The goal of this paper is to construct the elliptic analogues of the multiple polylogarithms and to set up the necessary algebraic and analytic background required  to study  multiple elliptic  zeta values.
The former are iterated integrals on  the configuration space $\E^{(n)}$ of $n+1$ marked points on a complex elliptic curve,  i.e., the 
 fiber of the map $\Mod_{1,n+1}\rightarrow \Mod_{1,1}$, where $\Mod_{1,m}$ denotes  the  moduli space of curves  of genus 1 
with $m$ marked points. They generalize the classical elliptic polylogarithms studied in \cite{Andrey}, and are the universal periods of  unipotent variations of mixed elliptic Hodge structures. 

In a sequel to this paper, we shall study the multiple elliptic zeta values, obtained by  specializing   multiple elliptic polylogarithms to the zero section of the universal elliptic curve. They define multivalued functions on $\Mod_{1,1}$ which  degenerate to ordinary multiple zeta values at the cusp. The existence of these functions  sheds light on the structural relations between ordinary multiple zeta values,  and in particular, the relation between double zetas and  period polynomials for cusp forms \cite{GHK}.

\subsection{The rational case}   Firstly we recall the definition of iterated integrals \cite{Ch1}. Let $M$ be a smooth real manifold, and let $\omega_1,\ldots,\omega_n$ denote 
smooth 1-forms on $M$. Let $\gamma:[0,1]\rightarrow M $ be a smooth path, and write  $\gamma^{*} \omega_i= f_i(t) dt$ for some smooth functions $f_i:[0,1]\rightarrow \R$, where $1\leq i\leq n$.  
The iterated integral of $\omega_1,\ldots, \omega_n$ is defined by  
\begin{equation}\label{introitintdef} 
\int_{\gamma} \omega_1 \ldots \omega_n = \int_{0\leq t_n \leq \ldots \leq t_1 \leq 1} f_1(t_1) \ldots f_n(t_n) \, dt_1 \ldots dt_n\ .\end{equation}
Now let    $M=\Pro^1\backslash \{0,1,\infty\}$, and let  $\omega_0= {dz\over z}$ and $\omega_1 = {dz \over 1-z}$. Let $0<z<1$, and denote  the straight  path 
 from $0$ to $z$ by $\gamma_z$. The initial point  $\gamma(0)$ does not in fact lie in $M$, but the following iterated integral still makes sense nonetheless,
and gives 
\begin{equation} \label{introIz}
\int_{\gamma_z} \underbrace{\omega_0 \ldots \omega_0}_{n_r-1} \omega_1  \ldots \underbrace{\omega_0 \ldots \omega_0}_{n_1-1} \omega_1=\Li_{n_1,\ldots, n_r}(1,\ldots, 1,z) \ .
\end{equation}
This is easily proved by a series expansion of the forms $\omega_1$. 
The periods of the fundamental torsor of paths of $M$ from $0$ to $1$ (with tangential basepoints $1$, $-1$), are obtained by taking the limit as $z\rightarrow 1$. In the case $n_r\geq 2$, this yields
 \begin{equation} \label{introitintmzv}
 \int_0^1 \underbrace{\omega_0 \ldots \omega_0}_{n_r-1} \omega_1  \ldots \underbrace{\omega_0 \ldots \omega_0}_{n_1-1} \omega_1=\zeta(n_1,\ldots, n_r)\ ,
 \end{equation}  
as first observed by Kontsevich. Similarly, one can define the regularized iterated integral from $0$ to $1$ of any word in the one-forms $\omega_0, \omega_1$, and in every case, it  is easy to show that it is  a linear combination of multiple zeta values. 
 
To verify that all (homotopy invariant) iterated integrals on $M$ are expressible in terms of multiple polylogarithms requires Chen's 
reduced bar construction.
    By Chen's general theory, the iterated integrals on 
  a manifold $M$  are
  described by the zeroth cohomology of the bar construction on the de Rham complex of  $M$. 
  To write this down explicitly for $M=\Pro^1\backslash \{0,1,\infty\}$, we can use the  rational model 
 $$\Q \oplus (\Q \,\omega_0 \oplus \Q \,\omega_1)  \hookrightarrow \Omega^{\dt}_{DR} (\mathbb{P}^1\backslash \{0,1,\infty \};\Q)$$
 which is a quasi-isomorphism of differential graded algebras. From this one 
   deduces that  $H^0(\B(\Omega^{\dt}_{DR}(\Pro^1\backslash\{0,1,\infty\};\Q))) \cong T^c(\Q\,\omega_0 \oplus \Q\,\omega_1)$, where $\B$ is the bar complex, and $T^c$ denotes the tensor coalgebra. This is   the $\Q$-vector space generated by  words in the forms $\omega_0, \omega_1$ (equipped with the shuffle product and the coproduct for which $\omega_0, \omega_1$ are primitive),  and leads to    integrals of the form  $(\ref{introitintmzv})$. The upshot is that the periods of the unipotent fundamental groupoid 
   $\pi^{un}_1(\Pro^1\backslash \{0,1,\infty\},0^+,z)$ are  multiple polylogarithms $(\ref{introIz})$, and the periods 
   of $\pi^{un}_1(\Pro^1\backslash \{0,1,\infty\},0^+,1^-)$ are multiple zetas $(\ref{intromzv})$.
   
   We need to generalize this  picture  to higher dimensions as follows. For $r\geq 1$, let 
  \begin{equation}\label{intromoddef} \Mod_{0,r+3} (\C)= \{ (t_1,\ldots, t_r) \in \C^r: t_i \neq 0,1 \ , t_i \neq t_j\}
  \end{equation}
  denote the moduli space of genus $0$ curves with $r+3$ marked points. One can likewise write down 
  a rational model for the de Rham complex in terms of the one-forms 
  $${dt_i -dt_j \over t_i -t_j}  \ , \  {dt_i \over 1- t_i} \ , \ {dt_i \over t_i } $$ 
  which satisfy certain quadratic relations due to Arnold. 
  Forgetting a marked point defines a  fibration $\Mod_{0,r+3}\rightarrow \Mod_{0,r+2}$,  and by   general properties of the bar construction of fibrations, one can likewise write down 
  all homotopy invariant iterated integrals on $\Mod_{0,r+3}$ \cite{Br} and show that they are expressible in terms of the functions
  \begin{equation}\label{introIfuncs}
  I_{n_1,\ldots, n_r} (t_1,\ldots, t_r) = \Li_{n_1,\ldots, n_r} \big( {t_1 \over t_2}, \ldots, {t_{r-1}\over t_r}, t_r\big) \ .
  \end{equation}

 The purpose of this paper is to generalize the above picture for $\Pro^1\backslash \{0,1,\infty\}$ to the case of a punctured complex elliptic curve $\E^{\times}$. In particular, we compute the periods of 
 $\pi^{un}_1(\E^{\times}, \basepnt,\xi)$, where $\rho,\xi \in \E^{\times}$  are finite basepoints (the case of tangential basepoints is similar and will be  postponed to a later paper).
 There are two parts:  first,  to write down the iterated integrals generalizing the left hand side of 
$(\ref{introIz})$  using Chen's general theory, and the second is to construct multiple elliptic  polylogarithm functions which correspond to the right-hand side of 
$(\ref{introIz})$.  In so doing, we are   forced  to consider the  higher-dimensional  configuration spaces $\En$ to construct the functions on $\E^{\times}$.

 \subsection{The elliptic case} Let $\E$ be an elliptic curve, viewed as the analytic manifold  $\C/   \Z  \tau \oplus \Z$, where $\tau \in \C$ satisfies $\Image (\tau)>0$.
 We first require a model for the de Rham complex on $\E^{\times}$. 
For this, we  construct a universal family of smooth one-forms
 \begin{equation} \label{introchoiceofmassey} \nu,\omega^{(0)}, \omega^{(1)},\omega^{(2)},\ldots \quad  \in \Ao^1(\E^{\times}) \ , 
 \end{equation}
  where $\omega^{(0)},\nu$ are closed and form a basis of $H^1(\E^{\times})$ which is compatible with its Hodge structure. 
 The forms $\omega^{(i)}$, $i\geq 1$ can be viewed as higher Massey products satisfying
  $$d\omega^{(i)} = \nu \wedge \omega^{(i-1)} \qquad \hbox{for } i \geq  1\ .$$
 Note that \emph{a priori}   $\E^{\times}$ has no natural $\Q$-structure on its de Rham complex.  However, the 
forms $(\ref{introchoiceofmassey})$ have good modularity and rationality properties as a function of the moduli $\tau$, and  there are  good  reasons
to take
$$X= \hbox{graded }  \Q \hbox{-algebra spanned by  the }  \nu, \omega^{(i)}, i\geq 0$$
as 
 a $\Q$-model for the $C^{\infty}$-de Rham complex on $\E^{\times}$, (indeed, $X\otimes_{\Q}\C \hookrightarrow \Ao^{\dt}(\E^{\times})$ is a quasi-isomorphism).
We  also define a higher-dimensional  model for the configuration space $\En$ of $n+1$ points on $\E$. It is an elliptic  version of Arnold's theorem describing the cohomology of the configuration space of $n$ points in $\Pro^1$.
 
 The next stage is to write down the  bar construction of $X$, 
 which defines a $\Q$-structure on the  iterated integrals on $\E^{\times}$. The bar construction has a filtration by the  length, and the associated graded is just the tensor coalgebra on $\omega^{(0)}$ and $\nu$:
 \begin{equation} \label{introbarX}
 \gr^{\ell} H^0 (\B(X)) \cong T^c(\Q\,\omega^{(0)} \oplus \Q\, \nu) \ . \end{equation} 
  The  Massey products $\omega^{(i)}$, for $i\geq 1$, give a canonical way to lift an element  of $(\ref{introbarX})$  to $H^0 (\B(X))$, and thus enable us to  write down explicitly all 
  the iterated integrals on $\E^{\times}$. These
 are indexed by any word in the two one-forms $\omega^{(0)}$ and $\nu$.
  The Hodge filtration on the space of iterated integrals is related to  the number of $\nu$'s.
  This completes the algebraic  description of the iterated integrals on $\E^{\times}$.

The main problem is then to write down explicit formulae for these  iterated integrals, and for this we write  the elliptic curve via its Jacobi uniformization
$$\E \cong \C^{\times} /q^{\Z}\ ,$$
where $q=\exp(2\pi i \tau)$.
In order to construct  multivalued functions on $\E^{\times}$, the basic idea is to average a multivalued function on $\C^{\times}$ with respect to  multiplication by $q$
 as  was done for the classical polylogarithms \cite{Bloch, Andrey}. 
 However, applying this idea naively to the multiple polylogarithms in one variable $(\ref{introIz})$ does not lead to elliptic functions. 
 Instead, the correct approach is to view the multiple polylogarithms in $r$ variables  $(\ref{intromultpoly})$, or rather their variants $(\ref{introIfuncs})$,  as multivalued functions on 
 $$\Mod_{0,r+3} \cong  \underbrace{\C^{\times}\times \ldots \times \C^{\times} }_r\backslash  \hbox{diagonals} $$
 and average with respect to the group $q^{\Z}\times \ldots \times q^{\Z}$ ($r$ factors).  Since polylogarithms have logarithmic singularities at infinity, the straightforward average  diverges, and so instead one must  take a weighted average with respect to some auxilliary parameters $u_i$ to dampen the singularities.  In short, one considers the functions
  \begin{equation}\label{introweightsum} 
  \sum_{m_1,\ldots, m_r\in \Z} u_1^{m_1}\ldots u_r^{m_r}  I_{n_1,\ldots, n_r}(q^{m_1}t_1,\ldots, q^{m_r} t_r)\end{equation}
  which converge uniformly under some conditions on the  $u_i$.  
   A considerable part of this paper is devoted to studying the 
 structure of the poles of $(\ref{introweightsum})$ in the $u_i$ variables, which are related to the geometry of 
 $\overline{\Mod}_{0,r+3}$ and the asymptotics of the polylogarithms $(\ref{introIfuncs})$ at infinity.  Finally,   writing $u_i=\exp(2i\pi \alpha_i)$ for $1\leq i\leq r$, 
  the multiple elliptic polylogarithms are defined to be the coefficients of (the finite part of) $(\ref{introweightsum})$ with respect to the $\alpha_i$. 
 The analysis involved in this  averaging procedure is quite general and should apply to a  class of functions of finite determination and moderate growth on certain toric varieties.

 The functions obtained in this way are multivalued functions on 
 $\En$.  By allowing some of the arguments $t_i$ of $(\ref{introweightsum})$ to degenerate to $1$, we obtain 
multivalued functions on   $\E^{\times}$. By computing the differential equations satisfied by these functions,   we see that they are iterated integrals in the forms
$(\ref{introchoiceofmassey})$
and, using the description of the bar construction of $X$, we deduce that all the iterated integrals on $\E^{\times}$ are of this form.
 
 An important caveat is that  the rational structure that we define on the de Rham fundamental group on $\E^{\times}$  is the correct $\Q$-structure from the point of view of averaging functions in genus zero, but   does \emph{not} coincide with the canonical $\Q$-structure in the special case
 when $\E$ is defined over $\Q$ (see \S5 of \cite{LR}).

 \subsection{Plan of the paper}
First,   $\S2$ consists of reminders on multiple polylogarithms and the moduli space of curves of genus $0$ which are used throughout the paper.
Thereafter, the exposition splits into  two parts - the first part, consisting of sections $3$,$4$, and $5$, concerns the de Rham complex of differential forms on $\En$. 
The second, consisting of  sections $6$, $7$ and $8$, concerns the procedure for averaging multiple polylogarithms. Since this is quite involved, we  give a separate  overview of the method in $\S6.1$.

In $\S3$ we use the Kronecker series  (Proposition-Definition \ref{Fproperties}) to define the fundamental one-forms  $(\ref{introchoiceofmassey})$.
 In $\S4$, we define some  differential graded algebras  $X_n$ and prove by a Leray spectral sequence argument that they are $\Q$-models for the 
 de Rham complex on $\En$.  In $\S5$ we use the models $X_n$ to study Chen's reduced bar construction on $\En$, and hence obtain an algebraic 
 description for the iterated integrals on $\En$. Some of the results of this section require some generalities on the bar construction of differential graded algebras which we decided to relegate to a separate paper \cite{BDR}.

In $\S6$, we study the general averaging procedure for functions on $\Mod_{0,n}(\C)$.  This requires constructing a certain partial compactification of $\Mod_{0,n}$ and analyzing the asymptotics of series in the neighbourhood of boundary divisors.   We apply this formalism to the classical multiple polylogarithms in \S7.   In $\S8$, which
is logically independent from the rest of the paper, we compute the asymptotics of the Debye polylogarithms at infinity in terms of a certain coproduct. 
The Debye multiple polylogarithms (definition \ref{defnclassicalDebye})  are essentially generating series  of multiple polylogarithms and are useful for simplifying many formulae.
In \S9, we treat the case of the classical elliptic polylogarithms (depth 1) and the double elliptic polylogarithms (depth 2) in detail.   
The two parts of the story recombine in \S10 where we prove that all iterated integrals on $\E^{\times}$, with respect to finite basepoints,  are obtained by averaging. 

\subsection{Related work} 
One of many motivations for this paper is the study  of mixed elliptic motives.  Since the Beilinson-Soul\'e conjectures are currently unavailable  in this case, 
our goal was to tease out the elementary consequences of such a  theory, and in particular, write down  the  underlying numbers and functions in the belief that they will find applications in other parts of mathematics. We learned  at a conference in Bristol in 2011 that there has been recent progress in 
 constructing categories of mixed elliptic motives \cite{Terasoma, Pat}, and universal elliptic motives \cite{HM} with similar goals. In this paper we completely neglected
 the Betti side,  which is dual to the bar construction, and its  relation to quantum groups and stable derivations. This  is treated in  \cite{CEE, En, Poll}. 
 Somewhat further afield, it may also be helpful to point out related work in the  
 profinite setting \cite{Nak}, and possible diophantine applications  \cite{Kim}.
\\

 \emph{Acknowledgements}.
%
%
The first author is partially supported by ERC grant 257368. The second author is partially supported by
`the National Research University Higher School of Economics' academic fund number 11-01-0133 for 2012-2013, and also 
AG Laboratory GU-HSE, RF government
grant, ag. 11 11.G34.31.0023. We thank the  Fondation des Sciences Math\'ematiques de Paris,  the  Institut Poncelet in Moscow, the MPIM, and Humboldt University for support 
and hospitality.

\section{Preliminaries: the rational case}

\subsection{Standard coordinates on $\Mod_{0,n+3}$} \label{sectStandardCoords}
 Let $\Mod_{0,n+3}$ denote the moduli space of curves of genus 0 with $n+3$ ordered marked points. By placing three of the marked points at $0,1,\infty$, it can be identified with an affine  hyperplane  complement: 
\begin{equation} \label{Mondef}
\Mod_{0,n+3}\cong  \{ (t_1,\ldots, t_n) \in \Pro^1\backslash\{0,1,\infty\}: t_i\neq t_j \}\ .
\end{equation}
We refer to the coordinates $t_i$ as simplicial coordinates. We will often write $t_{n+1}=1$. 
There is a smooth compactification  $\overline{\Mod}_{0,n+3}$, such that the complement $\overline{\Mod}_{0,n+3}\backslash \Mod_{0,n+3}$ is a normal crossing  divisor.
Its irreducible components are indexed by  partitions $S\cup T$ of the set of marked points, which we denote by $S|T$, or, $T|S$,  where $S\cap T=\emptyset$ and $|S|,|T|\geq 2$.
 The corresponding divisor is isomorphic to $\overline{\Mod}_{0,|S|+1}\times \overline{\Mod}_{0,|T|+1}$.

\subsection{Multiple polylogarithms} \label{sectMultiplePolyreminder} These are defined for $n_1,\ldots, n_r\in \N$  by 
\begin{equation}\label{Lindef} \Li_{n_1,\ldots, n_r}(x_1,\ldots, x_r) = \sum_{0<k_1<\ldots <k_r} {x_1^{k_1} \ldots x_r^{k_r} \over k_1^{n_1}\ldots k_r^{n_r}}
\end{equation}
which  converges absolutely for $|x_i|<1$, and therefore defines a  family of holomorphic functions in the neighbourhood of the origin.
We can also write these functions in simplicial coordinates, using standard notations:  
\begin{equation}\label{Idef} 
I_{n_1,\ldots, n_r} (t_1,\ldots, t_r) = \Li_{n_1,\ldots, n_r} \Big( {t_1 \over t_2} ,\ldots, {t_{{r-1}}\over t_{r}}, t_{r}\Big) \ .\end{equation}
We denote the quantities  $N=n_1+\ldots +n_r$ and $r$ by the weight  and the depth, respectively. There is an iterated  integral representation
\begin{equation}\label{itintrep}
I_{n_1,\ldots, n_r}(t_1,\ldots, t_r) =(-1)^r \int_{0\leq \sigma_1 \leq \ldots \sigma_r\leq 1} {d\sigma_1\over \sigma_1-\rho_1} \ldots  {d\sigma_{N} \over \sigma_N-\rho_N} \ ,
\end{equation}
where  $(\rho_1,\ldots, \rho_N) = (t_1^{-1}, 0^{n_1-1}, t_2^{-1},  0^{n_{2}-1},\ldots, t_{r}^{-1}, 0^{n_r-1})$, $N=n_1+\ldots +n_r$,  and $0^n$ denotes a string of $n$ $0$'s.
This is easily proved by expanding the integrand into a geometric series. 
One can deduce that they extend to multivalued functions on $\Mod_{0,n+3}(\C)  \subset (\C^{\times})^n$.
The following differential equation is easily verified from  $(\ref{Lindef})$.
\begin{eqnarray} \label{I11diff}
\qquad \qquad dI_{1,\ldots, 1}(t_1,\ldots, t_n)  =    \sum_{k=1}^n \big[dI_1\Big({ t_k \over t_{k+1}}\Big) - dI_1\Big({t_k \over t_{k-1}}\Big)\big] I_{1,\ldots, 1}(t_1,\ldots, \widehat{t_k},\ldots, t_n) 
\end{eqnarray}
where  by convention we take $dI_1(t_1/0) = 0$,  $t_{n+1} =1$ and $t_0=0$. The differential equations for the multiple polylogarithms $I_{n_1,\ldots, n_r}(t_1,\ldots, t_r)$ in the general case are easily computed and left to the reader, since they are not required in the sequel.

\begin{defn}  \label{defnclassicalDebye} The generating series of multiple Debye polylogarithms is:
$$\Lambda_r(t_1,\ldots, t_r;\beta_1,\ldots, \beta_r) = t_1^{-\beta_1}\ldots t_r^{-\beta_r} \!\!\sum_{m_1,\ldots, m_r\geq 1} I_{m_1,\ldots, m_r}(t_1,\ldots, t_r) \beta_1^{m_1-1}\ldots \beta_r^{m_r-1}$$
\end{defn}
One easily verifies from   $d\Li_n(t) = t^{-1}\Li_{n-1}(t)$, valid for $n\geq 2$, that
$$d\Lambda_1(t;\beta) = {t^{-\beta} } d\Li_1(t)\ .$$
In general, they satisfy a differential equation  which is entirely analogous to   equation $(\ref{I11diff})$ for the multiple 1-logarithm. Namely,  $d\Lambda_r(t_1,\sdots, t_r;\beta_1,\sdots, \beta_r)=$
\begin{eqnarray} 
&=& \sum_{k=1}^{r} d\Lambda_1\Big({ t_k \over t_{k+1}},\beta_k\Big) \Lambda_{r-1}(t_1,\sdots, \widehat{t}_k, \sdots, t_r; \beta_1,\sdots, \beta_{k}+\beta_{k+1},\sdots, \beta_r) \nonumber \\
&-&  \sum_{k=2}^{r} d\Lambda_1\Big({t_k \over t_{k-1}},\beta_k\Big) \Lambda_{r-1}(t_1,\sdots, \widehat{t}_k, \sdots, t_r; \beta_1,\sdots, \beta_{k-1}+\beta_{k},\sdots, \beta_r) \label{Debeqn}
\end{eqnarray}
where $\beta_{r+1}=0$, so the  last term in the first line is $d\Lambda_1(t_r;\beta_r)\Lambda_{r-1}(t_1,.\,. \,, t_{r-1};\beta_1,.\,.\,, \beta_{r-1})$.

\subsection{Asymptotics of regular nilpotent connections}
Let $X$ be a smooth projective complex variety, let $D\subset X$ be a simple normal crossing divisor, and let $U=X\backslash D$.
Let $V\subset X$ be a  simply connected open set  and let 
$z_1,\ldots, z_n$ denote local coordinates on $V$ such that $V\cap D=\cup_{i=1}^k \{z_i=0\}$, for some $k\leq n$. 

\begin{defn} \label{propunipgrowth} Let $f$ be a  multivalued holomorphic function on $U$ (i.e. $f$ is holomorphic on a  covering of $U$).  We say that  $f$ has locally unipotent mondromy  (or is locally unipotent) on $V$ if it  admits a finite expansion:
\begin{equation} \label{fexpansion} 
f(z_1,\ldots, z_n) = \sum_{I=(i_1,\ldots,i_k)} \log^{i_1}(z_1)\ldots \log^{i_k}(z_k) f_I (z_{1},\ldots, z_n)\ , 
\end{equation}
where   $f_I(z_1,\ldots, z_n)$ is holomorphic on $V$.  We say that  a multivalued function $f$  on $X\backslash D$ is  unipotent  if it is everywhere locally unipotent. \end{defn}

The main class of  functions which are studied in this paper, and in particular the multiple polylogarithms, are unipotent.
The expansion $(\ref{fexpansion})$ can be  characterized by a property that $f$ has `moderate' growth near $D$ (in particular, no poles), and  
that for sufficiently large $N$,  $(\mathcal{M}_{i_1}-id)\ldots (\mathcal{M}_{i_N}-id) f=0$ for all  
$i_1,\ldots, i_N \in \{1,\ldots, k\}$,  where $\mathcal{M}_i$ denotes analytic continuation around a small loop encircling $z_i=0$.

\section{Differential forms on $\En$}

 \subsection{Basic notations.} 
Let $\e(z)$ denote the function $\e(z)=\exp(2\pi i z)$.
In accordance with  \cite{W},  Greek letters $\xi$ and $\eta$ denote coordinates  on $\Bbb C$, the letter $\alpha$ will denote a formal variable, and $\tau$ will
denote a point of the upper half-plane $\Bbb H= \{\tau\in \C: \Image(\tau)>0\}$.
Put $z=\e(\xi)$, $w=\e(\eta)$ and $q=\e(\tau)$. Then 
$z,w\in \Bbb C^*$, and  $0<|q|<1$.

\subsection{Uniformization} \label{sectUniformization}
We represent a complex elliptic curve  $\E=\C/(\tau \Z+\Z)$, where $\tau\in \Hp$,  and $q=\e(\tau)$,  via its Jacobi uniformization 
$$\E \overset{\sim}{\To}\C^* /q^{\Z}\ .$$
Let  $\xi$ (resp. $z=\e(\xi)$)
 denote the coordinate on $\E$ (resp. $\C^*$). The punctured curve $\E^{\times}=\E\backslash \{0\}$ is isomorphic to   $\C^*\backslash\{q^{\Z}\}/q^{\Z}$. For $n\geq 1$, let $\E^{(n)}$ denote the configuration space of $n+1$ distinct points on $\E$ modulo translation by $\E$. Thus 
 $$\E^{(n)} \cong \{(\xi_1,\ldots, \xi_n)\in  (\E\backslash \{0\})^n : \xi_i \neq \xi_j \hbox{ for } i\neq j\}\ ,$$
 and has an action of $\Sym_{n+1}$ which permutes the marked points. Setting  $t_i=\e(\xi_i)$ for $1\leq i \leq n$, and setting $t_{n+1}=1$,  gives an isomorphism 
  $$\E^{(n)}\cong  \{(t_1,\ldots, t_n)\in \C^*, t_i \notin q^{\Z} t_j \hbox{ for } 1\leq  i<j \leq n+1\}/q^{\Z^{n}}\ .$$
The set on the right-hand side is the largest open subset of $\Mod_{0,n+3}(\C)$   stable under translation by $q^{\Z^{n}}$. 
 The symmetry group $\Sym_{n+1}$ of $\E^{(n)}$   can thus be identified with  the subgroup of $\Aut(\Mod_{0,n+3}(\C))\cong\Sym_{n+3}$ which fixes the marked points $0$ and $\infty$.

In order to fix branches when considering multivalued functions, and  to  ensure convergence when averaging functions on $\Mod_{0,n+3}(\C)$, we must fix
 certain domains in $\E^{(n)}$.  Let $D$ be  the standard open fundamental domain for $\Z+\tau \Z$ (the parallelogram with corners
 $0,1,\tau, 1+\tau$),  and let 
\begin{equation}\label{Udef}
U=\{(\xi_1,\ldots, \xi_n)\in D^n:  \Image(\xi_n)<\ldots<\Image(\xi_1)\}\end{equation}

\subsection{Elliptic functions.} Let  $\theta(\xi,\tau)$ denote   ``two thirds of
the Jacobi triple formula'':
\begin{equation} \theta(\xi,\tau)=q^{1/12}(z^{1/2}-z^{-1/2})\prod_{j=1}^{\infty}(1-q^jz)
\prod_{j=1}^{\infty}(1-q^jz^{-1})=\frac{\theta_{11}(\xi,\tau)}{\eta(\tau)},
\end{equation}
where $\theta_{11}(\xi,\tau)$ is the standard odd elliptic theta function
and $\eta(\tau)$ is the Dedekind $\eta$-function $q^{1/24}
\prod_{j=1}^{\infty}(1-q^j)$. Recall from \cite{W}  that the Eisenstein summation  of a double
series  $(a_{m,n})_{m,n\in \Z}$ is defined by:
$${\sum_{m,n}} \!\! \phantom{|}_e \, \,  a_{m,n}  =\lim_{N\to\infty}\lim_{M\to\infty}  \sum_{n=-N}^N\sum_{m=-M}^M a_{m,n}$$
Define the Eisenstein functions $E_j(\xi,\tau)$ and the Eisenstein series
$e_j(\tau)$, for $j\geq 1$, by
$$E_j(\xi,\tau)={\sum_{m,n}} \!
\phantom{|}_e\frac1{(\xi+m+n\tau)^j},\quad e_j(
\tau)={\sum_{m,n}}^\prime\! \!\! \phantom{|}_e\frac1{(
m+n\tau)^j},$$ where $\prime$  means that we omit $(m,n)=(0,0)$ in the summation.

\begin{lem} \label{lemEj} It follows from the definitions that  for $j\geq 1$, 
$${\partial \over \partial \xi} E_j(\xi,\tau) = -jE_{j+1}(\xi,\tau)\ , \qquad   {\partial\over \partial \xi}
\log (\theta(\xi,\tau))= E_{1}(\xi,\tau), $$
and $E_1(\alpha,\tau)=1/\alpha-\sum_{k=0}^\infty
e_{k+1}(\tau)\alpha^k$.  The series $e_j(\tau)$
vanish for odd indices $j$. \end{lem}
The Weierstrass function $\wp$ is equal to $E_2-e_2$, and
$\wp'=-2E_3$. The coefficients of the Weierstrass  equation
${\wp'}^2=4 \wp^3-g_2\wp-g_3$ are given by  $g_2=60e_4, g_3=140e_6$.

\subsection{The Kronecker function.} See also \cite{Andrey,LR, Z} for further details.

\begin{propdef}  \label{Fproperties} The following three definitions are
equivalent:
 \begin{eqnarray}
i) \qquad  F(\xi,\eta,\tau) &= &\frac{\theta'(0)\theta(\xi+\eta)}{\theta(\xi)\theta(\eta)}\ , \nonumber \\
ii) \qquad F(\xi,\eta,\tau) &= &-2\pi i\left(\frac{z}{1-z}+\frac1{1-w}+\sum_{m,n>0}(z^mw^n-z^{-m}w^{-n})q^{mn}\right)\ , \nonumber \\
iii) \qquad F(\xi,\alpha,\tau)&=& \frac1\alpha\exp\left(-\sum_{j\geq 1} \frac{(-\alpha)^j}j(E_j(\xi,\tau)-e_j(\tau))\right) \nonumber \ .
\end{eqnarray}
\end{propdef}
The equivalence of  $(i)$ and $(ii)$  is proved in
 \cite{W}.  The equivalence of  $(i)$ and $(iii)$ 
follows  by computing the logarithmic derivative of $F$,  from the relationship between $E_1$ and
$\log (\theta)$ (lemma \ref{lemEj}), and the  Taylor 
expansion of $E_1$ at a  point $\alpha$.  The following properties of the Kronecker function $F$ will be important for the sequel.

\begin{prop} \label{prop3prop}  $F(\xi,\eta,\tau)$ has the following properties:

i) Quasi-periodicity with respect to  $\xi\mapsto \xi+1$ and $\xi \mapsto \xi +\tau$:
$$F(\xi+1,\eta,\tau)= F(\xi,\eta,\tau) \qquad F(\xi+\tau,\eta,\tau)=
w^{-1}F(\xi,\eta,\tau)$$

ii)  The mixed heat equation: $$2\pi i {\partial
F\over \partial \tau}={\partial^2 F \over \partial \xi\partial\eta}\ .$$

iii) The Fay identity:  
$$F(\xi_1,\eta_1,\tau)F(\xi_2,\eta_2,\tau)=
F(\xi_1,\eta_1+\eta_2,\tau)F(\xi_2-\xi_1,\eta_2,\tau)+
F(\xi_2,\eta_1+\eta_2,\tau)F(\xi_1-\xi_2,\eta_1,\tau)\ .$$
\end{prop}
\begin{proof}
The quasi-periodicity is immediate from the first definition of $F$.
The mixed heat equation follows from the second definition 
of $F$. The last statement is a consequence of the third representation of $F$,
and the Fay trisecant equation (see \cite{Mum}).
\end{proof}

The following formula is an  easy corollary of $iii)$:
\begin{equation}\label{Faydiff}
F(\xi,\alpha_1)
F^\prime_2(\xi,\alpha_2)-F^\prime_2(\xi,\alpha_1) F(\xi,\alpha_2)
=F(\xi,\alpha_1+\alpha_2)(E_2(\alpha_1)-E_2(\alpha_2))\ , \end{equation}
where $F^\prime_2$ denotes  the derivative of $F$ with respect to its
second argument.

\subsection{Massey products on  $\En$}  \label{sectMassey} We use the Eisenstein-Kronecker series $F$ to write down some explicit  one-forms on $\En$.
First consider a single elliptic curve $\E^\times$ with coordinate $\xi$ as above.  Write $\xi=s+r\tau$, where $r,s\in \R$ and $\tau$ is fixed, and 
 let $\omega=d\xi$ and   $\nu = 2 \pi i dr$. The classes $[\omega],[\nu]$ form a basis for $H^1(\E^\times;\C)$.

\begin{lem} \label{lemomega} The form  $\Omega(\xi;\alpha) = \e(\alpha r) F(\xi;\alpha)d\xi$
is invariant under elliptic transformations $\xi \mapsto \xi+\tau$ and $\xi \mapsto \xi+1$, and satisfies
$d\, \Omega(\xi;\alpha) =   \nu \alpha \wedge \Omega(\xi;\alpha) $.
\end{lem}
\begin{proof} Straightforward calculation using proposition \ref{prop3prop}i).
\end{proof}
We can view  $\Omega(\xi;\alpha)$ as a generating series  of   one-forms  on $\E^\times$.  Let $\xi_1,\ldots, \xi_{n}$ denote  the usual holomorphic coordinates on $\E^\times \times\ldots \times \E^{\times}$  and set $\nu_i =2 i \pi dr_i$.

\begin{defn} Let $\xi_0=0$ and  define  one forms  $\omega_{i,j}^{(k)} \in \mathcal{A}^{1}(\En)$ for all $0\leq i\leq j\leq n$ and $k\geq 0$ of type $(1,0)$ by the generating series:
\begin{equation} 
\Omega(\xi_i-\xi_j;\alpha) = \sum_{k\geq 0} \omega_{i,j}^{(k)} \alpha^{k-1}\ .
\end{equation}
\end{defn} 
We clearly have $\omega^{(k)}_{i,i}=0$ and $\omega_{i,j}^{(k)} +(-1)^k\omega_{j,i}^{(k)} =0$ for all $i,j,k$.
The leading terms $\omega^{(0)}_{i,j}$ are equal to $d\xi_i -d\xi_j $  and therefore satisfy the relations:
\begin{equation} \label{omega0rel}
\omega^{(0)}_{i,j}+\omega^{(0)}_{j,k}= \omega^{(0)}_{i,k} \quad \hbox{for all } i,j,k \ .
\end{equation}
The higher terms  $\omega^{(k)}_{i,j}$  can be viewed as Massey products via the equation: 
\begin{equation} \label{omegaijMassey}
 d \, \omega^{(k+1)}_{i,j} = (\nu_{i}-\nu_j)\wedge \omega^{(k)}_{i,j}\  \hbox{ for } k\geq 0 \ ,
 \end{equation}
which follows from lemma \ref{lemomega}. The Fay identity implies that  
\begin{eqnarray}\label{OmegaFay}
\qquad\qquad  \Omega(\xi_i-\xi_{\ell};\alpha) \wedge\Omega(\xi_{j}-\xi_{\ell} ;\beta)  & + &  \Omega(\xi_{j}-\xi_{i};\beta)\wedge \Omega(\xi_i-\xi_{\ell};\alpha+\beta) \\
 &+  &  \Omega(\xi_{j}-\xi_{\ell};\alpha+\beta) \wedge \Omega(\xi_i-\xi_j;\alpha)   \quad = \quad 0  \nonumber
 \end{eqnarray}
which gives rise to infinitely many quadratic relations between the $\omega^{(k)}_{i,j}$. Finally, the definition of  $F$ shows that the residues of these forms are given by
 \begin{equation} \label{OmegaRes}
  \mathrm{Res}_{\xi_i=\xi_j}  \, \omega^{(k)}_{i,j} =  2i \pi \,\delta_{1k}\ ,
 \end{equation}
where $\delta$ denotes the Kronecker delta. 
 Now consider the projection  $\En\rightarrow \E^{(n-1)}$  given by $(\xi_1,\ldots, \xi_n) \rightarrow (\xi_1,\ldots, \xi_{n-1})$. Its fibers  $\E_{F_n}$ are  isomorphic to the  punctured elliptic curve $\E^{\times}\backslash \{ \xi_1,\ldots, \xi_{n-1} \}$ with  coordinate    $\xi_n$. Let $\overline{\omega}^{(k)}_{i,j}$ (resp. $\overline{\nu}_n$) denote the  relative forms obtained by restricting  $\omega^{(k)}_{i,j}$ (resp. $\nu_n$)
to the fiber. Clearly $\overline{\omega}^{(0)}_{n,i} = d\xi_n $ for all $i$.
\begin{lem} \label{lemindeponfib}
The  1-forms $\{\overline{\nu}_n$, $d\xi_n$,  $\overline{\omega}^{(k)}_{n,i}$ for $k\geq 1$, all $i\}\subset  \mathcal{A}^1(\E_{F_n})$, and the 2-forms
\begin{equation}\label{indepdeg2onfib} \{\overline{\nu}_n\wedge d\xi_n\ , \  \overline{\nu}_n \wedge \overline{\omega}^{(k)}_{n,i} \hbox{ for } k\geq 1, \hbox{ all } i \} \subset \mathcal{A}^2(\E_{F_n})
\end{equation}
  are linearly independent over $\C$.
\end{lem}
\begin{proof}
Since the $\overline{\omega}$'s are of type $(1,0)$  and $\overline{\nu}_n$ is not,  it follows from  $(\ref{OmegaRes})$ 
that  the  forms $d\xi_n, \overline{\omega}^{(1)}_{n,0}, \ldots , \overline{\omega}^{(1)}_{n,n-1},  \overline{\nu}_n$
are linearly independent. Consider  a non-trivial relation
$$\sum_{0\leq i <n, \ k \leq w} c_{i,k}\, \overline{\omega}^{(k)}_{n,i} =0\ , \hbox{ where } c_{i,k}\in \C \ ,   $$ 
and  $w$ is minimal.   Differentiating gives $\overline{\nu}_n \wedge\big( \sum_{i,k} c_{i,k} \overline{\omega}^{(k-1)}_{n,i}\big)=0$, by  $(\ref{omegaijMassey})$.
Since $\overline{\nu}_n$  has a non-zero component of type $(0,1)$,  the wedge product by $\overline{\nu}_n$ on $(1,0)$-forms is injective, giving a smaller relation
 $ \sum_{i,k\leq w-1} c_{i,k+1}\,  \overline{\omega}^{(k)}_{n,i} =0$, which is  a contradiction.  The same argument proves that  $(\ref{indepdeg2onfib})$ are linearly independent.
 \end{proof}

\section{A rational model for the de Rham complex on $\E^{(n)}$}
We construct  a differential graded algebra $X_n$   over $\Q$ which is defined  by generators and quadratic relations, along with a quasi-isomorphism
$X_n\otimes_{\Q}\C \hookrightarrow \Ao_n,$  where $\Ao_n=\Ao^\bullet(\E^{(n)})$ is the $C^{\infty}$-de Rham complex on the configuration space of $n+1$ points on  $\E$.
We show that $X_n$ carries a  mixed Hodge structure and give a presentation for $H^{\bullet} (\En)$ which is an elliptic analogue of Arnold's theorem in the genus $0$ case.

 \subsection{Differential graded algebras and fibrations}   \label{sectDGA}
 Let  $k$ be a field of characteristic zero. Recall that a (positively-graded) DGA over $k$ is a  graded-commutative algebra  $A=\bigoplus_{n\geq 0} A^n$ with a differential $d:A\rightarrow A$ of degree $+1$
 which satisfies the Leibniz rule.  It is said to be connected if
 $A^0\cong k$. We shall consider algebras $A$ which are either finite-dimensional in each degree, or else carry a second grading (called the weight grading) for which they are finite-dimensional in every bidegree.

Let $A_T$ be such a DGA with differential $d_T$,  and let $A_B\subset A_T$ be a sub-DGA. Define
\begin{equation}\label{LerayFilt}  A_F =  A_T / A_B^{\geq 1} A_T\ ,
\end{equation}
which  inherits a differential $d_F$ from $d_T$. We call the triple $A_B,A_T,A_F$  a \emph{fibration} if 
  $A_T$ is a free $A_B$-module.  The indices $T,B,F$  stand for the total space, base, and fiber.  Now suppose that we are given a splitting $i_F:A_F \rightarrow A_T$
of $A^0_B$-modules. When $A_B,A_T,A_F$ is a fibration, the map $i_F$ defines  an isomorphism of $A_B$-modules:
\begin{equation}
A_T \cong A_B\otimes_{A^0_B} A_F = \textstyle{ \bigoplus_{i\geq 0}} A_B^i \otimes_{A^0_B} A_F\ , 
\end{equation} 
 which does not necessarily respect the differential or  algebra structure.

 \subsection{The model $X_n$} \label{sectXn}
We  consider the differential graded algebra  $X_n$ generated by  symbols corresponding to the  forms considered in \S\ref{sectMassey}.  By abuse of notation, we  use the same symbol to denote the 
generators in $X_n$  and their images in $ \Ao_n=\Ao^\bullet(\E^{(n)})$. This will be justified when we  show that $X_n\rightarrow \Ao_n$ is injective (corollary  \ref{corphiinj}).
\begin{defn} Let $X_n$ be the $\Q$-differential graded algebra generated by elements
\begin{eqnarray}
\omega^{(k)}_{i,j}   &&\hbox{ for }  k\geq0  \hbox{ and } 0\leq i\leq j\leq n \nonumber \\
\nu_i &&\hbox{ for } 1\leq i\leq n \nonumber
\end{eqnarray} 
in degree 1, modulo the graded-commutative ideal generated by the  relations  $(\ref{omega0rel})$ and  the coefficients of $(\ref{OmegaFay})$. The differential is given by $d\nu_i=0$,  $d\omega^{(0)}_{i,j}=0$, and 
 $(\ref{omegaijMassey})$ in all other cases. It is a simple calculation to check  that  the differential ideal generated by the Fay identity  $(\ref{OmegaFay})$  is equal to the (usual) ideal it generates. 
 \end{defn}
 There is an obvious map $X_{n-1}\rightarrow X_n$. Let $X_{n-1}^+$ be the ideal in $X_n$ generated by  the images of  elements of $X_{n-1}$ of positive degree,
and let  $X_{F_n}=X_n/X^+_{n-1}$. 
 Denote the images  of  $\omega^{(k)}_{i,j}$ and $\nu_{i}$  under the natural map $X_n \rightarrow X_{F_n}$ by 
 $\overline{\omega}^{(k)}_{i,j}$ and $\overline{\nu}_{i}$, respectively.
  
 \begin{lem} \label{lemXF} $X_{F_n}$ is isomorphic to the $\Q$-differential graded algebra generated by 
 $\overline{\omega}^{(k)}_{n,i}$ and $\overline{\nu}_n$
in degree 1,  subject to the relations: $\overline{\omega}^{(0)}_{n,i}=\overline{ \omega}^{(0)}_{n,j}$ for all $i,j$; \  $\overline{\nu}_n\wedge \overline{\nu}_n=0$; and
\begin{equation} \label{omegawedge} \overline{\omega}^{(k)}_{n,i} \wedge \overline{ \omega}^{(\ell)}_{n,j} =0\qquad \forall \quad i,j,k,\ell\ .\end{equation}
 The differential is  given by  $d\,\overline{\omega}^{(0)}_{n,i}=d\,\overline{\nu}_n=0$ and 
$d\,\overline{\omega}^{(k+1)}_{n,i}= \overline{\nu}_n \wedge \overline{\omega}^{(k)}_{n,i} \quad \hbox{ for }  k\geq 1.$
 \end{lem}
 \begin{proof} All the relations are obvious except for  $(\ref{omegawedge})$. It follows from the Fay identity $(\ref{OmegaFay})$ that 
 $\Omega(\xi_{n}-\xi_{i},\alpha) \wedge \Omega(\xi_{n}-\xi_{j},\beta) \equiv 0 \mod X^+_{n-1}$. 
 \end{proof} 
\noindent 
In particular, $X_{F_n}$ is concentrated in degrees $0$, $1$, and $2$. 
Let   $i_{F_n}:X_{F_n}\rightarrow X_n$ denote the splitting of  the quotient map $X_n\rightarrow X_{F_n}$ defined by:
$$i_{F_n} (\overline{\nu}_n) = \nu_n \qquad \ , \qquad i_{F_n}(\overline{\omega}^{(k)}_{n,i}) = \omega^{(k)}_{n,i}- \omega^{(k)}_{0,i}\ .$$
 \subsection{Mixed Hodge structure on $X_n$}  \label{sectMHS} The complex of $C^{\infty}$ forms on $\En$ with logarithmic singularities carries a Hodge  and  weight filtration.
 The weight filtration on 1-foms is  concentrated in degrees 1 and 2. But it turns out that there is a refined weight filtration on $X_n$. To define it,  set $W_0 X^1_n=0$ and
 $$ W_\ell X^1_n =\langle \nu_i, \omega_{i,j}^{(k)}: k<\ell \rangle \qquad \hbox{ for all }  \quad \ell \geq 1\ ,$$
  and extend it  by multiplication to $X_n$. It is well-defined because the relations  implied by $(\ref{OmegaFay})$ are homogeneous for the weight.
  This  filtration is induced by the grading for which $\nu_i$ has weight $1$ and $\omega_{ij}^{(k)}$ has weight $k+1$.  
     The Hodge filtration is given by 
  $$F^0 X^1_n = X^1_n \quad \supset \quad F^1 X^1_n = \langle \omega_{i,j}^{(k)} \rangle \quad \supset  \quad F^2 X^1_n=0 $$
  and extends to $X_n$ in the same way. One easily verifies that this defines a mixed Hodge structure on $X_n$ such that $d:X_n\rightarrow X_n$ is homogeneous
  for the weight. Likewise $X_{F_n}$ inherits a mixed Hodge structure which is compatible with the map $i_{F_n}$.

\subsection{Quadratic Algebras} We give a sufficient criterion for an algebra defined by quadratic relations to be a fibration. We shall only apply this in the case of $X_n$.

\begin{defn} Let $V$ be a finite dimensional vector space over a field $k$. Let $R\subseteq \bigwedge^2 V$ be  a subspace (the space of relations).
The associated quadratic algebra is
$$Y^{\cdot} =  \textstyle{\bigwedge^{\cdot} }V / \langle R\rangle\ ,  $$
where $\langle R \rangle \subseteq  \bigwedge^{\cdot}V$ is the ideal generated by $R$.  We have $Y^0=k$, $Y^1=V$.
\end{defn}

Now suppose that $V_B\subseteq V$ is a subspace, and let $V_F=V/V_B$. Choose  a splitting
$$V=V_B \oplus V_F\ , $$
which induces a splitting $\bigwedge^2 V =  \bigwedge^2 V_B \oplus ( V_B \otimes_k V_F) \oplus \bigwedge^2 V_F$. 
Let $\pi_F: \bigwedge^2 V\rightarrow \bigwedge^2 V_F$ denote projection onto the last component.  Assume that the space of relations splits:
$$
R = R_B \oplus R_F \ , $$
where $R_B\subseteq 
 \bigwedge^2 V_B$, and $R_F\subseteq (V_B \otimes_k V_F) \oplus 
 \bigwedge^2 V_F$. Let $Y^{\cdot}_B= \bigwedge^{\cdot} V_B / \langle R_B \rangle$.
 
\begin{prop}  \label{propquadfib}Suppose that
\begin{equation} \label{RFisom}
\pi_F : R_F \To \textstyle{\bigwedge^2} V_F \hbox{  is an isomorphism.}
\end{equation}
In this case,  the relations $R_F$ define the graph of  a map 
$\alpha:  \textstyle{\bigwedge^2 }V_F \rightarrow V_B \otimes_k V_F$, where $\alpha= id-\pi^{-1}_F $.
Extend $\alpha$ to a map $V_F^{\otimes 2} \rightarrow  \textstyle{\bigwedge^2 }V_F \rightarrow V_B \otimes_k V_F$. Consider the  two  different ways of composing
$\alpha$ with itself, namely
$$   \alpha^{(1)}_3:   V_F^{\otimes 3}\overset{\alpha \otimes id_{V_F}}\To  V_B\otimes_k  V_F\otimes_k  V_F \overset{id_{V_B} \otimes \alpha }\To  V_B^{\otimes 2}\otimes_k  V_F\To \textstyle{\bigwedge^2} V_B \otimes_k V_F$$
and
$$   \alpha^{(2)}_3:   V_F^{\otimes 3} \overset{ id_{V_F} \otimes \alpha }\To  V_F\otimes_k  V_B\otimes_k  V_F \cong  V_B\otimes_k  V_F\otimes_k  V_F \overset{id_{V_B} \otimes \alpha }\To V_B^{\otimes 2}\otimes_k  V_F\To \textstyle{\bigwedge^2} V_B \otimes_k V_F$$
satisfy the associativity condition
\begin{equation} \label{alphassoc}
\alpha^{(1)}_3  - \alpha^{(2)}_3 \in R_B \otimes_k V_F\ .
\end{equation}
Then $Y_B\rightarrow Y$ is injective, and  a  fibration, with fibers  $Y_F$, where $Y^0_F =k$, and $Y^1_F\cong  V_F$, and $Y^k_F=0$ for $k\geq 2$. Thus there is an isomorphism of $Y_B$-modules: 
$$Y\cong Y_B \otimes_k Y_F\cong Y_B \oplus (Y_B \otimes_k V_F)\ .$$\end{prop}

\begin{proof} There is an obvious natural map
$$i:Y_B \oplus (Y_B \otimes_k V_F) \To Y\ .$$
We construct an inverse to $i$ by defining by induction a sequence of  linear maps
$$\alpha_n:  \textstyle{\bigwedge^n} V_F \To  Y^{n-1}_B \otimes_k V_F \qquad \hbox{ for } n\geq 2\ ,$$
such that $i \circ \alpha_n(\xi) \equiv \xi \mod \langle R\rangle$.  For this, let $\alpha_2$ be the map $\alpha= id- \pi^{-1}_F$ defined above, and let $\alpha_n$ be the 
map obtained by composing $\alpha$ with itself $n-1$ times.  By the associativity property $(\ref{alphassoc})$, $\alpha_n$ is well-defined. It is clear from the definition that
$i\circ\alpha_2\equiv id \mod  R$, and from this we deduce that  $i \circ \alpha_n\equiv id \mod \langle R\rangle $ for all $n$ by induction. Now write
$$\textstyle{\bigwedge^n} V=\bigoplus_{i=0}^n \, \textstyle{\bigwedge^i} V_B \otimes_k  \textstyle{\bigwedge^{n-i}} V_F\ . $$
If we set $\alpha_0:k \rightarrow k$ and $\alpha_1: V_F \rightarrow V_F$ to be the identity maps, we deduce  a map
$$\rho=\bigoplus_{i=0}^{n} \pi^i_B\otimes \alpha_{n-i} : \textstyle{\bigwedge^n} V\To Y_B \otimes_k (Y_B \otimes_k V_F)\ ,$$
where $\pi^i_B: \bigwedge^i V_B \rightarrow Y_B$ is the natural map. Since $\alpha_n (\langle R\rangle)=0$ for all $n$, and since $R=R_B \oplus R_F$,  the  map  $\rho$ passes to the quotient to define a map 
$$\overline{\rho} : Y \To Y_B \oplus (Y_B \otimes_k V_F)$$
which satisfies $\overline{\rho} \circ i = id$ by definition and $i \circ \overline{\rho}$ is an isomorphism since $i \circ \alpha_n \equiv id \mod \langle R\rangle$.  Thus $i$ is   an isomorphism.
\end{proof}

\begin{rem} \label{infgradremark} 
In the previous discussion, we can also replace $V$ with a graded vector space which is of finite dimension in every degree, and $R$ by a graded subspace.
\end{rem}
 \subsection{Structure of $X_n$} We show that $X_{n-1}, X_n, X_{F_n}$ is a fibration of DGA's. 
  
\begin{lem} \label{lemnutens} There is an isomorphism of graded-commutative algebras 
$$X_n\cong \bigwedge(\Q\nu_1\oplus \ldots \oplus \Q \nu_n) \otimes_\Q Z_n\ ,$$
where $Z_n$ is the subalgebra of $X_n$ spanned by the elements $\omega^{(k)}_{i,j}$.   Likewise,
$$X_{F_n}\cong \bigwedge (\Q \overline{\nu}_n)\otimes_{\Q} Z_{F_n}\ ,$$
where $Z_{F_n}$ is the subalgebra of $X_{F_n}$ spanned by $\overline{\omega}^{(k)}_{n,i}$. Note that these isomorphisms do not respect the differential structures on $X_n$ and $X_{F_n}$.
\end{lem}
\begin{proof}  All defining relations of $X_n$ have Hodge filtration $\geq 1$, so $X_n/F^1X_n$ is isomorphic to the  free graded-commutative algebra spanned by $\nu_1,\ldots, \nu_n$. An identical argument  gives the corresponding isomorphism  for $X_{F_n}$.
\end{proof}

\begin{lem} The map $X_{n-1}\rightarrow X_n$ is injective, and  $X_n$ is a free  $X_{n-1}$-module. 
\end{lem}

\begin{proof}
We must prove that  $X_{n-1}\hookrightarrow X_n$ and  $X_n\cong X_{n-1}\otimes_{\Q} X_{F_n}$ as $X_{n-1}$-modules. By  lemma $\ref{lemnutens}$ this is equivalent to showing that $Z_{n-1}\hookrightarrow Z_n$ and 
$Z_n \cong Z_{n-1}\otimes_{\Q} Z_{F_n}$ as $Z_{n-1}$-modules. Since $Z_{n}$ is quadratic, it is enough to 
verify the criteria of proposition $\ref{propquadfib}$.  The quadratic relations $R$ are defined by $(\ref{OmegaFay})$ and so $R_F$ is generated by 
\begin{equation} \label{fiberomegaFay}
  (i,n;\alpha) \wedge (j,n,\beta)   + (j,n;\alpha+\beta)\wedge (i,j;\alpha) +  (j,i;\beta) \wedge (i,n;\alpha+\beta)  =  0 
 \end{equation}
 where $i,j \leq n-1$ and $  (i,n;\alpha)$ denotes  $\Omega(\xi_i-\xi_{n};\alpha) $, etc. Since every term $\omega^{(k)}_{n,i} \wedge  \omega^{(\ell)}_{n,j}$ for $k, \ell\geq 1$ occurs exactly once in the  Taylor expansion of the first term of  $(\ref{fiberomegaFay})$, the condition $(\ref{RFisom})$ is verified. The cases where $k$ or $\ell=0$ are trivial to check.
To verify $(\ref{alphassoc})$, apply the  identity  $(\ref{fiberomegaFay})$ four times to get:
\begin{eqnarray}
{\big[\big[}  (i,n;\alpha) \wedge (j,n ;\beta)\big] \wedge  (k,n; \gamma)\big] &= &(j,i;\beta)\wedge (k,i;\gamma)\wedge (i,n;\alpha+\beta+\gamma)\nonumber \\
& +&   (k,j;\gamma)\wedge (i,j;\alpha)\wedge (j,n;\alpha+\beta+\gamma)\nonumber \\ 
& +&   (i,k;\alpha)\wedge (j,k,\beta)\wedge (k,n;\alpha+\beta+\gamma) \nonumber 
\end{eqnarray}
  Since the right-hand side is antisymmetric, the left hand side clearly does not depend on the bracketing, and the analogue of proposition $\ref{propquadfib}$ holds in the infinite graded case (remark \ref{infgradremark}), where the grading is given by the weight grading of \S\ref{sectMHS}.
\end{proof}

Let us write   $\Ao_{F_n} = \Ao_n / \Ao^+_{n-1}$ and let $\phi$ denote the natural map $X_n \rightarrow \Ao_n$.
 The choice of coordinate $\xi_n$ on the fiber of  $\En\rightarrow \E^{(n-1)}$  gives an isomorphism
 \begin{equation} \label{Atens}
 \Ao_{n-1} \otimes_{\Ao^0_{n-1}} \Ao_{F_n} \cong \Ao_n\ .
 \end{equation}
 
\begin{cor}  \label{corphiinj}The map  $\phi$ is injective.
\end{cor}
\begin{proof}   By  lemma \ref{lemXF},  $X_{F_n}$ is concentrated in degrees at most  two, so  
it follows   from lemma \ref{lemindeponfib} that   $X_{F_n} \rightarrow \Ao_{F_n}$ is injective. The injectivity of 
$X_1 \rightarrow \Ao_1$ is a special case.  The  lemma  follows by  induction on $n$ using the previous lemma and $(\ref{Atens})$.
\end{proof} 

\subsection{Proof that $X_n$ is a model}   We now show  that $\phi:X_n\otimes_\Q \C \rightarrow \Ao_n$ is a quasi-isomorphism.
 First we compute $H^1(X_{F_n})$  and the Gauss-Manin connection on it.
 
\begin{lem} We have $H^0(X_{F_n})=\Q$,  $H^k(X_{F_n})=0$ if $k\geq 2$, and 
$$\gr^W_1 H^1(X_{F_n}) \cong \Q[\overline{\nu}_n] \oplus \Q[\overline{\omega}_{n,0}^{(0)}]\ ,  \quad \gr^W_2 \, H^1(X_{F_n}) \cong \bigoplus_{1\leq i\leq n-1} \Q [\overline{\omega}_{n,i}^{(1)}-\overline{\omega}_{n,0}^{(1)}]\ ,$$ 
where  $H^1(X_{F_n}) \cong \gr^W_1 H^1(X_{F_n}) \oplus \gr^W_2 H^1(X_{F_n})$.
\end{lem}

\begin{proof}  For all $k\geq 3$,  $X^k_{F_n}=0$  and so  $H^k(X_{F_n})=0$.  By  $(\ref{omegawedge})$ and $\overline{\nu}_n\wedge \overline{\nu}_n=0$, any 
two-form in $X_{F_n}$ can be written 
$$\sum_{k,i} c^k_{n,i}\,   \overline{\nu}_n \wedge \overline{\omega}^{(k)}_{n,i} = d \,  \big(\sum_{k,i} c^k_{n,i}\,  \overline{\omega}^{(k+1)}_{n,i}\big)\qquad \hbox{where } c^k_{n,i}\in\Q\ , $$
so is  exact. Thus  $H^2(X_{F_n})=0$ and  clearly  $H^0(X_{F_n})\cong X^0_{F_n}= \Q$. 
Since $\overline{\nu}$ and $\overline{\omega}^{(0)}_{n,0}$ are closed, it suffices by lemma \ref{lemXF} to consider a  one-form
$$\eta=\sum_{k\geq 1,0\leq i< n} c^k_{n,i}\, \overline{\omega}^{(k)}_{n,i} \quad \hbox{where } c^k_{n,i}\in \Q \ ,\quad  \hbox{ such that }  d\eta=0\ .$$
This implies that $ d\eta=\overline{\nu}_n\wedge \big( \sum_{k\geq 1,0\leq i<n} c^k_{n,i} \overline{\omega}^{(k-1)}_{n,i} \big) =0$. By lemma \ref{lemnutens} we have
$$\sum_{k\geq 1, 0\leq i <n  } c^k_{n,i} \, \overline{\omega}^{(k-1)}_{n,i} =0\ .$$ 
Since the forms $\overline{\omega}^{(k)}_{n,i}$, $k\geq 1$ and $\overline{\omega}^{(0)}_{n,0}$ are linearly independent in $X_{F_n}$ by lemma \ref{lemindeponfib}, 
and since $\overline{\omega}_{n,n-1}^{(0)}=\ldots = \overline{\omega}_{n,0}^{(0)}$,  we conclude that 
the closed forms in $X^1_{F_n}$ are spanned by 
$$\overline{\nu}_n\ , \ \overline{\omega}^{(0)}_{n,0}\ ,  \quad \hbox{  and   } \,\, \{\eta = \sum_{0\leq i <n} c^1_{n,i} \overline{\omega}^{(1)}_{n,i} \quad \hbox{such that } \sum_{0\leq i <n} c^1_{n,i}=0\}\ .$$
This implies the result, along  with the definition of the mixed Hodge structure \S\ref{sectMHS}.
\end{proof}

Since $X_n$ is a fibration, we can easily compute the Gauss-Manin connection on $H^1(X_{F_n})$ (see \cite{BDR} for further details). It is \emph{a priori}  nilpotent
 since the weight filtration is 
defined on all of $X_n$, and satisfies $W_0 X_{n-1}=0$.   Explicitly, it is
\begin{eqnarray} \label{GaussManinaction}
  H^1(X_{F_n}) & \rightarrow  & X^1_{n-1}\otimes_\Q H^1(X_{F_n})   \\
\nabla [\overline{\omega}^{(1)}_{n,0} - \overline{\omega}^{(1)}_{n,i}] &= & \nu_i \otimes [\overline{\omega}^{(0)}_{n,0}] + \omega_{i,0}^{(0)} \otimes  [\overline{\nu}_n]  \nonumber   \\
\nabla [\overline{\omega}^{(0)}_{n,0}] & = & 0  \nonumber \\
\nabla [\overline{\nu}_n] & = & 0  \nonumber 
\end{eqnarray}
Using the fact that    $\omega^{(0)}_{n,i}= \omega^{(0)}_{n,0}-\omega^{(0)}_{i,0}$,
the first line follows from the calculation
\begin{eqnarray}
 i_{F_n} (\overline{\omega}^{(1)}_{n,i} -\overline{\omega}^{(1)}_{n,0}  ) & = & \omega^{(1)}_{n,i}- \omega^{(1)}_{0,i} -\omega^{(1)}_{n,0} \nonumber \\
d ( i_{F_n} (\overline{\omega}^{(1)}_{n,i} -\overline{\omega}^{(1)}_{n,0}  )) & = & (\nu_n-\nu_i) \wedge \omega^{(0)}_{n,i} - \nu_i \wedge \omega^{(0)}_{i,0} - \nu_n \wedge \omega^{(0)}_{n,0}\ ,  \nonumber  \\
& = & -\nu_i \wedge \omega^{(0)}_{n,0} - \nu_n \wedge \omega^{(0)}_{i,0} \ .\nonumber 
 \end{eqnarray}
The second and third lines  of $(\ref{GaussManinaction})$ follow from the fact that $\omega^{(0)}_{n,0}$ and $\nu_n$ are exact.

 \begin{lem}    $H^1(\phi): \gr^W_{\dt}  H^1(X_{F_n}) \otimes_{\Q}\C \rightarrow  \gr^W_{\dt} H^1(\Ao_{F_n})$   is an isomorphism.
 \end{lem}
 \begin{proof}
The differential graded algebra   $\Ao_{F_n}$ computes the de Rham cohomology of the fiber of the map $\E^{(n)}\rightarrow \E^{(n-1)}$, which is isomorphic to $\E$ minus $n$ points. 
Furthermore, it carries a Hodge and weight filtration which induce the corresponding filtrations on $H^1(\E \backslash \{n \hbox{ points}\})$.
 The Gysin  sequence gives:
\begin{equation}  
0\rightarrow H^1(\E;\C) \rightarrow H^1(\E\backslash \{n \hbox{ points}\};\C) \rightarrow \C(-1)^{n-1} \rightarrow 0\ , \nonumber 
\end{equation}
where the third map is given by the residue.  Therefore 
$\gr^W_1 H^1(\Ao_{F_n}) \cong H^1(\E)$ and  $\gr^W_2 H^1(\Ao_{F_n}) \cong \C(-1)^{n-1}$.
The lemma follows from the fact that $[\phi(\overline{\nu}_n)],  [\phi(\overline{\omega}_{n,0}^{(0)})]$ is a basis of $H^1(\E)$ and 
$ \phi(\overline{\omega}_{n,i}^{(1)})$  has residue $2 \pi i$ at $\xi_n=\xi_i $, by $(\ref{OmegaRes})$.
\end{proof}

\begin{thm} $\phi:X_n\otimes_{\Q}\C \hookrightarrow \Ao_{n}$ is a quasi-isomorphism.
\end{thm}
\begin{proof}  The case $n=1$ follows from   the previous lemma.
The  case $n>1$ follows by induction by a standard Leray spectral sequence argument. Details are given in   \cite{BDR}.
\end{proof} 

In conclusion, $X_n$ is a model for the de Rham complex on $\En$ and provides a universal $\Q$-structure on its cohomology.

\subsection{A simplified model} Although we shall not use it, one can  consider a finitely-generated DGA model $Y_n$ for the cohomology of a configuration of elliptic curves.

\begin{defn} Let $Y_n$ be the commutative graded $\Q$-algebra defined by generators  $\omega_i,\nu_i$ for $1\leq i\leq n$   and $\omega_{ij}$ for $1\leq i \leq j \leq n$ in degree one such  that
\begin{eqnarray}\label{Ydef}
\omega_{ij}-\omega_{ji} & = & 0  \\
\omega_i \wedge \nu_i  & = &  0  \nonumber \\
\omega_{ij} \wedge \omega_i +   \omega_{ji} \wedge \omega_j  & = & 0\nonumber \\
\omega_{ij} \wedge \nu_i + \omega_{ji} \wedge \nu_j & = &  0  \nonumber \\ 
  \omega_{i\ell}\wedge \omega_{j\ell} + \omega_{j\ell} \wedge\omega_{ij} + \omega_{ji} \wedge\omega_{il} &=& 0  \nonumber 
\end{eqnarray}   
and define a differential $d:Y_n\rightarrow Y_n$ by $d\omega_i=d\nu_i=0 $ and 
\begin{equation} \label{Ydiff}
d \omega_{ij} = \omega_i \wedge \nu_j + \omega_j \wedge \nu_i\ .\end{equation}
\end{defn}
There is a  surjective map $X_n \rightarrow Y_n$ which sends $\omega_{i,0}^{(0)}$ to $\omega_i$ and $\omega^{(1)}_{i,j}$ to $\omega_{ij}$, $\nu_i$ to $\nu_i$  and all  $\omega^{(k)}_{i,j}$ for $k\geq 2$,  to zero. We can define a mixed Hodge structure on $Y_n$ in the same way as for $X_n$, i.e., $\omega_i,\nu_i$ have weight one and $\omega_{ij}$ weight two.

\begin{thm} The map $X_n\rightarrow Y_n$ is a quasi-isomorphism, i.e.,  $H^\bullet(Y_n) \cong H^\bullet(\En)$.
\end{thm}
\begin{proof} (Sketch) Follow the steps of the proof that $X_n\rightarrow \Ao_n$ is a quasi-isomorphism:  first define $Y_{F_n}$ to be $Y_n /Y_{n-1}^+$ and
check that it is a fibration. Then apply the Leray spectral sequence argument, noting that  $(\ref{Ydiff})$ exactly corresponds to the Gauss-Manin connection on $H^1(X_{F_n})\cong H^1(Y_{F_n})$.
\end{proof}
The model $Y_n$  kills all higher Massey products in $X_n$.  Since  $Y_n$ is finitely generated it may be useful for explicit  implementation of the algebra $H^{\bullet}(\En)$. 
 We have quasi-isomorphisms 
$  Y_n \leftarrow \!\!\!\!\!\leftarrow  X_n \hookrightarrow \Ao_n$
but note that  there is no map from $Y_n$ to $\Ao_n$.

\section{Bar construction of the de Rham complex of $\En$} The model $X_n$ enables us to put a $\Q$-structure on the bar construction of $\Ao_n$.

\subsection{Reminders on the bar construction} See \cite{Ch2}, \S1 for further details.
Let $A$ be a DGA as in \S\ref{sectDGA}, and further  assume that $A$ has an augmentation map $\varepsilon:A \rightarrow k$.  Let $s:A\rightarrow A$ denote the map $s(a) = (-1)^{\deg a} a$. Let $A^+=\ker \varepsilon$ denote the  augmentation ideal of   $A$  and let  
$T(A^+) = k\oplus A^+ \oplus A^{+\otimes 2}\oplus \ldots $ denote the tensor algebra on $A^+$, where  all tensor products are  over $k$.  We write the   element $a_1\otimes\ldots \otimes a_n\in A^{\otimes n}$  using the bar notation  $[a_1|\ldots |a_n]$. Recall that $T(A^+)$ is a commutative Hopf algebra for the shuffle product  ($\Sigma(r,s)$ denotes the set of $r,s$ shuffles):
$$[a_1|\ldots |a_{r}]\sha [a_{r+1}|\ldots |a_{r+s}] = \sum_{\sigma \in \Sigma(r,s)} \varepsilon(\sigma) [a_{\sigma(1)}|\ldots |a_{\sigma(r+s)}]\ , $$
where the sign $\varepsilon(\sigma)$ also depends on the degrees of the $a_i$'s but is always equal to $1$ if all $a_i$ are of degree $1$.
The   coproduct $\Delta: T(A^+) \rightarrow T(A^+) \otimes_k T(A^+)$ is given by
$$\Delta ([a_1 |\ldots |a_{r}] ) = \sum_{i=1}^r  [a_1 |\ldots |a_{i}]\otimes [a_{i+1}|\ldots |a_r]\ . $$
The \emph{length filtration} is the increasing filtration associated to the  tensor grading 
$$F_n T(A^+) = \bigoplus_{0\leq i\leq n} A^{+\otimes i} \ .$$
The \emph{bar complex}  is the double complex with terms $(A^{+\otimes p})_q$ (elements of total degree $q$ and length $p$ in $T(A^+)$), and  with one  differential  $(-1)^p d_i:(A^{+\otimes p})_q \rightarrow (A^{+\otimes p})_{q+1}$:
$$ d_{i}( [a_1|\ldots |a_p]) =  \sum_{1\leq i \leq p} (-1)^i [sa_1|\ldots|sa_{i-1}|da_i|a_{i+1}|\ldots |a_p] $$
and a second  differential $d_e: (A^{+\otimes p})_q \rightarrow (A^{+\otimes p-1})_{q}$ where:
$$d_{e}([a_1|\ldots |a_p]) = \sum_{1\leq i<p} (-1)^{i+1} [sa_1|\ldots | s a_{i-1} |sa_i\wedge a_{i+1}|a_{i+2}|\ldots |a_p] $$
The  bar construction  $\B(A)$   is defined to be  the total complex $\bigoplus (A^{+\otimes p})_q$ with total  differential $D=d_i+d_e$.
Note that the total degree of elements in $\B(A)$ is given by
\begin{equation} \label{totdegree}  \deg ([a_1|\ldots|a_p])  = \deg(a_1)+\ldots +\deg(a_p) - p  \ .\end{equation}
Let $V(A) = H^0(\B(A))$ be the  zeroth cohomology. It is  a commutative Hopf algebra.   

\subsection{Connected case}  When $A$ is connected,  this construction  simplifies.  The augmentation ideal   $A^+ = \bigoplus_{n\geq 1} A^n$ is the set  of elements of positive degree. It is clear from $(\ref{totdegree})$  that  only elements of $A$ of degree 1 contribute to $V(A)=H^0(\B(A))$, so  
  let $T(A^1)$ denote the tensor algebra generated by elements of degree $1$.
   Note that  $(A^{+\otimes p})_q=0$ if $p>q$ and the total degree $(\ref{totdegree})$ is  non-negative.
  The total differential  $-D=d_i+d_e$  reduces to   $D:T(A^1) \rightarrow T(A)$:
$$D ([w_1|\ldots |w_r] ) = \sum_{i=1}^{r} [w_1|\ldots |w_{i-1} |dw_i|w_{i+1} |\ldots |w_r] +  \sum_{i=1}^{r-1} [w_1|\ldots |w_{i}\wedge w_{i+1}|\ldots |w_r]$$
where we changed the overall sign for convenience.  We can simply write
\begin{equation}\label{VofAdef}
V(A) = \ker  
\big(D: T(A^1) \rightarrow T(A)\big)\ .
\end{equation}
We say that elements  $\xi\in T(A^1)$ satisfying $D\xi=0$ are \emph{integrable}.

\begin{defn} Define the bar construction  of $\En$  to be   $V(X_n)=H^0(\B(X_n))$, where $X_n$ is our model (\S\ref{sectXn}). It is a commutative  Hopf algebra over $\Q$, filtered by the length.
\end{defn}
By $(\ref{VofAdef})$, $V(X_n)$ is the subalgebra of $T(X^1_n)$ given by the  integrable words in $X^1_n$. 
Since $X_n$ carries a mixed Hodge structure, so too does $V(X_n)$.

\subsection{Description of $V(X_{F_n})$} We first give an explicit description of the bar construction of an elliptic curve with punctures.  By the general theory, the length-graded
$$\gr^{\ell} V(X_{F_n})  \cong \bigoplus_{\ell\geq 0} H^1(\E_{F_n})^{\otimes \ell}\ ,$$
since there is no integrability condition for one-dimensional varieties. The right-hand side is just the set of words in the generators of $H^1(\E_{F_n})$.  These generators are represented by the closed one-forms  $\overline{\omega}^{(0)}_n$, $\overline{\nu}_n$ and $ \eta_i: = \overline{\omega}^{(1)}_{n,i}-\overline{\omega}^{(1)}_{n,0},$ for $1\leq i< n$.

\begin{prop} \label{propwordlift} 
Any word in the letters $[\overline{\omega}^{(0)}_n],  [\overline{\nu}_n], [\eta_i]$ of length $\ell$ can be canonically lifted to an integrable element in $V(X_{F_n})$ using  the forms $\overline{\omega}^{(k)}_{n,i}$ for $k< \ell$.
\end{prop}

\begin{proof}
Use Chen's formal power series connections \cite{Ch1}.  Let  $S=\Q\langle\langle \x_0,\x_1,\y_1,\ldots, \y_{n-1}\rangle\rangle$ denote the ring of non-commutative formal power series
in symbols $\x_0,\x_1,\y_1,\ldots, \y_{n-1}$  and let $\alpha= -\mathrm{ad}(\x_0)$.  Consider the formal 1-form \cite{CEE,LR}:
$$J= \overline{ \nu}\, \x_0 + \alpha\, \overline{ \Omega}(\xi_n;\alpha) \x_1 +   \sum_{i=1}^{n-1}  \big(\overline{\Omega}(\xi_n-\xi_i;\alpha) - \overline{\Omega}(\xi_n;\alpha)   \big)  \y_i  \in X_{F_n}\otimes_{\Q} S $$
It is well-defined since  all polar terms in $\alpha$ cancel, and is of the form
$$J = \overline{\nu}_n\, \x_0 + \overline{\omega}^{(0)}_n \x_1 + \eta_1 \y_1+\ldots +  \eta_{n-1}  \y_{n-1}+ \hbox{ higher order terms  in } \x\hbox{'s,}\y\hbox{'s}  $$
One easily checks from lemma \ref{lemXF} that  $dJ=-J\wedge J$.  It follows that the formal element $\Xi=[J]+[J|J]+[J|J|J]+\ldots$ lies in $V(X_{F_n})\otimes_{\Q} S$.
The topological dual of $S$ may be identified with the tensor coalgebra $T^c(U)$,
where $U$ is the $\Q$-vector space spanned by 
the    alphabet $\{\x^0, \x^1, \y^1,\ldots, \y^{n-1}\}$ dual to $\{\x_0, \x_1, \y_1,\ldots, \y_{n-1}\}$. Thus we can view
$$\Xi \in \mathrm{Hom} ( T(U), V(X_{F_n}))\ .$$
On the other hand, consider the element
$$J^{(1)}=[\overline{\nu}_n]\, \x_0 + [\overline{\omega}^{(0)}_n] \x_1 + [\eta_1 ] \y_1+\ldots + [ \eta_{n-1}]  \y_{n-1}\in H^1(X_{F_n}) \otimes U^{\vee}\ . $$
It defines an isomorphism $j: U \cong H^1(X_{F_n})$, and hence $j: T^c(U) \cong  T^c(H^1(X_{F_n}))$. Composing $j^{-1}$ with $\Xi$ gives a map
$$\gr^{\ell}(V(X_{F_n})) \cong T^c(H^1(X_{F_n}))  \To V(X_{F_n})$$ 
which splits the quotient map.
 Note  that it follows from the proof that this splitting  has integral coefficients.
\end{proof}

It follows that $V(X_{F_n})$ is canonically isomorphic to the tensor coalgebra spanned by $[\overline{\omega}]$, $[\overline{\nu}]$, $[\eta_i]$,
equipped with the shuffle product.   The Hodge and  weight filtrations  are induced from the corresponding filtrations on $X_n$.
More precisely we have
\begin{equation} \label{HdgeFil} [x_1|\ldots |x_n] \in F^r V(X_{F_n}) \quad  \hbox{ if } \quad  \big|\{i: x_i =  \overline{\nu}\} \big|\leq n- r\ .
\end{equation}
The weight comes from  a grading $w:\gr^{\ell} V(X_{F_n})\rightarrow \N$ which is  defined by 
\begin{equation} \label{WeightFil}
w([x_1|\ldots |x_n])=n+  \big|\{i: x_i \in\{\eta_j\}\}\big| \ ,
\end{equation} 
and is obtained by giving $\overline{\omega}^{(0)}$ and  $\overline{\nu}$  weight $1$, and the $\eta_i$'s weight $2$.

\subsubsection{Example: the bar construction on $\E^\times$} In the case of a single puncture:
$$\gr^\ell_{\dt} V(X_1) \cong T(\Q[\omega^{(0)}]\oplus \Q [\nu])\ .$$
The formal power series connection $J$ reduces in this case to
\begin{eqnarray}
J &= &  \nu \,\x_0 + \Omega(\xi,-\mathrm{ad}(\x_0)) \,\x_1 \quad \in \quad X^1 \langle\langle \x_0,\x_1\rangle\rangle \nonumber \\
& = &   \nu\, \x_0 + \x_1 \omega^{(0)} - [\x_0,\x_1]\omega^{(1)} + [\x_0,[\x_0,\x_1]] \omega^{(2)} + \ldots \nonumber
\end{eqnarray}
and gives an explicit way to lift  any word in the letters $[\omega^{(0)}], [\nu]$ to  $V(X_1)$.

\begin{cor} \label{corw=l} The weight and the length filtrations on $V(X_1)$ coincide.
\end{cor}

\begin{ex}
The elements of  $V(X_1)$ of length   at most one  are  $1, [\omega^{(0)}], [\nu]$. In length $\leq 2$ we   also have:
$[\omega^{(0)}|\omega^{(0)}]$, $[\omega^{(0)}|\nu]+[\omega^{(1)}]$, $[\nu|\omega^{(0)}]-[\omega^{(1)}]$, and $[\nu|\nu].$ 
\end{ex}

\subsection{Structure of $V(X_n)$} One of the  main results of \cite{BDR} implies:

\begin{thm} \label{thmBartensorstructure} There is an isomorphism of algebras
\begin{equation}\label{BXdecomp} V(X_n) \cong \bigotimes_{i=1}^n V(X_{F_i})\ .
\end{equation}
The length and weight filtrations  on $V(X_n)$ coincide.
\end{thm}
\begin{proof} The bar Gauss-Manin connection 
$\nabla_{\B}: V(X_{F_n}) \rightarrow X^1_{n-1}\otimes_{\Q} V(X_{F_n}),$
which is defined in \cite{BDR}, is nilpotent with respect to the weight grading. It implies  
the existence of a map $V(X_{F_n}) \rightarrow V(X_n)$ and an
 isomorphism
$V(X_n) \cong V(X_{n-1})\otimes_{\Q} V(X_{F_n})$
  of algebras, from which the statement follows by induction (see \cite{BDR} for the proofs).
\end{proof} 
Note that $(\ref{BXdecomp})$ does not respect the Hopf algebra, or differential structures on $X_n$.
It is however  a complete algebraic description of all iterated integrals on $\En$, and in particular, enables one
to write down a basis for them.

\begin{rem} 
Theorem $\ref{thmBartensorstructure}$ is proved in \cite{BDR} by first showing that the bar-de Rham cohomology 
 of $X_n$ is trivial. This   is equivalent to  the exactness of the  sequence:
$$0\To \Q \To X^0_n \otimes_{\Q} V(X_n) \To X^1_n\otimes_{\Q} V(X_n) \To \ldots \To X^n_n \otimes_{\Q} V(X_n)\To 0 \ . $$
One way to think of the refined weight grading on $X_n$ is as follows.  It defines a weight grading  on $T^c(X_n)$ with the property that the inclusion $V(X_{n})\hookrightarrow T^c(X_n)$ is compatible with the 
weight filtration on $V(X_n)$.
\end{rem}
$$ * \quad  \qquad * \qquad  \quad  *$$

\newpage
\section{Averaging  unipotent functions} \label{sectAverage}

\subsection{Introduction}
The main  idea for constructing multivalued functions on an elliptic curve is to use the Jacobi uniformization 
$$\E= \C^{\times}/q^{\Z}$$
and average a function on $\C^{\times}$ with respect to multiplication by $q$. Consider the example of the multivalued function $\Li_1(z) = -\log(1-z)$.
Let $q\in \C^{\times}$ such that $|q|<1$ and   $z\in \C^{\times}$ such that $1\notin q^{\R} z$.  The spiral $q^{\R}z$ can be lifted to a universal covering space of $\C\backslash \{0,1\}$, and the function $\Li_1(z)$ has a 
well-defined analytic continuation along it. Near the origin, the function $\Li_1(z)$ vanishes,  but at the point $z=\infty$ it has  a logarithmic singularity, so the naive average diverges. One way to ensure convergence is to  consider the generating series
$$E(z;u)=\sum_{m\in \Z} u^m \Li_1(q^mz)$$
where $u^{-1}$ is chosen small enough to dampen the logarithmic singularity at infinity, but not so small as to wreck the convergence at the origin. For $m\ll 0$, the asymptotic is
$u^m \log(q^m z)$,  which is bounded  if $u>1$. For $m\gg 0$ the terms are asymptotically $u^m q^m z$, which  is bounded  if $u<|q|^{-1}$. Thus for 
$1<u<|q|^{-1}$ the series $E(z;u)$ converges absolutely, and is almost periodic with respect to  multiplication by $q$. 

From this one can easily show that $E(z;u)$ has a simple pole at $u=1$. Now one must view $E(z;u)$ as a function of $\xi$ and $\alpha$, where
$u=\e(\alpha)$ and $z=\e(\xi)$. Thus the pole at $u=1$ contributes a pole at $\alpha=0$ which can be removed to obtain 
$$E^{\reg}(\xi;\alpha) = E(\xi;\alpha) - {1\over \alpha} \ . $$ This function now admits a  Taylor expansion at the point $\alpha=0$. 
The procedure for constructing multivalued functions on the elliptic curve $\E^{\times}$ is to take the 
 coefficients of $\alpha^i$, $i\geq 0$ in this Taylor expansion.
  \vspace{0.05in}

 The situation is more complicated in  the case of several complex variables. Suppose that we have a function  $f(t_1,t_2)$
 on 
 $\Mod_{0,5}(\C) =\{(t_1,t_2): t_1,t_2\neq 0,1, t_1 \neq t_2\}$, with singularities  along  the removed hyperplanes.  We wish to average the function
\begin{equation}\label{6introsum} \sum_{m_1,m_2\in \Z} u_1^{m_1} u_2^{m_2}f(q^{m_1} t_1, q^{m_2} t_2)
\end{equation}
 The first problem that we encounter is that the function $f(t_1,t_2)$ is simply not well-defined as $t_1,t_2\rightarrow \infty$ since its singularities $t_i=\infty$, $t_1=t_2$ do not cross normally at that point, and so the limit depends on the direction in which it is approached.  The standard solution is to blow-up the points $(0,0)$, $(\infty,\infty)$, as below: 
 \begin{figure}[h!]
 \begin{center}
    \leavevmode
    \epsfxsize=4.0cm \epsfbox{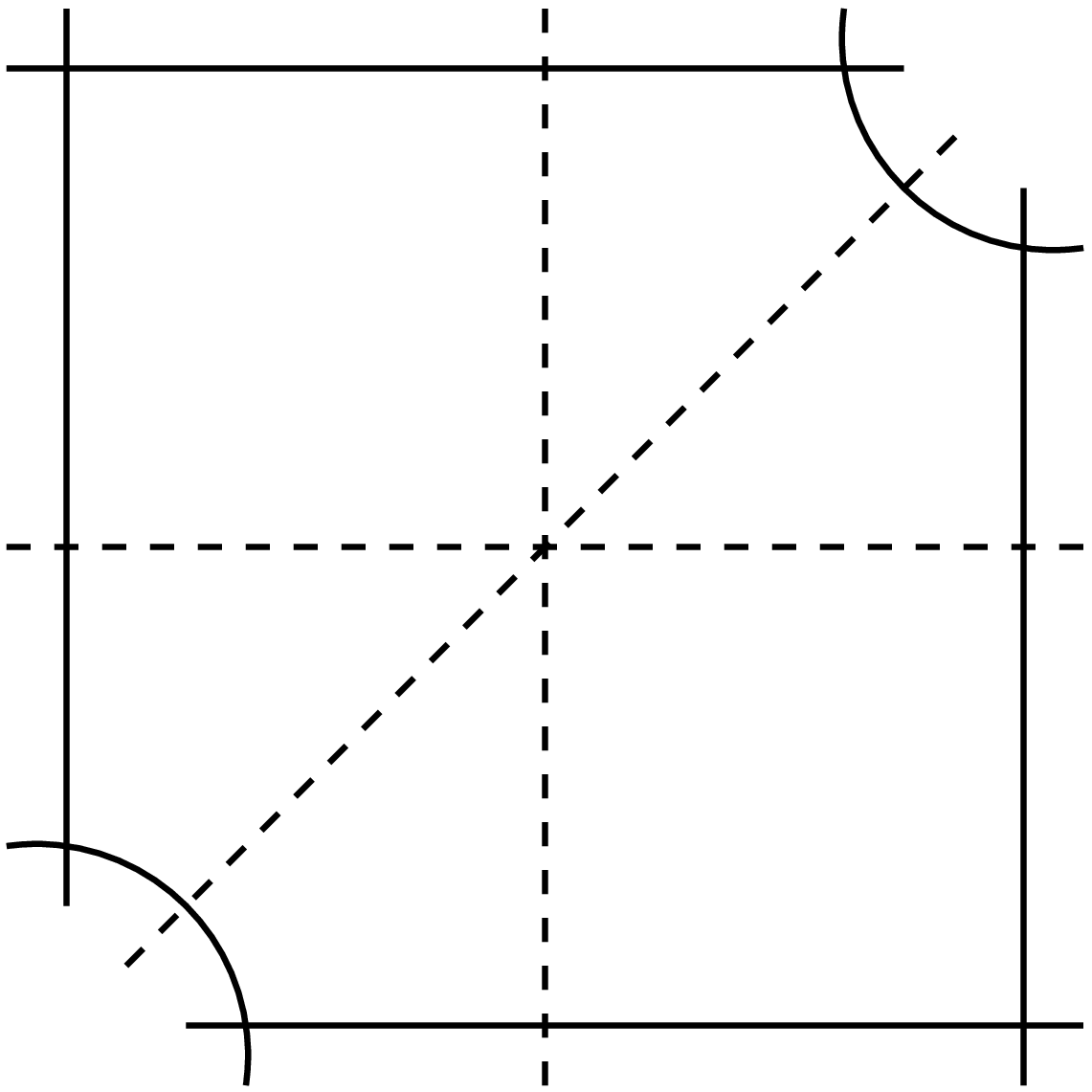}
  \put(-146,55){$t_2=1$}  \put(-146,105){$t_2=\infty$}
  \put(-70,-10){$t_1=1$}   \put(-20,-10){$t_1=\infty$}
     \end{center}
 \end{figure} 
 
 Note that  we do not need to blow up the point $t_1=t_2=1$, which is also a non-normal crossing point, because there are only finitely many lattice points $\{(q^\Z t_1, q^\Z t_2)\}$ in its neighbourhood.
 The domain of summation naturally decomposes into the six sectors pictured above, each of which is homeomorphic to a square $[0,1]\times [0,1]$.  
 By analysing the behaviour of $f(t_1,t_2)$ in the neighbourhood of each sector, one finds necessary and sufficient conditions on $u_1,u_2$ to ensure the absolute convergence of
 $(\ref{6introsum})$. It turns out that the poles in the $u_1,u_2$ plane are in one-to-one correspondence  with  the boundary divisors in the figure, and depend on the local asymptotic behaviour of $f$.
 One can then remove poles in the $\alpha_i$ plane (where $u_i= \e(\alpha_i)$) and compute Taylor expansions to extract multivalued functions on $\E^{(2)}$ with unipotent monodromy.
 \vspace{0.05 in}

The plan of the second, analytic, part of this paper  is as follows.
 \begin{enumerate}
 \item First we construct an explicit partial compactification of $\Mod_{0,n}$ which is adapted to this averaging procedure. 
 \item By studying the analytic properties of the multiple polylogarithms $(\ref{Idef})$ we  find necessary and sufficient conditions on the dual variables $u_i$ to ensure absolute convergence of the averaging function:
 $$ \sum_{m_1,\ldots, m_r\in \Z}  u^{m_1}_1 \ldots u^{m_r}_r I_{n_1,\ldots, n_r} (q^{m_1}t_1,\ldots, q^{m_r}t_r) 
$$
 \item  From its differential equation, we compute the pole structure in the $u_i$ coordinates. The multiple elliptic polylogarithms
 can be defined as  the coefficients in its regularized Taylor expansion at $\alpha_1=\ldots =\alpha_r=0$, where $u_i =\e(\alpha_i)$. Note that, since the singularities in the space of $\alpha_i$ parameters are not normal crossing, the regularization must be performed with some care. We do not address the  question  of explicit regularization  in the present paper.
  \end{enumerate}
 
The entire procedure from (1) to (3) will work  more generally for  any functions satisfying  some growth conditions on certain  toric varieties. 
Section \S\ref{sectAverage} covers the general steps $(1)$ and $(2)$.  The definition of the multiple elliptic polylogarithms is completed in \S\ref{sect8}, with examples given in \S\ref{sectexamples123}. In \S\ref{sect7}, which is independent from the rest of the paper,  we show how to  compute the asymptotics of the Debye polylogarithms explicitly at infinity 
using a certain coproduct. Finally, in \S\ref{sect10} we prove that all iterated integrals on a punctured elliptic curve
can be obtained by this averaging procedure.

\subsection{Preliminaries} \label{sectprepmap} Let $q=\e(\tau)$, where $\tau$ is in the upper-half plane, and let
$t_i=\e(\xi_i)$, for $i=1,\ldots, r$, where $\xi_1,\ldots, \xi_r$ are in the domain  $U$ defined by $(\ref{Udef})$,  in \S \ref{sectUniformization}. 
  
 Consider the following preparation map to a universal covering of $\Mod_{0,r+3}(\C)$: 
 \begin{eqnarray}
\sigma: \R^n  &\To& \widetilde{\Mod}_{0,r+3}(\C) \nonumber \\
(s_1,\ldots, s_r)  & \mapsto &  (q^{s_1}t_1,\ldots, q^{s_r}t_r)\ , \nonumber 
\end{eqnarray}
where we  view $\Mod_{0,r+3}(\C)\subset \C^r$ in   simplicial coordinates.
Suppose that  we have a multivalued function $f(t_1,\ldots, t_r)$  on $\Mod_{0,r+3}(\C)$,  with a fixed  branch 
in the neighbourhood of some $(t_1,\ldots, t_r)$ such that $t_it_j^{-1}\notin q^{\R}$ for $i\neq j$ and $t_i\notin q^{\R}$. 
Then $f(t_1,\ldots, t_r)$ admits a unique analytic continuation to the image of $\sigma(\R^n)$,
and in particular, the values $f(q^{n_1}t_1,\ldots, q^{n_r}t_r)$ are well-defined for all $(n_1,\ldots, n_r)\in \Z^r$.

We  apply this to the multiple polylogarithm functions
\begin{equation}  \label{Imassum}
I_{m_1,\ldots, m_n}(t_1,\ldots, t_n) = \sum_{a_1,\ldots, a_n\geq 1} {t_1^{a_1}\ldots t^{a_n}_n \over a_1^{m_1}(a_1+a_2)^{m_2}\ldots (a_1+\ldots +a_n)^{m_n} } \ ,
\end{equation} 
which give rise to multivalued unipotent functions on  $\Mod_{0,r+3}(\C)\subset \C^r$, and vanish along the divisors $t_i=0$. The power series expansion above defines
a  canonical branch in the neighbourhood of the origin. 

\subsection{Compactification of the hypercube} \label{sectCompactification}
Let  $S=\{1,\ldots, n\}$ and let  us write $\Pro^1_S$ for  $\mathrm{Hom}(S,\Pro^1) \cong (\Pro^1)^n$.
Let $\square_n\subset \Pro^1_S$ denote the real hypercube $[0,\infty]^n$. For any disjoint pair of subsets $I,J\subset S$, let
$$F_{I}^{J} = \bigcap_{i\in I} \{t_i=0\} \cap \bigcap_{j \in J} \{t_j=\infty\} \subseteq \Pro^1_S $$
denote the corresponding  coordinate linear subspace. The sets $F_{I}^{J}\cap \square_n$ give the standard stratification of the hypercube by  its faces.

 Working first in simplicial coordinates, consider the  divisor 
 $$X= \bigcup_{1\leq i< j\leq n} \{t_i- t_j = 0\}\cup\bigcup_{1\leq i\leq n} \{t_i=1\}\ ,$$
 which meets the set of faces $F^J_{I}$ non-normally.  Let us write $F_I = F^{\emptyset}_{I}$, $F^J=F_{\emptyset}^{J}$ and call such divisors
 of type $0$ or type $\infty$, respectively.  Consider the sets of faces:
\begin{equation} \nonumber
 \Fo^0=\{ F_{I}:  |I|>1\}\ ,  \quad  \Fo^{\infty}=\{ F^J : |J|>1\}\ ,   
\end{equation}  
of codimension $\geq 2$. 
Following the standard practice for blowing up linear subspaces, we  blow up the set of faces in $\Fo^0$ of smallest dimension, followed by the strict transforms of faces
$F_I$ where $|I|=n-1$, and so on, in increasing order of dimension, until the strict transforms of all elements in $\Fo^0$ have been blown up. 
Now repeat the same procedure with  $\Fo^{\infty}$, 
and  denote the corresponding space by $P_S$, with 
\begin{equation} \label{PSdef}
\pi: P_S \To \Pro^1_S \end{equation}
It does not depend on the chosen order of blowing-up.

Let us denote the strict transform of any face $F_I$ (resp. $F^J$) by $D_I$ (resp. $D^J$), for all $|I|, |J|\geq 2$. 
 Let $D_{i}$  (resp. $D^{j}$) denote the strict transform of the divisor $F_{i}^{\emptyset}$ (resp. $F^{j}_{\emptyset}$) which corresponds to a facet of the original hypercube, and   
 let us denote by $D=\bigcup_{|K|\geq 1} D_{K}\cup D^{K} $, the union of all the above.
 The strict transform of $\square_n$ is a certain polytope  $\mathcal{C}_n$, whose facets are in bijection with the irreducible components of $D$, which 
number $2\times (2^n-1)$.   Let $X'\subset P_S$ denote the strict transform of $X$.
\begin{prop} The divisor $D\cup X'\subset P_S$ is   locally   normal crossing near $D$. 
\end{prop}
The proof will be given by computing explicit normal coordinates in every local  neighbourhood of $D$, using a decomposition into  sectors.

\subsection{Sector decomposition}
Let us view $\Mod_{0,n}(\R)$ in simplicial coordinates as the complement of 
divisors of the form $t_i=t_j$  and $t_i=1$ in $(\R^\times)^n$. Then $\square_n\cap \Mod_{0,n+3}(\R)$ admits
a decomposition into 
$(n+1)!$ connected components:
$$\Delta_\pi=\{0< t_{\pi(1)}< \ldots < t_{\pi(n+1)} < \infty \}\ ,$$
where $\pi$ is a permutation of $(1, \ldots, n+1)$, the $t_i$ are simplicial coordinates on each component of $(\Pro^1)^n$, and 
  where  $t_{n+1}=1$.  The permutation $\pi$ should be viewed as a dihedral  ordering  of the $n+3$ marked points $0,1,\infty, t_1,\ldots, t_n$ on $\Pro^1(\R)$.
  To every such  $\pi$ we associate  local  `sector' coordinates  on $\Mod_{0,n+3}(\C)$ as follows:
\begin{equation}\label{sectorcoords}
s^{\pi}_1 ={ t_{\pi(1)} \over t_{\pi(2)}}\ , \ldots \ , s^{\pi}_{n} ={ t_{\pi(n)} \over t_{\pi(n+1)}}\ .\end{equation}
The coordinates $s_i^{\pi}$ give a homeomorphism from $\Delta_\pi$ to the unit cube $(0,1)^n$. When $\pi$ is the trivial permutation, the coordinates $s^{\pi}_i$ are the same coordinates $x_i$  used to define the multiple polylogarithms in \S\ref{sectMultiplePolyreminder}.  For each $\pi$, we define the open affine scheme
$$U_{\pi} = \Spec \Z[s^{\pi}_1,\ldots, s^{\pi}_n, \{(\prod_{i\leq k\leq j} s^{\pi}_k -1 )^{-1}\}_{1\leq i\leq j\leq n} ]\ .$$
Note that  the $U_{\pi}$  are all canonically isomorphic,  and $(0,1)^n\subset U_{\pi}(\R)$. 
\begin{lem} For every $\pi$,  $U_\pi$ defines an affine chart on $\overline{\Mod}_{0,n+3}$:
$$\Mod_{0,n+3} \subset U_{\pi} \subset \overline{\Mod}_{0,n+3}$$
\end{lem}
\begin{proof}
Consider the set of forgetful maps (or `cross-ratios') $f_T:\overline{\Mod}_{0,n+3}\rightarrow \overline{\Mod}_{0,4}\cong \Pro^1$, 
where $T$ is a subset of  any 4 of the $n+3$ marked points. Then $\Mod_{0,n+3}\subset \overline{\Mod}_{0,n+3}$ is the open subscheme where all $f_T$'s are non-zero.
It suffices to check that $U_{\pi}$ is isomorphic to the open subscheme of $\overline{\Mod}_{0,n+3}$ where some of the $f_T$'s are non-zero. For this, one can write
each $s^{\pi}_k$ and  $\prod_{i\leq k\leq j} s^{\pi}_k -1 $ as cross-ratios, and conversely every cross-ratio as a function of the $s^{\pi}_i$. We omit the details.
\end{proof}

\begin{defn} Let $U_n = \bigcup_{\pi} U_{\pi} \subset \overline{\Mod}_{0,n+3}$ be the scheme obtained by  gluing all charts $U_{\pi}$ together. Viewing  $\Sym_{n+1}$ as the stabilizer of $0,\infty$ in $\Aut(\overline{\Mod}_{0,n+3}) \cong \Sym_{n+3}$, we have
$$U_n= \bigcup_{\pi\in \Sym_{n+1}} \pi( U_{id})\ ,$$ where  $U_{id}$  corresponds to the trivial permutation.
\end{defn}

The smooth scheme $U_n\subset \overline{\Mod}_{0,n+3}$ is equipped with a set of  normal crossing divisors defined in each chart by the vanishing of  the $s^{\pi}_k$.

\begin{example} Let $n=2$.  The coordinate square $(0,\infty)\times (0,\infty)\subset \Mod_{0,5}(\R)$ is covered by six sectors $\Delta_{\pi}$  as shown below   after blowing up $F_{12}=(0,0)$, $F^{12}=(\infty,\infty)$.
\begin{figure}[h!]
 \begin{center}
    \leavevmode
    \epsfxsize=5.0cm \epsfbox{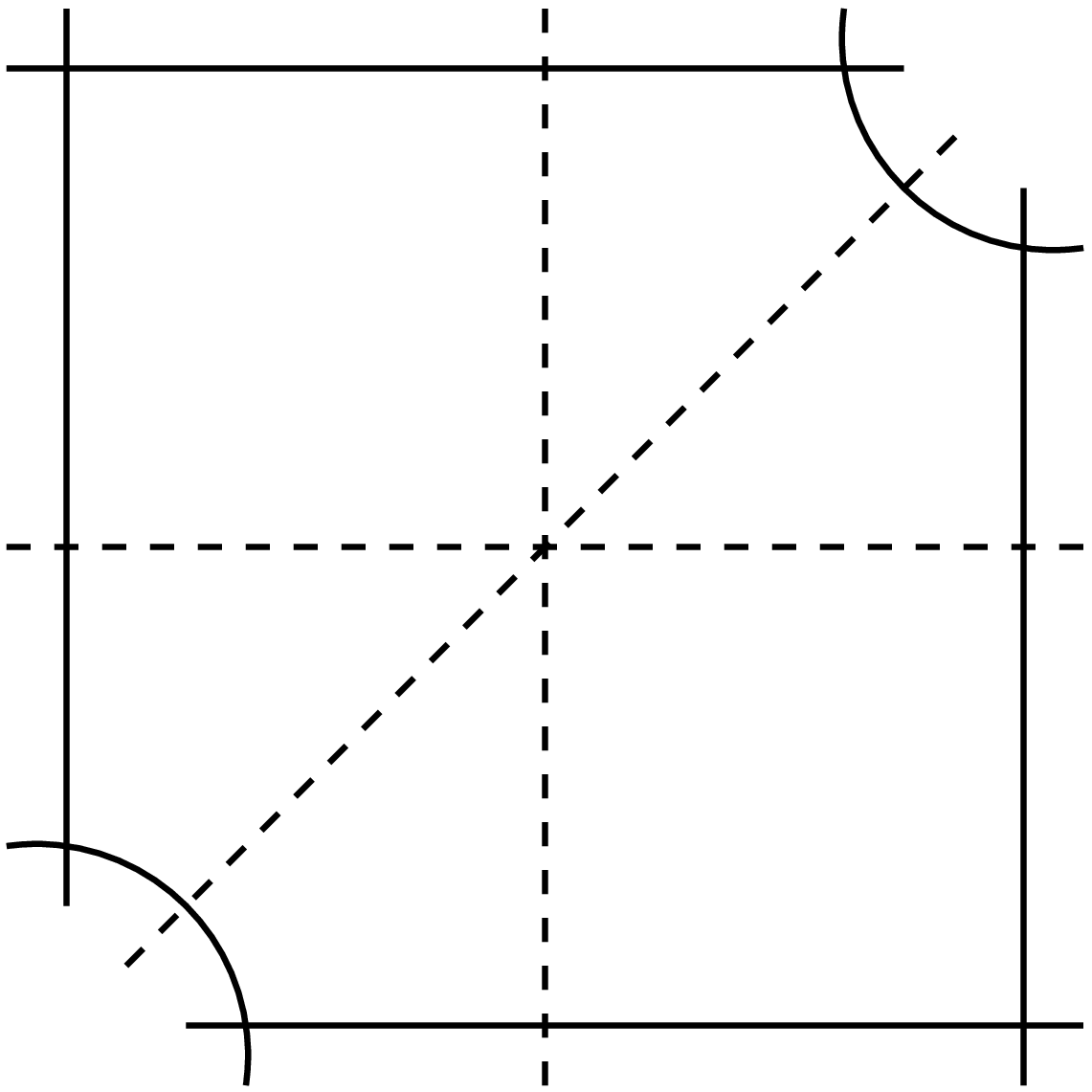}
 \put(-160,130){$D^2$}
 \put(-170,68){$t_2=1$}
 \put(-15,-10){$D^1$}
\put(5,5){$D_2$}
 \put(-138,145){$D_1$}
 \put(-85,-10){$t_1=1$}
 \put(-162,30){$D_{12}$}
\put(5,107){$D^{12}$}
\put(-60,110){$\Delta_{\pi_0}$}
  \end{center}
 \end{figure}

\noindent
Consider the sector denoted $\Delta_{\pi_0}$, where $\pi_0=(1,t_1,t_2)$, which  corresponds to the dihedral ordering $0<1<t_1<t_2<\infty$ on the marked points of $\Mod_{0,5}(\R)$.
Its sector coordinates are $s^{\pi_0}_1= t_1^{-1}, s^{\pi_0}_2=t_1/t_2$. The divisor $s^{\pi_0}_2=0$  corresponds to the partition $\{0,1,t_1\}|\{t_2,\infty\}$ on $\overline{\Mod}_{0,5}$, and 
the divisor $s^{\pi_0}_1=0$ corresponds to $\{0,1\}|\{t_1,t_2,\infty\}$. 
\end{example}

In the general case, we have:
\begin{lem} \label{lemSpi2IJ} 
 The  divisor  $s^\pi_k=0$ corresponds on  $U_{\pi}\subset \overline{\Mod}_{0,n+3}$   to the partition  (\S\ref{sectStandardCoords}) 
$$\{0,t_{\pi(1)},\ldots, t_{\pi(k)}\}\big| \{t_{\pi(k+1)},\ldots, t_{\pi(n+1)},\infty \}\ .$$
Therefore the scheme $U_n$ is the complement in $\overline{\Mod}_{0,n+3}$ of the  set of divisors $A$ corresponding to partitions in which the marked points $0$ and $\infty$
lie in the same component. 

The simplicial  coordinates $t_1,\ldots, t_n$ give  a canonical map $U_n\rightarrow (\Pro^1)^n$, which identifies $U_n$ with the blow-up $P_S$ of $(\Pro^1)^n$ defined in \S\ref{sectCompactification}. Thus, we have
$$P_S\backslash X' \cong \overline{\Mod}_{0,n+3}\backslash A \cong U_n \ . $$
This identifies the following divisors on $P_S\backslash X'$, $\overline{\Mod}_{0,n+3}\backslash A$, and $U_n$, respectively:
\begin{eqnarray}
D_I \!\!&\leftrightarrow& \!\!  \{0, t_k: k\in I\} \big| \{\infty,1,  t_k: k\notin I\}  \, \leftrightarrow \,\, s^\pi_{|I|}=0 \hbox{ on every } U_{\pi} \hbox{ st } \pi(I)=I  \nonumber \\
D^J \!\!&\leftrightarrow&  \!\! \{0, 1, t_k: k \notin J\} \big| \{\infty, t_k: k\in J \} \, \leftrightarrow \, s^\pi_{n+1-|J|}=0 \hbox{ on every } U_{\pi} \hbox{ st } \pi(J)=J \nonumber 
\end{eqnarray}
We deduce that  $D_I\cap D_{I'}\neq \emptyset$ if and only if $I\subseteq I'$ or $I'\subseteq I$ (and likewise for $D^J,D^{J'}$) and $D_I\cap D^J \neq \emptyset$ if and only if $I\cap J =
\emptyset$.  
\end{lem}
\begin{proof}  Straightforward. One must only verify that the sector coordinates $s^{\pi}_k$ are precisely the local coordinates that  one obtains when  one blows up  $(\Pro^1)^n$
along divisors $F_I$ and $F^J$ in order of increasing dimension.
\end{proof}

In conclusion, we have three descriptions of the space $U_n$: first,  as a certain blow-up of $(\Pro^1)^n$ along the boundary of the 
hypercube; second, as the gluing together of affine  schemes $U_{\pi}$; and third, as the complement in $\overline{\Mod}_{0,n+3}$ of a certain family of divisors.

\begin{cor} Let $D\subset U_n$ be an irreducible divisor defined locally  by the vanishing of  an $s^{\pi}_k$.  Then 
$D\cong U_{k-1}\times U_{n-k}$, where $U_{0}$ is a point.
It follows that  the polytope $\mathcal{C}_n$  (which was defined to be the strict transform $\pi^{-1} (\square_n)$ ) has the following product structure on its facets: $\mathcal{C}_n\cap D_I \cong \mathcal{C}_{|I|-1}\times \mathcal{C}_{n-|I|}$ and similarly for $\mathcal{C}_n\cap D^J$.

\end{cor}

\subsection{Absolute convergence of multivalued  series}
Let
\begin{equation} \label{defnofSs}
\Ss=\{(q,t_1,\ldots, t_n):  0<|q|<|t_1|<\ldots <|t_n|<1,  \qquad  t_it_j^{-1}\notin q^\R \ , t_i \notin q^\R \}\ ,
\end{equation}
and let   $f$ be a unipotent function on $(\C^{\times})^n\backslash \{t_i=t_j\}$, i.e., a multivalued function on $\Mod_{0,n+3}(\C)\subset  \overline{\Mod}_{0,n+3}(\C)$ with everywhere local unipotent monodromy  around boundary divisors (definition \ref{propunipgrowth}). Consider a sum
\begin{equation}\label{Fsum} F(t_1,\ldots, t_n;q) = \sum_{m_1,\ldots, m_n\in \Z} u_1^{m_1}\ldots u_n^{m_n} f(q^{m_1}t_1,\ldots, q^{m_n}t_n)\ ,\end{equation}
By \S\ref{sectprepmap}, the values $f(q^{m_1}t_1,\ldots, q^{m_n}t_n)$ are
well-defined if we fix a branch of $f$ near some point $(t_1,\ldots, t_n)$, where 
$(q,t_1,\ldots, t_n) \in \Ss$.
We give sufficient conditions on the auxilliary variables $u_i$ to ensure the  absolute convergence of such a series, by bounding the terms of  $F$ in  different sectors.

\begin{defn} For $0< \varepsilon \ll1$,   let $U^{\varepsilon}_{\pi} \subset U_{\pi}(\C)$ denote the open set of points 
$$\{(s^{\pi}_1,\ldots, s^{\pi}_n)\in U_{\pi}(\C) : |s^{\pi}_i|<1 \hbox{ for all } i, \hbox{ and } |s_i^{\pi}| < \varepsilon \hbox{ for some } 1\leq i\leq n\}\ , $$
 and let $U^{\varepsilon}= \bigcup_{\pi\in \Sym_{n+1}} U^{\varepsilon}_{\pi}$.
\end{defn}
Let $K$ be a compact subset of $\Ss$  defined in $(\ref{defnofSs})$.
\begin{lem} \label{lemfinitenopoints} For $(q,t_1,\ldots, t_n)$ in $K$,  there are only finitely many  $\underline{m}=(m_1,\ldots, m_n)$ such that
$(q^{m_1}t_1,\ldots, q^{m_n}t_n)$ lies in the complement of $U_{\varepsilon}$, and 
 $f(q^{m_1}t_1,\ldots, q^{m_n}t_n)$ is uniformly bounded for such $\underline{m}$. 
\end{lem}
\begin{proof}
The first part follows since the complement of $U^{\varepsilon}$  in $\overline{\Mod}_{0,n+3}$ is compact, and does not contain the total transform  of any divisors $t_i=0,\infty$. 
 The definition of $\Ss$ ensures that $q^{m_i}t_i\neq q^{m_j}t_j $ and $q^{m_i}t_i \neq 1$ for all $m_i, m_j$, $i\neq j$,  and since these are the possible singularities of  $f(q^{m_1}t_1,\ldots, q^{m_n}t_n)$, it  is uniformly bounded on  $K$. 
\end{proof}

All the remaining terms of  $(\ref{Fsum})$  lie in some sector   $U^{\varepsilon}_\pi$. 
  Let us fix one such  permutation $\pi$ and work in 
local sector coordinates $s^{\pi}_1,\ldots, s^{\pi}_n$.
Then $U^{\varepsilon}_{\pi}$ can be further decomposed into smaller pieces as follows. For any non-empty  $A\subseteq \{1,\ldots,n \}$, let 
$$N^{\pi}_A= \{(s^{\pi}_1,\ldots, s^{\pi}_n) : \,\, |s^{\pi}_i|< \varepsilon \hbox{ for  } i \in A, \quad 1> |s^{\pi}_i|\geq \varepsilon \hbox{ for } i \notin A\}\ .$$
We clearly have:
$$U_{\pi}^{\varepsilon} = \bigcup_{\emptyset\neq A\subseteq \{1,\ldots, n\}} N^{\pi}_A\ .$$
\vspace{-0.2in}
\begin{figure}[h!]
 \begin{center}
    \leavevmode
    \epsfxsize=4.0cm \epsfbox{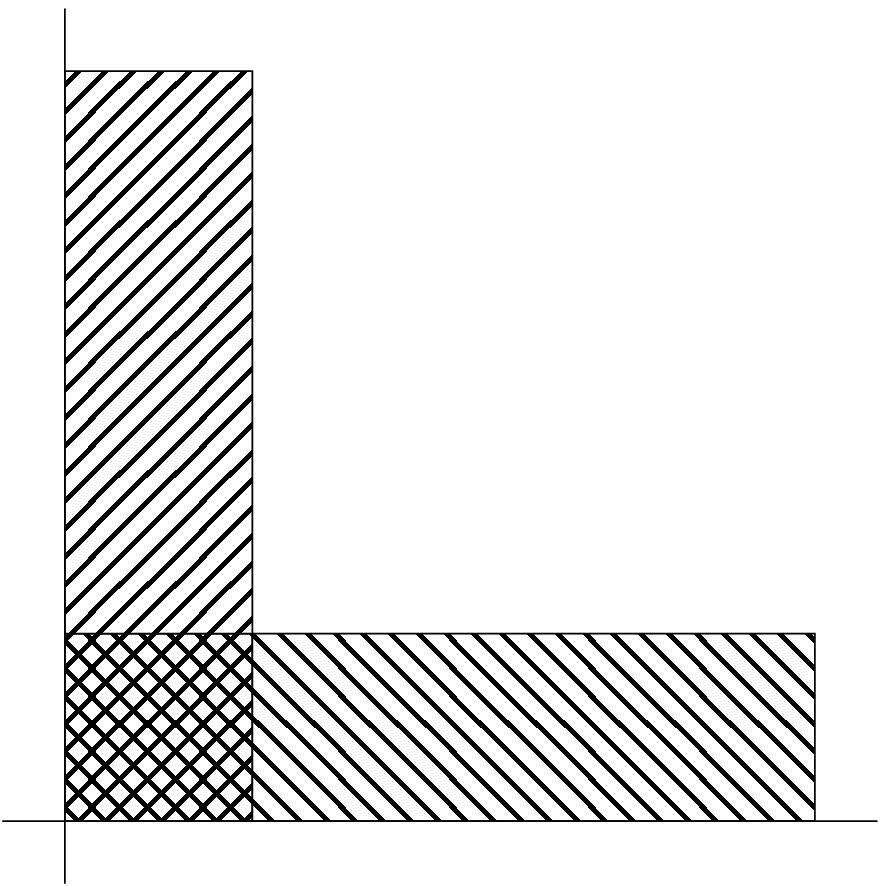}
\put(-50,18){$N^{\pi}_{2}$}
 \put(-100,18){$N^{\pi}_{12}$}
\put(-100,60){$N^{\pi}_1$}
\put(-117,100){$1$}
\put(-117,30){$\varepsilon$}
\put(-85,0){$\varepsilon$}
\put(-12,0){$1$}
  \end{center}
 \end{figure}
\vspace{-0.2in}

\begin{prop} \label{propfbound} There is a constant $C\in \R$ such that for all $(s^{\pi}_1,\ldots, s^{\pi}_n) \in N^{\pi}_A$,   
$$|f(s^{\pi}_1,\ldots, s^{\pi}_n) |\leq  C\, \big(\prod_{i\in A} \kappa_i(s^{\pi}_i) \big) \, f_A((s^{\pi}_j)_{j\notin A})\ ,$$
where  $f_A(s^{\pi}_j)_{j\notin A}$ is a unipotent function on $\bigcap_{i\in A} \{s^{\pi}_i=0\}$, and 
$$\kappa_i(s) = 
|s|^{M_i}\log^w |s| \ ,$$
where $f$ vanishes along $s^{\pi}_i=0$ to order $M_i\geq 0$, and $w$ is some integer $\geq 0$.
\end{prop}
\begin{proof} This follows immediately from the local expansion of a unipotent  multivalued function  in the neighbourhood of a normal crossing divisor (definition 
$(\ref{fexpansion})$).
\end{proof}
Recall from  definition $(\ref{sectorcoords})$ that $s^{\pi}_i = t_{\pi(i)}   t^{-1}_{\pi(i+1)}.$
Thus the action of the summation index $(m_1,\ldots,m_n)\in\Z^n$ on the sector coordinate $s^{\pi}_i$ is given by
 $s^{\pi}_i\mapsto q^{p^{\pi}_i} s^{\pi}_i$,
 where
 \begin{equation}\label{varspi}
 p^{\pi}_i = m_{\pi(i)} - m_{\pi(i+1)} \ .\end{equation}
 By lemma \ref{lemSpi2IJ}, the  divisor $s_{k}^{\pi}=0$ corresponds   to a divisor $D_I$ or $D^J$. Let us define
 \begin{equation}\label{vsdef}
 v^{\pi}_k= \prod_{i\in I} u_i  \qquad \hbox{ or } \qquad v^{\pi}_k= \prod_{j\in J} u_j^{-1} \end{equation}
 accordingly, where $u_i$ are the parameters in $(\ref{Fsum})$.
One verifies  that 
$$(v_1^{\pi})^{p^{\pi}_1} \ldots (v_n^{\pi})^{p^{\pi}_n}  = u^{m_1}_1\ldots u_n^{m_n}\ .$$
Thus, for the terms of $(\ref{Fsum})$  which lie in the sector $U_{\pi}$, we can write
$$u^{m_1}_1\ldots u_n^{m_n} f(q^{m_1}t_1,\ldots, q^{m_n}t_n) = (v_1^{\pi})^{p^{\pi}_1} \ldots (v_n^{\pi})^{p^{\pi}_n}f(q^{p^{\pi}_1}s^{\pi}_1,\ldots, q^{p^{\pi}_n} s^{\pi}_n) $$
in  local coordinates. Dropping cluttersome $\pi$'s from the notation, we have: 
 
 \begin{cor} For all $(q,t_1,\ldots, t_n)\in K$, there exists  a constant $C$ such that
 $$\big| v_1^{p_1}\ldots v_n^{p_n} f(q^{p_1}s^{\pi}_1,\ldots, q^{p_n} s^{\pi}_n) \big| \leq C \prod_{i\in A} 
|v_i q^{M_i}|^{p_i}  |p_i|^w $$
for all $(p_1,\ldots, p_n)\in \Z^n$ such that $(q^{p_1}s^{\pi}_1,\ldots, q^{p_n} s^{\pi}_n)\in N^{\pi}_A$.
 \end{cor}
 \begin{proof} Consider the bound in  proposition \ref{propfbound}. The function $f_A$ on the right-hand side can be uniformly bounded by  a version of  lemma \ref{lemfinitenopoints},
 applied to $\bigcap_{i\in A} \{s^{\pi}_i=0\},$ which is  isomorphic to a product of $U_k$'s.
 Since $\log^w | q^{p_i} s^{\pi}_i| = (p_i \log |q| + \log |s^{\pi}_i|)^w \leq C_i |p_i|^w$ for some constant $C_i$, for all $i \in A$, we deduce the bound stated above.
 \end{proof}
 
\begin{thm} \label{thmconv} Suppose that the chosen branch of $f$ vanishes along all divisors $D_I$ of type $0$ with multiplicity $|I|$.
Then  $(\ref{Fsum})$ converges absolutely on compacta of the polydisc 
 $$1<u_1, \ldots,u_n<|q|^{-1}$$ 
\end{thm}
\begin{proof} Consider first a divisor $D_I$ of type $0$. Then, in the previous corollary, the divisor $D_I=\{s^{\pi}_k=0\}$  in some chart $U_{\pi}$  will  correspond in the right-hand side  to  terms  of the form
$|v^{\pi}_kq^{|I|}|^{p} p^w$, where $p$ is large and positive. The assumptions on $u_i$ imply that 
  $$|v^{\pi}_k q^{|I|}|^p p^w = \Big( \prod_{i\in I} |u_i q| \Big)^p p^w <1\ ,$$
  and therefore  $|v^{\pi}_kq^{|I|}|^{p} p^w$ tends to zero exponentially fast in $p$. Now consider a divisor $D^J$ of type $\infty$.  It corresponds to terms of the form
$|v^{\pi}_k q^{M_k}|^{p}p^w$, where $p$ is large and positive and $M_k,w\geq 0$.  But by the assumptions on $u_i$, we have
  $$|v^{\pi}_k | = \Big( \prod_{j\in J} u^{-1}_j \Big) <1\ ,$$
and so once again, $|v^{\pi}_k q^{M_k}|^{p}p^w$ tends to zero exponentially fast in $p$.
\end{proof}

\subsection{Structure of the  poles}\label{sectPolestructure}
We  first make some general remarks about the pole structure as follows from the proof of theorem \ref{thmconv}. In \S\ref{sectpoleElliptic} we shall refine this result in the case of the multiple elliptic polylogarithms by exploiting their differential equation.

\begin{cor} \label{cor40} Let $f$ satisfy the conditions of  theorem $\ref{thmconv}$. For every $I\neq \emptyset$, let 
$w_I$ denote the order of the logarithmic singularity of $f$ along $D^I$. Every codimension $h$  face  of the  polytope $\mathcal{C}_n$ corresponds to a  flag
$I_1\subsetneq I_2 \subsetneq \ldots \subsetneq I_h, $
where $I_1,\ldots, I_h\subseteq \{1,\ldots, n\}$, and is contained in  the intersection  $E=D^{I_1}\cap \ldots \cap D^{I_h}$.
Let $s_1,\ldots, s_h, s_{h+1},\ldots, s_n $ denote local normal coordinates in which $D^{I_k}$ is given by $s_k=0$ for $1\leq k \leq h$.
Then in the neighbourhood of $E$, the function $f$ has an expansion of the form
$$f = \sum_{i_1\leq w_{I_1},\ldots, i_h\leq  w_{I_h} }  f_{i_1,\ldots, i_h}(s_{h+1},\ldots, s_n) \log^{i_1}(s_1) \ldots \log^{i_h} (s_h)\ ,$$
where $(s_{h+1},\ldots, s_n)$ are coordinates on $E$. After averaging, each term in the sum which is  indexed by $i_1,\ldots, i_h$ contributes singularities  to $(\ref{Fsum})$ of the form
$$(\prod_{k\in I_1} u_k -1)^{-j_1} \ldots (\prod_{k\in I_h} u_k-1)^{-j_h}\ ,  $$
with $j_1\leq i_1+1, \ldots, j_h\leq i_h+1$. In particular, the term $f_{0,\ldots, 0}(s_{h+1},\ldots, s_n)$ which is constant in $s_1,\ldots, s_h$, contributes a simple pole of the form
$$(\prod_{k\in I_1} u_k -1)^{-1} \ldots (\prod_{k\in I_h} u_k-1)^{-1}\ . $$
\end{cor}
\begin{proof} Following the method of proof of the previous theorem, one sees that the statement reduces to a local computation in the one-dimensional situation. In this case, it is clear that the averaging procedure applied to  $\log^i z$, for $i\geq 0$,  gives
$$\sum_{m\geq0 } u^m \log^i q^m z=\sum_{m\geq0 } u^m (m \log q+\log z)^i\ , $$
which has  a pole at  $u=1$ of order at most $i+1$.
\end{proof}

\begin{rem} Corollary \ref{cor40}  gives an upper bound on the singularities which occur:   it can happen that summing over one sector gives rise to   spurious poles in the $u_i$'s,  which cancel on taking the total contribution over all sectors. 
 \end{rem} 
The upshot of the previous corollary is that if we know the differential equations satisfied by $f$, then we can deduce the pole structure of $F$ completely from these differential equations, up to constants of integration (see $\S\ref{sectexamples123}$ below). The corollary states that the constants of integration necessarily contribute simple poles, and these are in bijection with the $n!$ maximal flags $I_1\subsetneq I_2\subsetneq \ldots \subsetneq I_{n-1} \subsetneq I_n$.

\begin{example}
In the case $n=2$, we have  poles along $u_i=1$ coming from logarithmic singularites along $D^i$, for $i=1,2$,  and  along $u_1u_2=1$ coming from $D^{12}$.
The typical contribution from $D^1$, for example, is of the form
$$\sum_{i=0}^{w}  { R_i(\xi_2; u_2)  \over (u_1-1)^{i+1}}$$
where $R_i(\xi_2;u_2)$ is the result of averaging a function of $t_2$. The two maximal flags $\{1\}\subset \{1,2\}$ and $\{2\}\subset \{1,2\}$  which correspond to the corners $D^1\cap D^{12}$ and $D^2 \cap D^{12}$, give constant contributions of the form
$${ c_{1,12} \over (u_1-1)(u_1u_2-1) } \quad \hbox{ and } \quad {c_{2,12} \over (u_2-1)(u_1u_2-1)}$$
where $c_{1,12}$ (resp. $c_{2,12}$) is the regularized limit of $f$ at $D^1\cap D^{12}$ (resp. $D^2 \cap D^{12}$).
For an explicit computation of  such a pole structure, see  \S\ref{sectDepth2}.
\end{example}

\section{Elliptic multiple polylogarithms}  \label{sect8}
In this section, we apply the results of \S\ref{sectAverage} to  prove the following theorem. 
\begin{thm} \label{thmmainconv} The series obtained by averaging the classical multiple polylogarithm
\begin{equation}\label{Lsum}
E_{n_1,\ldots, n_r} (\xi_1,\ldots, \xi_r;u_1,\ldots, u_r)= \sum_{m_1,\ldots, m_r\in \Z}  u^{m_1}_1 \ldots u^{m_r}_r I_{n_1,\ldots, n_r} (q^{m_1}t_1,\ldots, q^{m_r}t_r) \nonumber
\end{equation} 
converges  for $1<u_1,\ldots, u_r<|q|^{-1}$, and $(q,t_1,\ldots, t_r)\in \Ss$.  It defines a (generating series) of functions on $\E^{(r)}$ with poles  given by  products of consecutive $u_i$'s only:
$$u_i=|q|^{-1} \hbox{ for }  1\leq i\leq r \ ,  \quad \hbox{ and }  \prod_{i\leq k\leq j} u_k=1 \hbox{ for all } 1\leq i\leq j\leq r\ .$$
\end{thm}
In order to extract a convergent Taylor expansion in the variables $\alpha_i$, where $u_i=\e(\alpha_i)$, it suffices to know the exact asymptotic behaviour of 
$I_{n_1,\ldots, n_r} (t_1,\ldots, t_r)$ at infinity. This is carried out for the Debye polylogarithms in $\S7$.

\subsection{Analytic continuation  of  polylogarithms} 
\label{sectanalytic} The function $I_{n_1,\ldots, n_r}(t_1,\ldots, t_r)$ has a convergent Taylor expansion at the origin, and so defines the germ of a multivalued analytic function on $\Mod_{0,n+3}(\C)\subset \C^n$. As is well-known,  it is unipotent by $(\ref{I11diff})$,  has a canonical branch at the origin,  and vanishes along the divisors $t_i=0$.

It therefore extends by analytic continuation to a multivalued function on the blow-up $U_n\backslash X(\C)$, and can have at most logarithmic divergences  along the boundary components $D^J$ 
and $D_I$ of $X$.   It turns out that for some of these components $D$, there is no logarithmic divergence and we can speak of the continuation
$I_{m_1,\ldots, m_n}(t_1,\ldots, t_n)\big|_D$ 
of $I_{m_1,\ldots, m_n}(t_1,\ldots, t_n)$ to
$D$. This is not well-defined, since the function is multivalued. However, if  $D$ also meets the strict transform of a divisor $t_i=0$, for some $i$, there is a canonical branch which vanishes at  $\{t_i=0\} \cap D$, and defined in a neighbourhood of $\{t_i=0\}$.  
The following   lemma  is the key to the absolute convergence of $(\ref{Lsum})$.
\begin{lem} \label{lemR0vanishing}  Let $D_I$ be of type $0$. Then $I_{m_1,\ldots, m_n}(t_1,\ldots, t_n)\big|_{D_I}$ vanishes to order $|I|$.
\end{lem}
\begin{proof}  The sum $(\ref{Imassum})$ 
 converges in the neighbourhood of the origin and locally defines a holomorphic function   which vanishes along the divisors $t_i=0$. 
It  therefore vanishes on any exceptional divisor $D_I$ lying above the origin  to order $|I|$. 
\end{proof}

Setting $f(t_1,\ldots,t_n) = I_{m_1,\ldots, m_n}(t_1,\ldots, t_n)$ proves the first part of theorem \ref{thmmainconv}.

\begin{lem} \label{lemlimitpolylog}
Let   $J\subset \{1,\ldots, n\}$ be non-empty, and $1\notin J$. Let $J^c = \{1,\ldots, n\} \backslash J $ and write $J^c= \{i_1,\ldots, i_k\}$ where $1=i_1<\ldots <i_k$. Then
$$I_{m_1,\ldots, m_n}(t_1,\ldots, t_n)\big|_{D^J}  = (-1)^{|J|} I_{m'_{1},\ldots, m'_k}(t_{i_1},\ldots,t_{i_k})\ ,$$
where $m_1'= m_{i_1}+\ldots + m_{i_2-1}$, $m_2'= m_{i_2}+\ldots + m_{i_3-1}$,\ldots,  $m'_k= m_{i_k}+\ldots + m_n$.
\end{lem}
\begin{proof} This follows immediately from the iterated integral representation $(\ref{itintrep})$.
\end{proof}

\begin{cor}
Every multiple polylogarithm $I_{m_1,\ldots, m_n}(t_1,\ldots, t_n)$ is the analytic continuation to some exceptional divisor of a multiple logarithm
$I_{1,\ldots, 1}(t_1,\ldots, t_{N}).$
\end{cor}
Thus we can restrict ourselves to considering  only multiple logarithms if we wish.
Another way to interpret lemma \ref{lemlimitpolylog} is to notice that the terms $I\big|_{D^J}$ are in one-to-one correspondence with the  terms in the so-called stuffle product formula.

\subsection{Elliptic Multiple Polylogarithms}
The functions obtained by  averaging  multiple polylogarithms  satisfy differential equations which  are easily deduced from   $(\ref{I11diff})$.

\begin{lem}  \label{EKaveraged} The function $F(\xi;u)$ is the averaged  weighted generating series for ${z\over z-1}$:
$$F (\xi;u) =- 2 \pi i \sum_{n\in \Z} {q^n z \over 1-q^nz} u^n \ .$$
\end{lem}
\begin{proof} 
By decomposing the domain of summation into various parts we obtain:
\begin{eqnarray} 
\sum_{n\in \Z } {q^n z \over 1-q^n z} u^n  &=  &\sum_{n<0} {- q^{-n} z^{-1}\over 1-q^{-n} z^{-1}} u^n  -\sum_{n<0} u^n +  {z\over 1-z} +\sum_{n>0} {q^{n} z \over 1- q^n z} u^n    \nonumber \\
&= &\sum_{n>0} \sum_{m>0} q^{mn} (-z^{-m} u^{-n} + z^{m}u^{n}) +{z\over 1-z} +{ 1\over 1-u}\nonumber 
\end{eqnarray}
which is the definition of the  Eisenstein-Kronecker series $-(2i\pi)^{-1} F(\xi;u)$.\end{proof}
\begin{lem}  \label{EKaveragedqinv} The  averaged  weighted generating series for $d\Li_1(z)$ is:
$$\sum_{n\in \Z}  d\Li_1(zq^n) u^n  =   F(\xi;u) d\xi \ .$$
Likewise, the result of averaging $z^{-s} d\Li_1(z)$ is $\e(-\xi s) F(\xi;u) d\xi$.
\end{lem}
\begin{proof}
Follows from the previous lemma using the fact that $d\xi = {1\over 2i\pi } {dz \over z}$.
\end{proof}

Let us define the (unregularized)  multiple elliptic polylogarithm to be:
$$E_{n_1,\ldots, n_r}(\xi_1,\ldots, \xi_r;\alpha_1,\ldots, \alpha_r) = \sum_{m_1,\ldots, m_r\in \Z}  u_1^{m_1}\ldots u_r^{m_r} I_{n_1,\ldots,n_r}(q^{m_1}t_1,\ldots, q^{m_r}t_r) \ .$$
where $u_i = \e(\alpha_i)$ for $1\leq i\leq r$.

\begin{thm} \label{thmDiffforE} The total   derivative  
$dE_{1,\ldots, 1}(\xi_1,\ldots, \xi_n;\alpha_1,\ldots, \alpha_n)$ equals\begin{eqnarray} 
&=& \sum_{k=1}^{n} dE_1(\xi_k-\xi_{k+1};\alpha_k) E_{1,\ldots, 1}(\xi_1,\ldots, \widehat{\xi}_k, \ldots, \xi_n; \alpha_1,\ldots, \alpha_{k}+\alpha_{k+1},\ldots, \alpha_n) \nonumber \\
&-&  \sum_{k=2}^{n} dE_1(\xi_k-\xi_{k-1};\alpha_k) E_{1,\ldots, 1}(\xi_1,\ldots, \widehat{\xi}_k, \ldots, \xi_n; \alpha_1,\ldots, \alpha_{k-1}+\alpha_{k},\ldots, \alpha_n) \nonumber 
\end{eqnarray}
where $\xi_{n+1}=0$, $\alpha_{n+1}=0$,  and $dE_1(\xi;\alpha) = F(\xi;\alpha)d\xi$.
\end{thm}
\begin{proof}  Since it converges uniformly, we can differentiate term by term in the definition of $E_{1,\ldots, 1}$. 
 The differential equation then follows  from the corresponding differential equation $(\ref{I11diff})$ for $I_{1,\ldots, 1}(t_1,\ldots, t_n)$. The key observation is that a term such as
 $$\sum_{m_1,\ldots, m_n\in \Z} dI_1 \Big({q^{m_k} t_{k}\over q^{m_{k+1}} t_{k+1}}\Big) I_{1,\ldots, 1}(q^{m_1} t_1,\ldots, \widehat{q^{m_k} t_k},\ldots, q^{m_n}t_n)\, u_1^{m_1}\ldots u_n^{m_n}$$
 can be rewritten in the form   
 \begin{eqnarray}
\sum_{m_1,\ldots, m_n\in \Z}&&  dI_1 \Big(q^{m_k-m_{k+1}}{ t_{k}\over  t_{k+1}}\Big) \, u_k^{m_k-m_{k+1}}  \nonumber \\
  & \times & I_{1,\ldots, 1}(q^{m_1} t_1,\ldots, \widehat{q^{m_k} t_k},\ldots, q^{m_n}t_n)\, u_1^{m_1}\ldots(u_ku_{k+1})^{m_{k+1}}\ldots u_n^{m_n} \nonumber 
 \end{eqnarray}
 and the region of summation decomposes into a product after a triangular change of basis of the summation variables
 $(m_1,\ldots, m_n) \mapsto (m_1,\ldots, m_k-m_{k+1},\ldots, m_n)$.
 \end{proof}

\subsection{Elliptic Debye polylogarithms}

Recall the definition of the classical Debye polylogarithms (definition \ref{defnclassicalDebye}). Let us write
$\underline{\alpha}=(\alpha_1,\ldots, \alpha_r)$ and likewise for $\beta$.
\begin{defn}  \label{defnEllDebye} The generating series of elliptic Debye polylogarithms is:
$$\EE_r(\xi_1,\sdots, \xi_r; \underline{\alpha},\underline{\beta}) = \sum_{m_1,\sdots, m_r\in \Z}\e(m_1\alpha_1+\cdots+m_r\alpha_r) \Lambda_r(q_1^{m_1} t_1,\sdots,q_r^{m_r} t_r;\beta_1,\sdots,\beta_r)$$
The absolute convergence of the series is guaranteed by theorem \ref{thmconv}.
\end{defn}
One of the main reasons for considering such a generating series is because of a mysterious modularity property relating the parameters $\alpha$ and $\beta$ (see \cite{Andrey} when $r=1$).
\begin{prop} \label{propDiffforEDebye} Let $r\geq 2$. The differential 
$d\EE_r(\xi_1,\sdots, \xi_r;\underline{\alpha},\underline{\beta})$ is equal to \begin{eqnarray} 
&=& \sum_{k=1}^{n} d\EE_1(\xi_k-\xi_{k+1};\alpha_k,\beta_k)\EE_{r-1}(\xi_1,\sdots, \widehat{\xi}_k, \sdots, \xi_r; \alpha_1,\sdots, \alpha_{k}+\alpha_{k+1},\sdots, \alpha_n,  \beta_1,\sdots) \nonumber \\
&-&  \sum_{k=2}^{n} d\EE_1(\xi_k-\xi_{k-1};\alpha_k,\beta_k) \EE_{r-1}(\xi_1,\sdots, \widehat{\xi}_k, \sdots, \xi_r; \alpha_1,\sdots, \alpha_{k-1}+\alpha_{k},\sdots, \alpha_n, \beta_1,\sdots) \nonumber 
\end{eqnarray}
where $\xi_{n+1}=\alpha_{n+1}=\beta_{n+1}=0$, and in the right-hand side, the arguments in the $\beta$'s are of the same form as those for the $\alpha$'s. In the case $r=1$,  we have
\begin{equation} \label{dEequation} d\EE_1(\xi;\alpha;\beta) = \e(-\beta\xi) \, F(\xi; \alpha-\tau \beta) d\xi\ .\end{equation}
\end{prop}
The proof follows immediately from theorem $\ref{thmDiffforE}.$ 

\subsection{The structure of the poles of elliptic polylogarithms} \label{sectpoleElliptic} Let us write  
\begin{equation}\label{gammaidef}
\gamma_i= \alpha_i - \tau\, \beta_i  \quad \hbox{ for }  1\leq i\leq r\ . \end{equation}
\begin{prop}
 The Debye elliptic polylogarithms (definition \ref{defnEllDebye}) have at most simple poles along the divisors which have consecutive indices only:
 $$\sum_{i\leq j\leq k} \gamma_i =0 \qquad \hbox{ and } \qquad \sum_{i\leq j\leq k} \alpha_j=0\ .$$
 The multiple elliptic polylogarithm $E_{m_1,\sdots, m_n}(\xi_1,\sdots, \xi_n;\alpha_1,\sdots,\alpha_n)$  has poles along divisors of the form
$\sum_{i\leq j\leq k} \alpha_j=0$ of order at most $m_1+\ldots+m_n+1$.
\end{prop}
\begin{proof}  By induction.  
Suppose  that 
$\EE_r$  has simple  poles  along $\sum_{i\leq j \leq k} \alpha_j=0$, and  $\sum_{i\leq j\leq k} \gamma_j=0$  with consecutive indices only. This is automatically true for $r\leq 2$.
It follows from the shape of the differential equation (proposition \ref{propDiffforEDebye}), that $d\EE_{r+1}$ only has  simple poles along 
 $\sum_{i\leq j \leq k} \alpha_j=0$ and $\sum_{i\leq j\leq k} \gamma_j=0$. Thus the same conclusion also holds for $\EE_{r+1}$, except that the  constants of integration might give rise to supplementary poles. To see that such constants of integration must be zero,  let $I$ be a set of non-consecutive indices.  The divisor $D^I$ meets a divisor of the form $t_j=0$, for $j\notin I$, along which 
 the function $\Lambda(t_1,\ldots, t_n;s_1,\ldots, s_n)$ vanishes. It follows from the discussion above that $\Lambda$ has no divergence in the neighbourhood of $D^I\cap \{t_j=0\}$, and hence 
 no pole in either $\sum_{i\in I} \alpha_i=0$ or $\sum_{i\in I} \gamma_i=0$. This proves the result for  the generating series $\EE_r$.  The  corresponding statement for its coefficients $E_{m_1,\sdots,m_r}(\xi_1,\sdots, \xi_r;\alpha_1,\sdots,\alpha_r)$ follows on taking a series expansion in the  $\beta_i$.
\end{proof}
This  method  for computing the pole structure  is illustrated below (\S\ref{sectDepth2}).

\section{Asymptotics of Debye polylogarithms.} \label{sect7}
In the previous section we showed that the polar contributions in the averaging process come from the asymptotic expansion of polylogarithms at infinity. 
This expansion can be computed  explicitly in terms of a combinatorially defined coproduct.

\subsection{The coproduct for Debye polylogarithms}
The Debye multiple polylogarithms are defined by iterated integrals, and so by the general theory \cite{Ch1}  admit a coproduct which is dual to the composition of paths.
We describe it explicitly below.

\begin{defn}   Let $n\geq 1$. Let $I=\{1,\ldots, n\}$ be an ordered set of indices and let $\beta_1,\ldots, \beta_n$ be formal variables satisfying $\sum_{i=1}^n \beta_i=0$.
 Define a \emph{string} in $I$   to be a
consecutive subsequence $S=(i_1,i_2,\ldots, i_l)$ of length $2\leq l<n $, which is either in increasing or decreasing order and  such that $i_1 \neq n$.  Let
 $$\beta_S= \beta_{i_1}+\beta_{i_2}+\ldots + \beta_{i_{l-1}}\ . $$
For any such string $S=(i_1,\ldots, i_l)$, let $A_S$ denote the symbol
 \begin{equation}\label{ASdef}
  A_S=(t_{i_1}:t_{i_2}: \ldots :t_{i_l};\beta_{i_1},\beta_{i_2},\ldots, \beta_{i_{l-1}},-\beta_S)\ .
  \end{equation}
Let $H_n$ denote the commutative ring over $\Z$ generated by all symbols $A_S$ as $S$ ranges over the set of strings in $1,2,\ldots, n$.   The length of a string defines a grading on $H_n$.
\end{defn}
The Debye  polylogarithm defines a map from $H_n$ to generating series of  multivalued functions. If a string $S$ is given by 
$(\ref{ASdef})$ then we have
\begin{equation} \label{LambdaAS}
\Lambda (A_S) = \Lambda_{l-1} \Big( {t_{i_1}\over t_{i_\ell}}, {t_{i_2}\over t_{i_\ell}}, \ldots ,{t_{i_{l-1}}\over t_{i_\ell}};\beta_{i_1},\beta_{i_2},\ldots, \beta_{i_{l-1}}\Big)\ .
\end{equation}

Denote the last element  of  a string by  $\ell(i_1,i_2, \ldots, i_l)=i_l$, and 
 define the sign  $\sign(S)$ of $S$ to be  $1$ if $S$ is in increasing order, or 
 $(-1)^{l-1}$ if  it is in decreasing order.

\begin{defn} A finite collection $\mathcal{S}=\{S_\alpha\}$ of strings is \emph{admissible} if the strings intersect at most in  their last indices, i.e., if $S_\alpha, S_\beta$ are in $\mathcal{S}$ and ${\alpha}\neq {\beta}$ then either
\begin{eqnarray} 
& (1) & S_{\alpha}\cap S_{\beta}=\emptyset  \nonumber \\ 
\hbox{or } & (2) &  S_{\alpha}\cap S_{\beta} = \{\ell\}\ ,  \quad \hbox{where } \ell=\ell(S_\alpha)=\ell(S_\beta) \ . \nonumber 
\end{eqnarray}
Given an admissible set of strings $\mathcal{S}=\{S_\alpha\}$, define the set of  remaining indices
$$R_{{\cal S}} = \big( I \setminus\bigcup S_{\alpha}\big)  \cup \bigcup \ell(S_{\alpha})$$
with the ordering induced from $I$, and define the corresponding quotient sequence
 $$Q_{\cal S}=(t_{j_1}:t_{j_2}:\ldots :t_{j_m};\tilde{\beta}_{j_1},\tilde{\beta}_{j_2},\ldots ,\tilde{\beta}_{j_m})\ ,$$
  where $({j_1},{j_2},\ldots ,{j_m})=R_{{\cal S}}$ and $\tilde{\beta}_{j}=
\beta_j+\sum_{\alpha,\ell(S_\alpha)=j}\beta_{S_\alpha}$.\end{defn}

\begin{defn} Define a map $\Delta':H_n \To H_n \otimes H_n$ by 
\begin{equation}\label{coproddef}
\Delta' A_{J}=\sum_{\mathcal{S}=\{S_1,\ldots, S_k\}}  \sign(S_1)A_{S_1}\ldots \sign(S_k)A_{S_k} \otimes  Q_{\cal S}
\end{equation}
where $J\subseteq \{1,\ldots, n\}$,  and the sum is over all non-empty admissible collections of strings $\mathcal{S}$ in $J$ such that $Q_{\cal S}$ has at least two elements.
\end{defn}

Consider  the map which sends  $A_S$ of $(\ref{ASdef})$ to $ t_{i_1}^{\beta_{i_1}}\ldots t_{i_l}^{\beta_{i_l}}$, with $\beta_{i_l}=-\beta_S$, and extend by multiplicativity.
Then each term in   $(\ref{coproddef})$  maps to $\pm t_1^{\beta_1}\ldots t_n^{\beta_n}$.

\begin{example}  Writing $\beta_{ij}$ for $\beta_{\{i,j\}}=\beta_i+\beta_j$, and so on, formula  $(\ref{coproddef})$ gives: 
\begin{eqnarray} \label{depth2coprod}
\qquad \Delta' (t_1:t_2:t_3; \beta_1,\beta_2,\beta_3) &=& (t_1:t_2;\beta_1,-\beta_1)\otimes (t_2:t_3 ; \beta_{12},  \beta_3)   \\
   &-&   (t_2:t_1; \beta_2,-\beta_2) \otimes (t_1:t_3; \beta_{12}, \beta_3) \nonumber \\
   &+ &  (t_2:t_3; \beta_2,-\beta_2) \otimes (t_1:t_3 ; \beta_1, \beta_{23}) \ .  \nonumber 
   \end{eqnarray}
In general, a typical term in $\Delta'(t_1:\ldots :t_6; \beta_1,\ldots; \beta_6)$ is
$$ (t_1\!:\!t_2\!:\!t_3; \beta_1,\beta_2,-\beta_{12})(t_4\!:\!t_3;\beta_4,-\beta_4)(t_5\!:\!t_6;\beta_5,-\beta_5)\otimes (t_3\!:\!t_6;\beta_{1234},\beta_{56})
$$
\end{example}

\begin{prop}  Let $\Delta: H_n \rightarrow H_n \otimes H_n$ be $\Delta=1\otimes id + id \otimes 1 + \Delta'$. Then $H_n$, equipped with $\Delta$,  is  a commutative  graded  Hopf algebra.
\end{prop}
\begin{proof} We omit the proof. In fact it suffices to show that the $1$-part of the coproduct coincides with the differential for the  Debye polylogarithms (lemma \ref{lemdiff} below).
\end{proof}

Let  $\Delta^{(m+1)}:H_n \rightarrow H_n^{\otimes m+1}$ denote the $m$-fold iteration of    $\Delta$. Let 
$\Delta^{(m+1)}_{\star,  \ldots, \star, 1, \star, \ldots, \star}$ denote its component whose corresponding tensor factor contains only  strings of length two.
Thus $\Delta_{1,\star}$   extracts all ordered pairs of neighbouring
indices  except   $(n,{n-1})$.

\begin{lem}  \label{lemdiff} The differential equation for $\Lambda$ can be rewritten as 
$$d\,\Lambda =\mu \circ (d\Lambda \otimes \Lambda)\circ \Delta_{1,\star}\ ,$$ where $\mu$ denotes the multiplication map. 
\end{lem}
\begin{proof} Follows from the definition of $\Lambda$ together with the differential equation $(\ref{Debeqn})$.
\end{proof}

\begin{example}  \label{exD1star} For any $i\neq j$, let $t_{[i,j]}=(t_i:\ldots :t_j)$ denote the tuple of consecutive elements. Contributions to $\Delta_{1,\star}$ are of the following kinds (omitting indices $\beta_i$):
\begin{eqnarray}
(t_{i-1}:t_{i})  &\otimes& (t_1:\ldots :\widehat{t}_{i-1}: \ldots : t_n) \qquad \qquad 1<i\leq n \label{left1} \\
(t_{i}:t_{i-1})  &\otimes& (t_1:\ldots :\widehat{t}_{i}: \ldots : t_n)  \qquad \qquad  \quad 1<i<n \label{left2}
\end{eqnarray}

\noindent 
Contributions to $\Delta_{\star,1}$ are of the following kinds:
\begin{eqnarray}
t_{[1,n-1]} &\otimes& (t_{n-1}:t_n)  \label{right1} \\
t_{[2,n]} &\otimes &(t_{1}:t_n) \label{right2} \\
t_{[n-1,1]} &\otimes& (t_{1}:t_n)  \label{right3}\\
t_{[i,1]}t_{[i+1,n]} &\otimes &(t_{1}:t_n) \qquad \qquad 1<i<n-1 \label{right4}\\
t_{[1,k]} t_{[i,k]}t_{[i+1,n]} &\otimes &(t_{k}:t_n)  \qquad \qquad 1<k<i<n-1 \label{right5}
\end{eqnarray}
\end{example}

\subsection{Asymptotic of the Debye polylogarithms}
Let $J$ be a subset of $\{1,2, \ldots,n\}$. We study the asymptotics of   $\Lambda(A_{1,\ldots, n})$, as defined by $(\ref{LambdaAS})$ when $t_j$ for $j\in J$ simultaneously  tend to infinity: i.e., for some finite values $t^0_1,\ldots, t^0_n$, we set:
\begin{equation} \label{ttoinf}
t_j=Tt_j^0\ , \quad  \hbox{for } j\in J\ , \quad    t_k=t_k^0 \quad\hbox{Êfor } k \notin J\ ,  \hbox{ and let } T\rightarrow \infty\ .
\end{equation}

\begin{caveat} \label{caveat} The Debye polylogarithms are multivalued, and so their asymptotics are only well-defined up to monodromy. For divisors of the form $D^I$, where
$I\subsetneq \{1,\ldots, n\}$, and $i\neq I$,  there is a canonical  branch in the neighbourhood of $t_i=0$, where it vanishes (see   \S\ref{sectanalytic}). Only for the divisor $D^{\{1,\ldots, n\}}$  must one make some choice. In the following theorem, this ambiguity is contained in the constant $C$ in equation $(\ref{limitofDebviacoprod})$.
\end{caveat}
We call a string $S=(i_1,i_2, \ldots, i_l)$ \emph{essential} if $i_l\notin J$ and $i_1,\ldots, i_{l-1}\in J$, and \emph{regular} if: either all  indices belong to $J$,  or none of its indices  belongs to $J$. Set
$$ \Lambda^{\reg}(A_S)  =\left\{
                           \begin{array}{ll}
                             \Lambda(A_S)  & \hbox{if } S \hbox{ is regular}\ ,  \nonumber \\
                             0\   & \hbox{otherwise} \ , 
                           \end{array}
                         \right.
$$
  and likewise define $\Phi(A_S)$ to be $0$ if $S$ is non-essential and
$$
\Phi(A_S)= \frac{t_{i_1}^{-\beta_{i_1}}t_{i_2}^{-\beta_{i_2}}\ldots t_{i_n}^{-\beta_{i_n}}}
{\beta_{i_1}\beta_{i_1,i_2}\ldots \beta_{i_1,i_2,\ldots, i_{n-1}}}\qquad \hbox{ if } S \hbox{ is essential } \ .$$
\begin{thm}  \label{thmcoprodasymp} With the  assumptions $(\ref{ttoinf})$ above, for any $0<\varepsilon <\! \! <1$ we have
\begin{equation} \label{limitofDebviacoprod}
\Lambda(t_1:t_2: \ldots :t_n;\beta_1,\beta_2, \ldots,\beta_n)=\mu_3 \circ (\Phi\otimes\Lambda^{\reg}  \otimes C) \circ
\Delta^{(3)}+O(T^{\varepsilon-1})\end{equation}
for some functions  $C(t_1:t_2: \ldots :t_n;\beta_1,\beta_2, \ldots,\beta_n)=C(\beta_1,\beta_2, \ldots,\beta_n)$ which are constant in  the $t$'s, 
and where  $\mu_3$ denotes the triple product.
\end{thm}

\begin{proof} Induction on the depth $n$. For $n=2$ this theorem reduces to the well-known  asymptotics of the classical  Debye  polylogarithm. For strings of length two
$$d\Lambda(t_1:t_2; \beta_1,\beta_2) =  {t_1^{-\beta_1}t_2^{-\beta_2}} d \,\Li_1(t_1/t_2)$$
 where  $\beta_1+\beta_2=0$. It follows from this that for $S=(1,2)$, $A_S=(t_1:t_2;\beta_1,\beta_2)$, 
 $$ d\Lambda(A_S)   =\left\{
                           \begin{array}{ll}
                             d\Phi (A_S) +O(T^{\varepsilon-1}) & \hbox{if }  S \hbox{ is essential}\ ,  \nonumber \\
                            d\Lambda^{\reg}(A_S) \   &    \hbox{if } S   \hbox{ is regular } \ , \nonumber \\
                            O(T^{\varepsilon-1})            &                 \hbox{otherwise} \ .
                           \end{array}
                         \right.  $$
This  follows from the fact that $\Lambda$ diverges at most logarithmically at infinity, and  $\log(T)^a O(T^{\varepsilon-1})    = O(T^{\varepsilon-1})    $. Hence, by lemma \ref{lemdiff} it follows that
asymptotically
\begin{equation} \label{dlambdaasymptotics} d\Lambda \sim \mu \circ (d\,\Phi\otimes \Lambda)\circ\Delta_{1,\star}+   
\mu \circ (d\,\Lambda^{\reg}\otimes \Lambda)\circ\Delta_{1,\star} \ , \end{equation}
where $a\sim b$ means that $a-b = O(T^{\varepsilon-1})    $.
For the induction step,  we first check that the differential of the  difference between both sides  of $(\ref{limitofDebviacoprod})$  vanishes.
 By induction hypothesis, we replace $\Lambda$  in $(\ref{dlambdaasymptotics})$ by $(\ref{limitofDebviacoprod})$:
\begin{equation} \label{inpfoneside}  d\Lambda \sim  \mu_4 \circ( d\,\Phi\otimes \Phi\otimes\Lambda^{\reg}\otimes C +d\,\Lambda^{\reg}\otimes \Phi\otimes\Lambda^{\reg}\otimes C )\circ\Delta^{(4)}_{1,\star,\star,\star}\ . \end{equation} 
Now compute the differential of the right-hand side of $(\ref{limitofDebviacoprod})$, 
\begin{equation} \label{inpfotherside}
\mu_3\circ (d\,\Phi\otimes \Lambda^{\reg} \otimes C)\circ \Delta^{(3)}
+\mu_4\circ( \Phi\otimes d\,\Lambda^{\reg} \otimes\Lambda^{\reg} \otimes C)\circ \Delta^{(4)}_{\star,1,\star,\star}\ . \end{equation}
where in the second term we used lemma \ref{lemdiff} applied to  $d\Lambda^{\reg}.$
In order to show that $(\ref{inpfotherside})$ and the right-hand side of $(\ref{inpfoneside})$ coincide, it suffices to use the coassociativity of the coproduct and to show that the 
following expression vanishes:
$$ \Omega= d\,\Phi 
+\mu_2\circ  (\Phi\otimes d\,\Lambda^{\reg} )\circ \Delta_{\star,1}-
\mu_2 \circ( d\,\Phi\otimes \Phi +d\,\Lambda^{\reg}\otimes \Phi) \circ\Delta_{1,\star} $$
 as the difference of $(\ref{inpfotherside})$ and  $(\ref{inpfoneside})$   is  $\mu_3\circ(\Omega\otimes\Lambda^{\reg}
\otimes C)\circ \Delta^{(3)}$.
We will prove that
\begin{equation}\label{vanishingtocheck}
  d\,\Phi =
\mu_2\circ \big[ - (\Phi\otimes d\,\Lambda^{\reg})\circ \Delta_{\star,1}+
 (d\,\Phi\otimes \Phi)\circ\Delta_{1,\star}  + (d\,\Lambda^{\reg}\otimes \Phi) \circ\Delta_{1,\star} \big]\end{equation}
applied to $\xi$, where  $\xi=(t_1:t_2:\ldots :t_n;\beta_1,\beta_2,\ldots,\beta_n)$.
\vspace{0.1in}

\emph{Case when $\xi$ is essential}. Then $J=\{1,\ldots, n-1\}$. 
It follows from example \ref{exD1star} that the only quotient sequences arising from $\Delta_{1,\star}$ are of the form $(t_i: t_n)$ for some $i<n$, and 
are therefore not regular. Thus the first term in the right-hand side of $(\ref{vanishingtocheck})$  vanishes.
The only contributions to the second summand come from the string $(t_{n-1},t_n)$; only the
 strings $(t_{i-1},t_i)$ and $(t_{i},t_{i-1})$, for $1<i<n$, contribute to the last summand.  
 For such a string $a$, let $\xi/a$ denote its quotient.
 From the definitions:
 \begin{eqnarray}
 (d\,\Lambda^{\reg} \otimes \Phi)  (a\otimes \xi/a) &=& \beta_{1,2,\ldots, i-1}\,  d\, \Li  ({t_{i-1} t_i^{-1} }) \, \Phi(\xi)  \quad \hbox{ if } a= (t_{i-1},t_i)\nonumber \\
 (d\,\Lambda^{\reg} \otimes \Phi)  (a\otimes \xi/a) &=& \beta_{1,2,\ldots, i-1}\,  d\, \Li  ({t_{i} t_{i-1}^{-1} }) \, \Phi(\xi)   \quad \hbox{ if } a= (t_{i},t_{i-1}) \nonumber  
 \end{eqnarray}
 Using the fact that 
 $  d\, \Li  ({t_{i-1} t_i^{-1} }) - d\, \Li  ({t_{i} t_{i-1}^{-1} }) = d \, \log (t_i) - d\log(t_{i-1})$
 a straightforward calculation shows that both sides of $(\ref{vanishingtocheck})$ agree on $\xi$.
 
\vspace{0.1in}

\emph{Case when $\xi$ is non-essential}. Either $ n \in J$ or  some $i<n$ is not in $J$. Suppose first that $n\in J$.
The first term  of $(\ref{vanishingtocheck})$ vanishes  as either the argument of $\Phi $ is not essential, or the argument of $\Lambda^{\reg}$ is not regular. 
The second and third summands vanish since the arguments of  $\Phi$ are non-essential. Hereafter, we assume $n\notin J$.

Now suppose that $J^c$ contains at least 3 elements  $J^c\supseteq \{i,j,n\}$. Then the entire right-hand side of 
$(\ref{vanishingtocheck})$ vanishes, since every argument of $\Phi$ is always non-essential.

It only remains to check the equality of $(\ref{vanishingtocheck})$  when $J^c$ consists of two elements $\{k,n\}$ for some $k<n$.
 Consider the second and third terms on the right-hand side of $(\ref{vanishingtocheck})$. The quotient sequences of $(\ref{left1})$ and
 $(\ref{left2})$ are essential only for the strings $(t_k:t_{k+1})$ and $(t_k:t_{k-1})$. These are non-essential so the second factor $d\Phi\otimes \Phi$ vanishes for all possible values of $k$.
 The third factor $d\Lambda^{\reg}\otimes \Phi$ is non-trivial only when $k=n-1$ on the term $(t_{n-1}:t_n)$.
  In fact, in the case $k=n-1$, we have contributions from $\Phi\otimes d\Lambda^{\reg} (\ref{right1})$ and 
 $d\Lambda^{\reg}\otimes \Phi $ applied to $(\ref{left1})$, for $i=n-1$. They cancel.
  
 Thus in all remaining cases $k<n$ only the first term $\Phi\otimes d\Lambda^{\reg}$ of $(\ref{vanishingtocheck})$ 
 can be non-zero.  
 If  $k=1$   then we get  terms in the first summand corresponding to $(\ref{right2})$, $(\ref{right3})$, $(\ref{right4})$. The cancellation of these terms follows from the equality
\begin{equation}\label{kid1}\sum_{i=1}^{n-1} (-1)^{i-1} a_{[i,2]}^{-1} b_{[i+1,n-1]}^{-1}=0\end{equation}
where $a_{[i,j]}=\beta_i(\beta_i+\beta_{i-1})\ldots (\beta_i+\ldots +\beta_j)$,  if $i\geq j$ and is equal to $1$ otherwise, and 
$b_{[i,j]}=\beta_i(\beta_i+\beta_{i+1})\ldots (\beta_i+\ldots +\beta_{j})$, if $i\leq j$ and is equal to $1$ otherwise.  The general case $1<k< n-1$ is similar, and equivalent to
\begin{equation}\label{kid2} \sum_{1<k<i<n-1} (-1)^{k-i-1} b_{[1,k-1]}^{-1} a_{[i,k+1]}^{-1} b_{[i+1,n-1]}^{-1}  =0\end{equation}
Both identities $(\ref{kid1})$ and $(\ref{kid2})$ are easily checked by taking the residues along the divisors $\beta_{i}+\ldots +\beta_{n-1}=0$  and induction.

In conclusion, we have proved that  $\Omega$ vanishes,  and hence, by induction hypothesis, the differential of the difference between both sides of  $(\ref{limitofDebviacoprod}$) is $ O(T^{\varepsilon-1})    $. Thus the difference between both sides is a constant plus $ O(T^{\varepsilon-1})    $, which proves the theorem.
\end{proof}

\subsection{Asymptotics in depths 1 and  2}  \label{exampledepth2limits}  
In depth 1 we have,  $$\Lambda(t;\beta) \sim  \beta^{-1} t^{-\beta} + C(\beta) \quad \hbox{  as } t \rightarrow \infty  \ , $$
where $C(\beta) =  2i\pi \,(1-\e(\beta))^{-1}$ (see lemma $\ref{lemmacomputeconst1}$ below).
 
  Let $\beta_{12}=\beta_1+\beta_2$, $t_{12}=t_{21}^{-1}=t_1t_2^{-1}$. In depth two,  the coproduct $(\ref{depth2coprod})$
 yields
  \begin{eqnarray}  \label{depth2limits}
 \Lambda(t_1,t_2;\beta_1,\beta_2) \!\! &\sim & {t_{12}^{-\beta_1}\over \beta_1} \Lambda(t_2;\beta_{12}) + \Lambda(t_2;\beta_2) \, C(\beta_1) +C_1 \quad \hbox{  as } t_1 \rightarrow \infty  \nonumber \\ 
&\sim & {t_2^{-\beta_2}\over \beta_2}\big[  \Lambda(t_1;\beta_1) - t_1^{\beta_2} \Lambda(t_1;\beta_{12})\big] +C_2  \quad  \hbox{  as } t_2 \rightarrow \infty \nonumber \\ 
&\sim & {t_1^{-\beta_1}t_2^{-\beta_2} \over \beta_1 \beta_{12}}+ \big[ \Lambda (t_{12};\beta_1) -\Lambda(t_{21};\beta_2) \big] C(\beta_{12}) \\
&& \qquad \qquad  +  {t_2^{-\beta_2}\over \beta_2}C(\beta_1)+ C_{12} \quad \qquad \hbox{as } t_1,t_2 \rightarrow \infty    \nonumber
\end{eqnarray}
where $C_1,C_2,C_{12}$ are constant power series in $\beta_1,\beta_2$ to be determined. 
 The constant $C_1$ is clearly zero as can be seen by letting $t_2\rightarrow 0$  in the first equation of 
$(\ref{depth2limits})$. The same holds for $C_2$ (let $t_1\rightarrow 0$). The constant $C_{12}$ can be computed as follows.

\subsubsection{Limit at $D^2\cap D^{12}$}  From the second line  of $(\ref{depth2limits})$, we deduce that
\begin{equation}\label{limitD2,12}
\hbox{constant part of }\lim_{t_1\rightarrow \infty} \lim_{t_2\rightarrow \infty}  \Lambda(t_1,t_2;\beta_1,\beta_2) = 0\ ,
\end{equation} 
since $C_2=0$. 
Now  let $t_1,t_2\rightarrow \infty$ and then let $t_2/t_1\rightarrow \infty$. The third line gives a constant contribution
$ C_{12}(\beta_1,\beta_2) -{C(\beta_2) C(\beta_{12})}$. It follows that 
\begin{equation}\label{Ctwo}
C_{12}(\beta_1,\beta_2) =   C(\beta_2)C(\beta_{12})  \ .\end{equation}

\subsubsection{Limit at  $D^1\cap D^{12}$}
 From the first line  of $(\ref{depth2limits})$, we deduce that
\begin{equation}\label{limitD1,12}
\hbox{constant term of }\lim_{t_2\rightarrow \infty} \lim_{t_1\rightarrow \infty}  \Lambda(t_1,t_2;\beta_1,\beta_2) = C(\beta_1)C(\beta_2) \ .
\end{equation} 
Now let $t_1,t_2\rightarrow \infty$ and then let $t_1/t_2\rightarrow \infty$. The third line gives the constant contribution
${C(\beta_1) C(\beta_{12})} + C_{12}(\beta_1,\beta_2)$, which yields a second equation for $C_{12}(\beta_1,\beta_2)$. Note however, that the two limit computations
are for different branches (see caveat \ref{caveat}), and differ by the monodromy of the third line of $(\ref{depth2limits})$ around the point $t_1=t_2$ on $D^{12}$. 
The monodromy of $\Lambda(t;\beta)$  (resp. $\Lambda(t^{-1};\beta)$) is $\pi i$ (resp. $-\pi i$) around a positive upper semi-circle from $1^-$ to $1^+$.
 Therefore, by equating the two different 
formulae for $C_{12}(\beta_1,\beta_2)$ gives rise to an associator, or pentagon,  equation:
\begin{equation}
\big( C(\beta_1)+C(\beta_2)- 2i\pi\big) C(\beta_{12}) = C(\beta_1)C(\beta_2)
\end{equation}
which is indeed satisfied by $C(\beta)={2\pi i \over  1-\e(\beta) }$.

\subsection{Rationality of the constants} \label{sectratconst}
 The argument above generalizes:

\begin{prop} Let $v\in U_n$ be a vertex of $U_n$, i.e., $v$ is an intersection of boundary divisors $D^I$ of dimension $0$.
Then there is a branch of the Debye polylogarithm $\Lambda(t_1,\ldots, t_n;\beta_1,\ldots, \beta_n)$ in a neighbourhood of $v$ which is locally of the form
$$\sum_{I=(i_1,\ldots, i_n)} f_I(s_1,\ldots, s_n) \log^{i_1} s_1 \ldots \log^{i_n} s_n$$
where $f_I(0,\ldots, 0)\in \Q[ \pi i ]$, and $s_1,\ldots, s_n$ are local sector coordinates at $v=(0,\ldots,0)$.
\end{prop} 
\begin{proof}
It suffices to show that the constant coefficients lie in $\Q[\pi i ]$.   But this follows from a standard associator argument: the 1-skeleton of the polytope $\mathcal{C}_n\subset U_n(\R)$ is connected, and the restriction  of $\Lambda$ to a one-dimensional stratum is a depth one   Debye polyogarithm, whose limiting values at infinity have the desired property, by $(\ref{Cconst})$. By analytic continuation around the one-dimensional edges of  $\mathcal{C}_n$, we deduce that the constants at $v$ are expressible as sums and  products of 
the coefficients of $C(\beta)$. \end{proof}

\section{Examples in depths 1,2}\label{sectexamples123}

\subsection{Depth 1: the classical elliptic polylogarithms} Let $q=\e(\tau)$ with $\Image(\tau)>0$, and let $z=\e(\xi)$ with $\xi$ in the fundamental domain $D$ (\S \ref{sectUniformization}). 
Consider the multivalued generating series of polylogarithms of depth one:
$$L(z;\beta) = \sum_{n\geq 1 } \Li_n(z) \beta^{n-1}\ ,$$
which we wish to average over the spiral $(0,\infty)\cong q^{\R} z$ in the universal covering space of $\Mod_{0,4}(\C)$.
The calculations are simplified if one considers the Debye generating series
$\Lambda(z; \beta) = z^{-\beta} L(z;\beta)$. Since $d\, \Li_n(z) = z^{-1} \Li_{n-1}(z) dz$ for $n\geq 2$, we have:
\begin{equation} \label{depth1diffeq}
d\Lambda(z;\beta) =      z^{-\beta}d\Li_1(z)\ .\end{equation}
Note that $\Lambda(z;\beta)$ vanishes at $z=0$, and so by theorem $\ref{thmconv}$ the series
\begin{equation}\label{depth1av}
\EE(z;u,\beta)=\sum_{n\in \Z} u^n\Lambda(q^n z; \beta)\end{equation} 
converges absolutely for $1<u<|q|^{-1}$, and may have poles at $u=1$, which are given by the asymptotics of $\Lambda(z;\beta)$ at $z=\infty$. 
Since $d\,\Li(z)$ is asymptotically $-d\log(z)$ at infinity, we deduce  from $(\ref{depth1diffeq})$ that there is some constant $C(\beta)$ such that:
\begin{equation} \label{1Dasymp}
 \Lambda(z;\beta) \sim \beta^{-1}z^{-\beta}+C(\beta)\ .
\end{equation}

\begin{lem} \label{lemmacomputeconst1} The constant  at infinity is given by 
\begin{equation} \label{Cconst}
C(\beta)=-\beta^{-1}+ i\pi + \sum_{n\geq 1} 2\,\zeta(2n) \beta^{2n-1} =   { 2i \pi  \over 1- \e(\beta)}\ . \end{equation}
\end{lem} 
\begin{proof} 
The following functional equation  follows from $(\ref{depth1diffeq})$ and differentiating:
\begin{equation}\label{depth1funceq} \Lambda(z;\beta) + \Lambda(z^{-1};-\beta) = \beta^{-1}z^{-\beta}+C(\beta)
\end{equation}
Evaluating at $z=1$ gives the  expression for $C(\beta)$, since $\Li_n(1)=\zeta(n)$, for $n\geq 2$.
\end{proof} 
The corresponding constants in all higher dimensions are explicitly computable from $C(\beta)$.
It follows from $(\ref{1Dasymp})$ that the singular part  of $\EE(z;u,\beta)$ comes from:
$${1\over \beta} \sum_{n< 0} u^n \big((q^nz)^{-\beta} + C(\beta)\big) = {z^{-\beta}\over \beta(q^{-\beta} u-1)} + {C(\beta)\over u-1} = {\e(-\beta \xi) \over \beta(\e( \gamma)-1) }+{C(\beta) \over \e(\alpha)-1} \ . $$ 
where $u=\e(\alpha)$,  and $\gamma=\alpha-\beta\tau$. The second expression defines a Taylor series in $\beta$ with coefficients in $\Q[u,(1-u)^{-1},\log q, i \pi]$. 
Thus the singular part of $\EE(\xi;\alpha,\beta)$ is 
\begin{equation}\label{Esing} 
\EE^{\sing}(\xi;\alpha;\beta)= {\e(-\xi \beta)\over\beta  \gamma} + { C(\beta)\over \alpha} \  .\end{equation} 
In conclusion, the regularized generating series for the classical elliptic polylog is:
\begin{equation}\label{Eregdef}
\EE^{\reg}(\xi;\alpha;\beta)=\sum_{n\in \Z} \e(\alpha n)\Lambda(\e(\xi+n\tau), \beta) - {\e(-\xi \beta)\over  \beta \gamma} - { C(\beta)\over \alpha}\ ,
 \end{equation}
which admits a Taylor expansion in $\alpha, \beta$ at the origin. Thus we write
 $$ \EE^{\reg}(\xi;\alpha;\beta)= \sum_{m,n\geq 0} \Lambda^E_{m,n}(\xi;  \tau) \alpha^m \beta^n\ ,$$
 where $\Lambda^E_{m,n}(\xi; \tau)$ are the classical elliptic polylogarithms of \cite{Andrey}, and equal to the functions denoted $(-1)^n\Lambda_{m,n}(\xi;\tau)$ in \emph{loc. cit.}, Definition 2.1.

\begin{rem} \label{rem1dimreg} In order to  retrieve the explicit formula of \cite{Andrey}, Definition 2.1, one can write $\Lambda^E_{m,n}(\xi;  \tau)$ as an average of certain
modified (and regularized) Debye polylogarithms. For this, one simply replaces the term $\alpha^{-1}$   in $(\ref{Esing})$ by  the expression
\begin{equation} \label{poletoexp}
{1\over \alpha} =P(\alpha) -  \sum_{m>0}  \e( m \alpha)
\end{equation}
where $P$ is a power series   whose coefficients are related to Bernoulli numbers.  Replacing  the term in  $\gamma^{-1}$ 
by a similar expression to $(\ref{poletoexp})$  leads to the required result. 
\end{rem}

\subsection{Depth 2: the double elliptic polylogarithms} \label{sectDepth2}
Consider the generating series of depth two Debye multiple polylogarithms:
$$\Lambda(t_1,t_2; \beta_1,\beta_2) = t_1^{-\beta_1} t_2^{-\beta_2} \sum_{m_1,m_2\geq 1} I_{m_1,m_2}(t_1,t_2) \beta_1^{m_1-1}\beta_2^{m_2-1}\ .$$
The generating series of elliptic multiple polylogarithms is:
\begin{equation}\label{depth2av}
\EE_2(\xi_1,\xi_2; u_1,u_2,\beta_1,\beta_2)=\sum_{m_1,m_2\in \Z} u_1^{m_1} u_2^{m_2} \Lambda(q^{m_1}t_1,  q^{m_2}t_2; \beta_1, \beta_2)\end{equation}
which converges absolutely for $1<u_1,u_2<|q|^{-1}$ by theorem \ref{thmconv}, and has poles along 
$u_1=1,u_2=1, u_1u_2=1$ corresponding to logarithmic singularities of $\Lambda(t_1,t_2)$ along $D^1,D^2,D^{12}$.  Let $\gamma_i = \alpha_i-\tau\beta_i$, where $\e(\alpha_i)=u_i$, and $q=\e(\tau)$. 

\begin{lem} \label{lemsingindepth2} The singular part of  $\EE_2(\xi_1,\xi_2;\alpha_1,\alpha_2,\beta_1,\beta_2)$ is  
$\EE_2^{\sing}= \EE^{\sing(1)}_2+ \EE^{\sing(2)}_2,$
where $\EE_2^{\sing(i)}$ comes from singularities along divisors of codimension $i$. We have
 $$\EE_2^{\sing(2)}= { \e^{- \beta_1 \xi_1 - \beta_2 \xi_2 } \over  \beta_1\beta_{12}\gamma_1\gamma_2}+  {\e^{-  \beta_1 \xi_{12}}  \, C(\beta_{12}) \over \beta_1 \gamma_1\, \alpha_{12}} - {\e^{- \beta_2 \xi_{21} }\, C(\beta_{12}) \over \beta_2 \gamma_2\, \alpha_{12}}      + {\e^{-\beta_2\xi_2} \,C(\beta_1)\over \beta_2 \gamma_2\, \alpha_1 } + {C_{1,12}\over \alpha_1(\alpha_1+\alpha_2) }$$
   where $C_{1,12}=C(\beta_1)C(\beta_2)$ and  $C(\beta)$ is  the power series defined by  $(\ref{Cconst})$, and 
$$\EE_2^{\sing(1)}= {R_1 \over  \beta_1 \gamma_1} +  {R_2 \over  \beta_2\gamma_2} 
+{A_1\over \alpha_1}+  {A_{12} \over \alpha_{12}}  \ , $$
where 
\begin{eqnarray}
A_{1} & = &  \EE^{\reg}(\xi_2;\alpha_{12},\beta_2)\, C(\beta_1)\big|_{\alpha_1=0} \nonumberÊ\\
A_{12} & = & \big( \EE^{\reg}(\xi_{12};\alpha_1,\beta_1)-  \EE^{\reg}(\xi_{21};-\alpha_1,\beta_2)\big)\, C( \beta_{12})  \big|_{\alpha_{12}=0} \nonumber \\
R_1 &=&  \e^{-\beta_1\xi_{12}} \,\EE^{\reg}(\xi_2;\alpha_{12};\beta_{12}) \big|_{\gamma_{1}=0} \nonumber \\
R_2 &=& \e^{-\beta_2\xi_2}\,\EE^{\reg}(\xi_1;\alpha_1,\beta_1) -  \e^{-\beta_2\xi_{21}} \, \EE^{\reg}(\xi_1;\alpha_{12},\beta_{12}) \big|_{\gamma_{2}=0}  \nonumber \ .
\end{eqnarray}
 Here,   $\e^a=\e(a)$, $\alpha_{12}=\alpha_1+\alpha_2$, $\beta_{12}=\beta_1+\beta_2$, and 
$\xi_{12}=-\xi_{21}=\xi_1-\xi_2$.
\end{lem}
\begin{proof}
We can compute the singularities of $\EE_2$ from the differential equation
\begin{eqnarray} \label{EEdepth2diff}
\qquad d\EE_2(\xi_1,\xi_2; \alpha_1,\alpha_2, \beta_1,\beta_2) & = &  d\EE_1(\xi_1-\xi_2;\alpha_1,\beta_1) \,\EE_1(\xi_2;\alpha_1+\alpha_2,\beta_1+\beta_2)  \\
&- &   d\EE_1(\xi_2-\xi_1;\alpha_2,\beta_2)\, \EE_1(\xi_1;\alpha_1+\alpha_2,\beta_1+\beta_2)  \nonumber \\
& +& d\EE_1(\xi_2;\alpha_2,\beta_2) \, \EE_1(\xi_1;\alpha_1,\beta_1) \nonumber 
\end{eqnarray}
The fact that $d^2 \EE_2(\xi_1,\xi_2,\alpha_1,\alpha_2;\beta_1,\beta_2) = 0$ reduces to the Fay identity.
Substituting  the singular parts $\EE^{\sing}$ (given by  $(\ref{Esing})$)
  into $(\ref{EEdepth2diff})$ yields the pole structure of $\EE_2$. Using $(fg)^{\sing}=(f g^{\sing} + f^{\sing} g)- f^{\sing} g^{\sing}$, the differential equation for $\EE_1$ $(\ref{dEequation})$ and the additivity of the exponential function, we have $ \EE_2^{\sing}= \EE^{\sing(1)}_2-  \EE^{\sing(2)}_2$, where
   $$d\EE^{\sing(2)}= d\Big({ \e^{-\xi_1 \beta_1 -\xi_2\beta_2} \over  \beta_1\beta_{12}\gamma_1\gamma_2}+  {\e^{-\xi_{12} \beta_1}  \, C(\beta_{12}) \over  \beta_1 \gamma_1\, \alpha_{12}} - {\e^{-\xi_{21} \beta_2}\, C(\beta_{12}) \over \beta_2 \gamma_2\, \alpha_{12}}      + {\e^{-\beta_2\xi_2} \,C(\beta_1)\over \beta_2\gamma_2\, \alpha_1 }\Big)$$
  $$d\EE^{\sing(1)}_2 =  \Big[ d\EE_1(\xi_{12};\alpha_1,\beta_1)-  d\EE_1(\xi_{21};\alpha_2,\beta_2)\Big]{ C(\beta_{12})\over \alpha_{12}}  + 
 d\EE_1(\xi_2,\alpha_2,\beta_2){ C(\beta_1)\over \alpha_1}  $$
 $$ + {\kappa_{12} |_{\gamma_1+\gamma_2=0} \over \beta_{12} \gamma_{12}} + {\kappa_1 |_{\gamma_1=0} \over \beta_1 \gamma_1} + {\kappa_2|_{\gamma_2=0} \over \beta_2\gamma_2} $$
where $\gamma_{12}=\gamma_1+\gamma_2$ and 
\begin{eqnarray}
\kappa_{12} & =&  \big( F(\xi_{12},\gamma_1) + F(\xi_{21},\gamma_2)\big)\e^{-\beta_1\xi_1-\beta_2\xi_2} \, d\xi_{12}\nonumber \\
\kappa_{1} & =&  \e^{-\beta_1\xi_1-\beta_2\xi_2} F(\xi_2,\gamma_2) \, d\xi_2  -\beta_1 \,  \e^{-\beta_1\xi_{12}}  \EE_1(\xi_2; \alpha_{12}; \beta_{12}) \,  d\xi_{12}  \nonumber \\
\kappa_{2} & =& \beta_2 \,\e^{-\beta_2\xi_{21}} \EE_1(\xi_1;\alpha_{12}, \beta_{12}) \,  d\xi_{12}   - \beta_2 \, \e^{-\beta_2\xi_2}  \EE_1(\xi_1; \alpha_1; \beta_1)  \,  d\xi_2\nonumber 
\end{eqnarray}
It follows from the expansion $(iii)$ of Proposition-Definition \ref{Fproperties} plus the fact that $E_1(\xi,\tau)$ is an odd function of $\xi$ that $\kappa_{12}|_{\gamma_1+\gamma_2}=0$, and therefore does not contribute. There is an obvious  of $d\,\EE_2^{\sing(2)}$. By  integrating, we deduce that:
$$\EE_2^{\sing(1)}= {R_1 \over  \beta_1\gamma_1} +  {R_2 \over \beta_2 \gamma_2} 
+{A_1\over \alpha_1}+  {A_{12} \over \alpha_1+\alpha_2}  \ ,$$
since  $d\EE^{\reg}$ and $d\EE$ are equal up to higher order poles. 
It remains to add the constants of integration. Since these give at most simple poles in the $\alpha$'s and correspond to the limiting values in the corners, 
they contribute
$$ {C_{1,12}\over \alpha_1(\alpha_1+\alpha_2) } + {C_{2,12}  \over \alpha_2(\alpha_1+\alpha_2)} \ , 
$$
where $C_{1,12}$ and $C_{2,12}$ are  the constant part of  the asymptotic of the Debye double polylogarithm near $D^1\cap D^{12}$ and $D^2\cap D^{12}$ given by  $(\ref{limitD2,12})$ and $(\ref{limitD1,12})$.
\end{proof}
As in the depth 1 case, we therefore define the depth 2 multiple elliptic polylogarithms to be the coefficients in the Taylor expansion:
$$\EE_2 - \EE_2^{\sing}= \sum_{m_i,n_j\geq 0} \Lambda^E_{(m_1,m_2),(n_1,n_2)}(\xi_1,\xi_2;\tau) \,\alpha_1^{m_1}\alpha_2^{m_2}\beta_1^{n_1}\beta_2^{n_2}\ ,  $$ 
where $\EE_2$ is given by  $(\ref{depth2av})$, and $\EE^{\sing}_2$ by the previous lemma.

\subsection{Singular part computed from the coproduct}
Another way to arrive  at lemma $\ref{lemsingindepth2}$  is from the computation of the asymptotic of the depth 2 Debye polylogarithms
given in example $\ref{exampledepth2limits}$.  In general, we have:

\begin{cor} The singular structure of the elliptic Debye polylogarithm $\EE_n$ is obtained by averaging the asymptotic of the ordinary Debye polylogarithms. In particular,
it is explicitly computable from the coproduct  $(\ref{coproddef})$  and the constant terms $C$.
\end{cor}
In fact, the asymptotic of the Debye polylogarithms in the neighbourhood of boundary divisors of all codimensions can be  computed from the coproduct in two different ways. The first, via theorem \ref{thmcoprodasymp}, is to compute the asymptotic in the neighbourhood of codimension 1 divisors, and by induction apply the theorem to the arguments of $\Lambda$ to obtain the asymptotic in all codimensions. The other, is directly  from  formula  
$(\ref{limitofDebviacoprod})$ which immediately gives the asymptotic in all codimensions, provided that the definition of `essential', `regular', and the constants $C$ are modified accordingly.

\section{Iterated integrals on $\E^{\times}$} \label{sect10}
We compute the integrable words corresponding to the elliptic Debye polylogarithms viewed as functions of one variable, and compare  with the bar construction. From this we deduce
that all iterated integrals on $\E^{\times}$ are obtained by averaging.

\subsection{Projective coordinates and degeneration} Let $\G_m=\Pro^1\backslash \{0,\infty\}$, let $n\geq 1$, and write  $\Delta\subset \G_m^{n+1}$
for the union of all the diagonals. There is an isomorphism
\begin{equation}\label{GmtoMon} 
(\G_m^{n+1}\backslash \Delta )/ \G_m\overset{\sim}{\To} \Mod_{0,n+3}\ .
\end{equation}
If we write homogeneous coordinates on the left-hand side as $(t_1:\ldots: t_{n+1})$, then the isomorphism is given by 
$(t_1:\ldots :t_{n+1}) \mapsto (t_1t_{n+1}^{-1},\ldots, t_nt_{n+1}^{-1})$.
Let $\beta_1,\ldots, \beta_{n+1}$ be formal parameters satisfying $\beta_1+\ldots + \beta_{n+1}=0$. Recall  that we  set:
\begin{equation} \label{homLambda} \Lambda(t_1:\ldots: t_{n+1};\beta_1,\ldots, \beta_{n+1}) =\Lambda(t_1t_{n+1}^{-1},\ldots, t_nt_{n+1}^{-1};\beta_1,\ldots, \beta_n)
\end{equation}Ê
Forgetting the marked point $t_{n+1}$ gives rise to a fibration 
\begin{eqnarray}
 \Mod_{0,n+3} &\rightarrow & \Mod_{0,n+2} \\
(t_1:\ldots :t_{n+1}) & \mapsto  & (t_1:\ldots: t_{n}) \nonumber
\end{eqnarray} 
whose fiber over the point $(t_1:\ldots: t_n)$ of $\Mod_{0,n}$ is isomorphic to  $\G_m\backslash \{t_1,\ldots, t_n\}$.
The functions $(\ref{homLambda})$, when restricted to each fiber, have a particularly simple description.

\begin{lem}  \label{lemnomonodromy} For constant $t_1,\ldots, t_n$  (i.e., $dt_i=0$ for $i\leq n$), we have
Ê$$d\Lambda(t_1:\ldots: t_{n+1};\beta_1,\ldots, \beta_{n+1}) = d\Lambda(t_n:t_{n+1};\beta_n,-\beta_n) \quad \times   \qquad \qquad \qquad \qquad 
$$\begin{equation}  \label{tlastdiffeq}\qquad \qquad \qquad \qquad \qquad   \Lambda(t_1:\ldots: t_{n-1}:t_{n+1};\beta_1,\ldots, \beta_{n-1}, \beta_n+\beta_{n+1})
\end{equation}
For $1\leq i< j \leq n$, let $\Mo_{ij}$ denote analytic continuation along  a small loop around  $t_i=t_j$. Then  the functions $(\ref{homLambda})$ are single-valued
around $t_i=t_j$ for $i,j\geq 1$:
$$(\Mo_{ij} -id)\, \Lambda(t_1:\ldots: t_{n+1};\beta_1,\ldots, \beta_{n+1}) = 0  \quad \hbox{ if } \quad  i,j\leq n \ .$$
\end{lem}

\begin{proof} The differential equation  $(\ref{tlastdiffeq})$ follows from the differential equation for $\Lambda$. To prove the singlevaluedness,  note that 
$\Mo_{ij}$ commutes with ${\partial /\partial t_{n+1}}$ for $i, j \leq n$. From $(\ref{tlastdiffeq})$ the result follows by induction plus the fact that 
$ \Lambda(t_1:\ldots: t_{n+1})$ vanishes as $t_{n+1}$ tends to $\infty$. Alternatively, via $(\ref{Idef})$,   the Taylor expansion   $(\ref{Lindef})$ of the function
$I_{m_1,\ldots,m_n}(t_1,\ldots, t_n)$ at the origin shows that it has trivial monodromy around $t_i=t_j$ for $i<j\leq n$. Thus the same is true of 
$\Lambda(t_1:\ldots: t_{n+1})$ by definition. 
\end{proof}
The elliptic analogue of $(\ref{GmtoMon})$ is as follows.  Letting $\Delta \subset \E^{n+1}$ denote the union of all diagonals, and using the notation $(\xi_1:\ldots: \xi_{n+1})$ for 
coordinates on $\E^n/\E$, we have: 
\begin{eqnarray}
(\E^{n+1}\backslash \Delta)/\E & \overset{\sim}{\To} &  \En \\
(\xi_1:\ldots:\xi_{n+1}) & \mapsto & (\xi_1-\xi_{n+1},\ldots, \xi_n-\xi_{n+1}) \ .\nonumber
\end{eqnarray} 
Again,  forgetting the marked point $\xi_{n+1}$ gives rise to a fibration 
\begin{eqnarray} \label{newEfibration}
 \En &\rightarrow & \E^{(n-1)} \\
(\xi_1:\ldots :\xi_{n+1}) & \mapsto  & (\xi_1:\ldots: \xi_{n}) \nonumber
\end{eqnarray} 
whose fiber over the point  $ (\xi_1:\ldots: \xi_{n}) $ is isomorphic to   $\E \backslash \{\xi_1,\ldots, \xi_n\}$.

\begin{defn} Define the elliptic Debye hyperlogarithm to be the generating series:
 $$G_n(\xi ; \xi_1,\ldots, \xi_n, \alpha_1,\ldots, \alpha_n, \beta_1,\ldots, \beta_n) =\EE_n(\xi_1-\xi ,\ldots ,\xi_n-\xi;\alpha_1,\ldots, \alpha_n, \beta_1,\ldots, \beta_n) $$ 
viewed  as a multivalued function of the single variable $\xi\in \E \backslash \{\xi_1,\ldots, \xi_n\}$.
\end{defn} 
 It follows from   equation $(\ref{tlastdiffeq})$ that, for constant  $\xi_1,\ldots, \xi_n$ (i.e., $d\xi_i=0$),
\begin{equation} \label{Gdiff} 
d G_n(\xi;  \xi_1,\ldots, \xi_n; \alpha_1,\ldots, \alpha_n,\beta_1,\ldots, \beta_n  ) = d \EE_1(\xi_n-\xi; \alpha_n, \beta_n) \end{equation}
$$\qquad\qquad \qquad \qquad \times \quad  G_{n-1} (\xi;  \xi_1,\ldots, \xi_{n-1};\alpha_1,\ldots, \alpha_{n-1},\beta_1,\ldots, \beta_{n-1} ) $$
\begin{rem} There is no obvious way to determine the constant of integration in $(\ref{Gdiff})$ since the averaging process introduces constant terms related to Bernoulli numbers.
On the other hand, one natural normalization  for an iterated integral on $\En$ is for it to vanish along a tangential base point at $1$ on the Tate curve at infinity, which  is not the case for the averaged functions $\EE_n$. Thus, the comparison between the averaged functions $\EE_n$  and such iterated integrals must take into account the constants.
\end{rem}
In order to circumvent this issue, let $\basepnt \in \E\backslash \{\xi_1,\ldots, \xi_n\}$ be any point. Consider the $n+1$ square  matrix $M_{ij}$ with $1$'s along the diagonal, $0$'s below the diagonal, and 
$$M_{ij}= G_{j-i}(\xi;\xi_i,\ldots, \xi_{j-1};\alpha_i, \ldots, \alpha_{j-1}, \beta_i,\ldots, \beta_{j-1})\quad \hbox{ for } \quad  1\leq i<j\leq n+1\ .$$
Denote this matrix by $M_\xi$, viewed as a function of $\xi$.  The differential equation     $(\ref{Gdiff})$ translates into an equation of the form
$d M_{\xi}= M_{\xi} \Omega$ for some square matrix $\Omega$ of $1$-forms. It follows that $M^{-1}_\basepnt M_\xi$ satisfies the same equation. Therefore,
 we define  
$$G_i^{\basepnt}(\xi;\xi_{1},\ldots, \xi_i;\alpha_1, \ldots, \alpha_{i},\beta_1,\ldots, \beta_i )= ( M^{-1}_\basepnt M_\xi)_{1,i+1} \quad \hbox{ for } 0\leq i\leq n$$
 These functions
satisfy the differential equation $(\ref{Gdiff})$ and vanish at $\xi=\basepnt$ if $i\geq 1$.

\subsection{Reminders on iterated integrals} 
Given a smooth manifold $M$ over $\R$, a smooth path $\gamma:[0,1]\rightarrow M$, and smooth 
one-forms $\omega_1,\ldots, \omega_n$ on $M$, the iterated integral  of $\omega_1,\ldots, \omega_n$ along $\gamma$ is defined to be $1$ if $n=0$, and  for $n\geq 1$:
$$\int_{\gamma} \omega_1 \ldots \omega_n =\int_{0\leq t_n\leq \ldots \leq t_1\leq 1} \gamma^{*}(\omega_1)(t_1) \ldots   \gamma^{*}(\omega_n)(t_n)\ .$$
 Let $A$ be the $C^\infty$ de Rham complex on $M$, and let $V(A)$ denote the zeroth cohomology of the reduced bar complex of $A$. Choose a basepoint $\basepnt\in M$, and let $I_M$ denote the differential $\R$-algebra of multivalued holomorphic functions  on $M$ with global unipotent monodromy. A   theorem due to Chen \cite{Ch1} states that the map $V(A) \rightarrow V(M)$ given by 
\begin{eqnarray} \label{Chenmap}
\sum_{I=(i_1,\ldots, i_n)}  c_I {[}\omega_{i_1} |\ldots |\omega_{i_n}] & \mapsto & \sum_I c_I  \int_{\gamma} \omega_{i_1}\ldots \omega_{i_n}  
\end{eqnarray}
is an isomorphism, where $\gamma$ is any path from $\basepnt$ to $z$, and the iterated integrals are viewed as functions of the endpoint $z$. In particular, they only depend on the homotopy class of $\gamma$ relative to its endpoints. 
 The differential with respect to $z$ is
\begin{equation} \label{itintdiffform} {\partial \over \partial z}  \sum_I c_I  \int_{\gamma} \omega_{i_1}\ldots \omega_{i_n}=  \sum_I c_I  \, \omega_{i_1} \wedge \int_{\gamma} \omega_{i_2}\ldots \omega_{i_n} \ .\end{equation}
By successive differentiation, and using formula  $(\ref{itintdiffform})$, we can reconstruct a bar element in $V(A)$  which corresponds via   $(\ref{Chenmap})$ to any given  function in $I(M)$.  We shall apply this in the following situation. Suppose that $X\hookrightarrow A$ is  a  connected $\Q$-model for $A$, so we have an isomorphism $V(X)\otimes_{\Q} \R \cong V(A)$. Denote the image of the map $V(X)$ in $I(M)$  by $I(M)_{\Q}$. It defines a $\Q$-structure on the algebra $I(M)$.   If $F\in I(M)_{\Q}$, its bar element in $V(X)$ will be a unique element of $T(X^1)$ by   $(\ref{VofAdef})$  (since $X$ is connected).
  In the sequel, $M$ will be a single elliptic curve with several punctures. We have  a family of functions $\Fo\subset I(M)$, which are the functions obtained by averaging,  and want to show that $\Fo= I(M)_{\Q}$. For this  we shall write the elements of $\Fo$ as elements of  $V(X)\otimes_{\Q}\R$ by computing their differential equations, and check that: firstly they lie in $V(X)$, and secondly, using our explicit description of $V(X)$,  that they span  $V(X)$.

\subsection{Integrable words corresponding to the elliptic polylogarithms}
\begin{defn} \label{defnshuffexp} Define the shuffle exponential to be the formal power series:
$$e_{\sha}(\alpha\nu)=\sum_{n\geq 0}{ \alpha^n \over n!} \nu^{\sha \!  n}  = \sum_{n \geq 0} \alpha^n \underbrace{[\nu|\ldots |\nu]}_n \in T(\Q[\nu])[[\alpha]]\ .$$
The leading  term in the series  ($n=0$) is the empty word. 
Note that if $w_0,\ldots, w_n$ are  symbols and  $x=e_{\sha}(\alpha  w_0)$  then we have
\begin{equation}  \label{expshuffleformula}
x\sha  [w_1|w_2|\ldots |w_n] =  [x |w_1| x |w_2|\ldots | x |w_n|x]
\end{equation}
 as an equality of power series in $\alpha$ with coefficients in $T(\Q w_0  \oplus \Q w_1 \oplus\ldots \oplus \Q w_n)$.

\end{defn}

\begin{lem} \label{lemE1asitint} Let $\basepnt,\xi \in \E$, and let $\alpha, \beta$ be formal parameters, and $\gamma=\alpha-\tau\beta$. Then  we have the following equality of  generating series of multivalued functions:
 \begin{equation}\label{E1asitint} 
 \e(\beta \basepnt+\gamma r_\basepnt)\big(\EE_1(\xi;\alpha,\beta)-\EE_1(\basepnt;\alpha,\beta)\big) = \int_{\basepnt}^{\xi} [\Omega(\xi;\gamma) | e_{\sha}(-\beta'\omega^{(0)}-\gamma\nu)]\  
 \end{equation}
 where we write $\basepnt= s_{\basepnt}+r_{\basepnt}\tau$ (recall that $\xi=s+r\, \tau$).
 \end{lem}
\begin{proof}   Recall that $\omega^{(0)}=d\xi$ and $\nu= 2 \pi i dr$.  Let $\beta'= 2  \pi i \beta $.
 It therefore  follows immediately from definition \ref{defnshuffexp} and the shuffle product for iterated integrals that
$$\e(-\beta\xi-\gamma  r +\beta \basepnt + \gamma r_{\basepnt}) = \int_{\basepnt}^{\xi} e_{\sha}(-\beta'  \omega^{(0)} -\gamma \nu)\ .$$
 We have 
$d\EE_1(\xi;\alpha,\beta)=\e(-\beta \xi) F(\xi;\gamma) d\xi$ and $\Omega(\xi; \gamma)=\e(\gamma r) F(\xi;\gamma)d \xi$.  Hence
$$d\EE_1(\xi;\alpha,\beta) = \e(-\beta \xi-\gamma r) \,\Omega(\xi;\gamma)\ .$$
Combining these two facts, we see that, by $(\ref{itintdiffform})$,
$$d (\hbox{LHS of (\ref{E1asitint})}  ) = \Omega(\xi;\gamma) \int_{\basepnt}^{\xi} e_{\sha}(-\beta'  \omega^{(0)} -\gamma \nu)\ ,$$
so the  differentials of both sides of  $(\ref{E1asitint})$ agree,  and both  vanish at $\xi=\basepnt$. 
\end{proof} 
It is straightforward to verify  that the coefficients in the  right-hand side of $ (\ref{E1asitint})$ are integrable words in $V(X_1)\otimes_{\Q}\C$. 
For $n\geq 1$, let us define
\begin{equation}
H_n(\xi; \underline{\alpha};\underline{\beta}) = \e(\beta_{1,\ldots ,n} \basepnt + \gamma_{1,\ldots, n} r_{\basepnt})\,  G_n^{\basepnt}(\xi;\xi_{1},\ldots, \xi_n;\alpha_1, \ldots, \alpha_{n},\beta_1,\ldots, \beta_n ) 
\end{equation} 
where  we recall that $\beta_{1,\ldots, n}= \beta_1+\ldots +\beta_n$, and likewise for $\gamma$. By construction, the 
coefficients of $H$ are combinations of elliptic multiple polylogs.

Recall that $X_n$ is our rational model for the de Rham complex of $\En$, and $a\mapsto \overline{a}: X_n\rightarrow X_{F_n}$
is the restriction to the fiber. For any $a_1,\ldots, a_k\in X_n$, let us write
 $$\overline{[a_1|\ldots | a_k]} ={[\overline{a}_1| \ldots | \overline{a}_k]} \in X_{F_n}^{\otimes k}\ ,$$
and extend this definition in the obvious way for formal power series in $T(X_n)$.

\begin{prop} \label{propHasitint} The generating series of functions $H_n$ is the iterated integral:
 \begin{equation}\label{Hasitint} 
 H_n(\xi;\underline{\alpha}, \underline{\beta})  = \int_{\basepnt}^{\xi}\overline{W}_n(\xi)\ , 
 \end{equation}
 on the fiber $\E \backslash \{\xi_1,\ldots, \xi_n\}$ of $(\ref{newEfibration})$, where $W_1(\xi)=  [\Omega(\xi_1-\xi;\gamma_1) | e_{\sha}(-\beta'_1\omega^{(0)}-\gamma_1\nu)],$  and $W_n$ is defined inductively    for $n\geq 2$ by 
   $$W_n(\xi) = [\Omega(\xi_n-\xi;\gamma_n) | e_{\sha} ( -\beta'_n \omega^{(0)} - \gamma_n \nu) \sha W_{n-1}(\xi)]\ .$$
   Formula $(\ref{Hasitint})$ also remains valid in the case when the marked points $\xi_i$, $1\leq i\leq n$, are not necessarily distinct.
  \end{prop}
\begin{proof} The case $n=1$ is essentially equation $(\ref{E1asitint})$. For $n>1$, we have by $(\ref{Gdiff})$:
$$d H_n(\xi)=  \e(\beta_{n} \basepnt + \gamma_{n} r_{\basepnt}) d\EE_1(\xi_n-\xi;\alpha_n,\beta_n) \, H_{n-1}(\xi)$$
and furthermore, $H_n(\basepnt)=0$. The proof of the proposition in the generic case, i.e., when all $\xi_i$ are distinct  follows by induction  just as lemma \ref{lemE1asitint}.  Finally, it follows from lemma $\ref{lemnomonodromy}$ that  $H_n(\xi;\underline{\alpha},\underline{\beta})$ has no singularities along 
$\xi_i=\xi_j$ for $1\leq i< j \leq n$, and so equation $(\ref{Hasitint})$ remains true after degeneration of the arguments $\xi_i$.
\end{proof} 
 Note that both sides of $(\ref{Hasitint})$  have simple poles in the variables $\gamma_1,\ldots,\gamma_n$. Hereafter, extracting the coefficients of a generating series
 such as either side of $(\ref{Hasitint})$ 
 will mean  multiplying   by $\gamma_1\ldots \gamma_n$ and taking the Taylor expansion in $\alpha_i,\beta_i$.

  \subsection{Comparison theorem} 
  Let $\Sigma= \{\sigma_0,\ldots, \sigma_m\}$  be distinct points  on $\E$, where  $\sigma_0=0$. Fix  a basepoint 
  $\basepnt\in \E\backslash \Sigma$. 
Define $\Fo_{\basepnt} ( \E\backslash \Sigma)$ to be the $\Q$-algebra of  multivalued functions  spanned by the function $r-r_{\basepnt}$, and the coefficients of  the functions 
\begin{equation} \label{algfogenerators} 
H_n(\xi;\xi_1,\ldots, \xi_n; \underline{\alpha}, 0), \hbox{ for all } n\geq 1,  \hbox{  where  }  \xi_1,\ldots, \xi_n \in \Sigma\ . 
\end{equation} 
 
For every  $\sigma\in \Sigma$, let us write $ \omega^{(i)}_{\sigma}$ for the coefficients of
$$ \Omega(\xi-\sigma;\alpha) =\sum_{n\geq 0} \omega^{(i)}_{\sigma} \alpha^{i-1} \ ,  $$
and set $\eta_{\sigma}=\omega^{(1)}_{\sigma} -\omega^{(1)}_{\sigma_0}$, for $\sigma \neq 0$.
Recall that  our model $X_{F_n}$ for  the punctured elliptic curve $\E\backslash \Sigma$ is generated by  $\overline{\nu}$ and the $\omega^{(i)}_{\sigma}$ for $i\geq 0, \sigma \in \Sigma$ (lemma \ref{lemXF}).

\begin{thm}  \label{thmcomparisononfiber} The   map  $\int_{\basepnt}^{\xi}: V(X_{F_n}) \rightarrow \Fo_{\basepnt}(\E\backslash \Sigma)$ is an isomorphism.
\end{thm}
\begin{proof}  By $(\ref{expshuffleformula})$,   the integrand  $\overline{W}_n(\xi)$ of  proposition    \ref{propHasitint} 
   can also be written:
 \begin{equation}  \label{newWn}  { [} \overline{ \Omega}(\xi_n-\xi;\gamma_n) |P_n   |\overline{\Omega}(\xi_{n-1}-\xi;\gamma_{n,n-1}) |P_{n-1}|
 \ldots |   \overline{\Omega}(\xi_1-\xi;\gamma_{n,\ldots, 1}) | P_1       ]  
   \end{equation}
where $P_i = e_{\sha} ( -\beta'_{n,\ldots, i} \overline{ \omega}^{(0)} - \gamma_{n,\ldots, i}\overline{ \nu})$ for $1\leq i \leq n$.  The functions  $(\ref{algfogenerators})$ correspond to the constant terms  in   $(\ref{newWn})$ with respect to $\beta_i$, namely the iterated integrals:
  \begin{equation} \label{newWn2} \int_{\basepnt}^{\xi}    [ \overline{\Omega}(\xi_n-\xi;\alpha_n) |     e_{\sha} (  - \alpha_{n}  \overline{\nu})   |
 \ldots |   \overline{ \Omega}(\xi_1-\xi;\alpha_{n,\ldots, 1}) |     e_{\sha} (  - \alpha_{n,\ldots, 1}  \overline{\nu})       ] \ . 
 \end{equation}
One easily checks that $(\ref{newWn2})$ is integrable,  but this also follows  from equation $(\ref{Hasitint})$, since the iterated integral only depends on the endpoint $\xi$, and not 
the path of integration chosen.  
Thus the  coefficients of  $(\ref{newWn2})$ with respect to $\alpha$ lie in $V(X_{F_n})$. 

It  suffices to show that the iterated integral of every element of $V(X_{F_n}) $ arises in this way. 
For this,  choose any numbers $\varepsilon_1,\ldots, \varepsilon_n\in\{0,1\}$. By 
   the multilinearity of bar elements, the  iterated integral  from $\basepnt$ to $\xi$ of any  integrable word  of the form
 \begin{equation}  \label{mainintword} [\overline{\Omega}(\xi_n-\xi;\alpha_n) -\varepsilon_n  \overline{\Omega}(\xi;\alpha_n)|  e_{\sha} ( - \alpha_{n} \overline{\nu})    | \ldots   \qquad \qquad \qquad 
\end{equation} 
 $$\qquad \qquad \qquad \ldots  |   \overline{\Omega}(\xi_1-\xi;\alpha_{n,\ldots, 1}) - \varepsilon_1   \overline{\Omega}(\xi;\alpha_{n,\ldots, 1}) |         e_{\sha} ( - \alpha_{n,\ldots, 1} \overline{\nu})   ]  
   $$
   also lies in $\Fo_{\basepnt}(\E\backslash \Sigma)$. Now let $\pi_{\ell}:V(X_{F_n})  \rightarrow \gr^{\ell} V(X_{F_n})$ 
be the map which projects onto the associated graded for the length filtration, and extended  to power series in the obvious way.  It   kills all  Massey products of weight $\geq 2$. In particular, 
 \begin{eqnarray}  \pi_{\ell} \big(\overline{\Omega}(\sigma-\xi;\alpha ) -   \overline{\Omega}(\xi;\alpha) \big) & =&  - \overline{\eta}_{\sigma} \nonumber \\
    \pi_{\ell}\big( \overline{\Omega}(\sigma-\xi;\alpha ) \big)  &=  & \overline{\omega}^{(0)}Ê\alpha^{-1}  \nonumber
    \end{eqnarray}  
   Applying $\pi_{\ell}$ to $(\ref{mainintword})$  (multiplied by $\alpha_1\ldots \alpha_n$ to clear the poles in $\alpha_i$) gives a generating series in $\alpha_1,\ldots, \alpha_n$ whose coefficients are  all words of the form
   \begin{equation}\label{wordweget}   m_k \overline{\nu}^{i_k}         \ldots    m_2 \overline{\nu}^{i_2} m_1 \overline{\nu}^{i_1}    \end{equation}
   where $i_1,\ldots, i_k$ are any non-negative integers, and 
   $$ m_i =      \left\{
                           \begin{array}{ll }
                            \overline{\eta}_{\xi_i} & \hbox{ if  } \varepsilon_i = 1    \nonumber \\
                             \overline{\omega}^{(0)} & \hbox{ if  } \varepsilon_i = 0   \ 
                           \end{array}
                         \right. 
                        $$
It is easy to verify that every word in $\{ \overline{\nu},  \overline{\omega}^{(0)},  \overline{\eta}_{\sigma_1}, \ldots, \overline{\eta}_{\sigma_m} \}^{\times}$ is a linear combination of  shuffle products of  $\overline{\nu}\ldots \overline{\nu}$ with elements
$(\ref{wordweget})$. 
It follows from the description  (see proposition \ref{propwordlift} and preceding discussion):
$$  \gr^{\ell} V(X_{F_n}) \cong T(\Q\overline{\nu} \oplus \Q \overline{\omega}^{(0)} \oplus \Q  \overline{\eta}_{\sigma_1} \oplus \ldots \oplus  \Q  \overline{\eta}_{\sigma_m})\ ,$$ 
that the iterated integral of every element in $V(X_{F_n})$ appears  a linear combination of  products of the function 
$$r-r_{\basepnt}=\int_{\basepnt}^{\xi} \overline{\nu} $$
 with coefficients of $(\ref{mainintword})$. This completes the proof.
\end{proof}
In particular, every iterated integral on $\E^{\times}$ can be obtained in this way.    Our model $V(X_1)$ defines a $\Q$-structure on the de Rham fundamental groupoid
of $\E^{\times}$, hence:

 \begin{cor} The periods of the prounipotent  fundamental groupoid ${}_{\basepnt}\Pi_{\xi}(\E^{\times})$  for any   initial point $\basepnt\in \E^{\times}$ and endpoint $\xi\in \E^{\times}$, lie in the  $\Q$-algebra generated by  $r-r_{\basepnt}$, and   the coefficients of  $(\ref{algfogenerators})$ with respect to the $\alpha_i$'s.
\end{cor} 
 
\begin{rem} The previous corollary concerns a general complex elliptic curve. If $\E^\times$ happens to be defined over $\Q$, then the  $\Q$-structure on the de Rham fundamental groupoid
induced from the $\Q$-structure on $\E^\times$ is is not the same as the one we considered above. These two $\Q$-structures could in principle be compared using \cite{LR}, \S$5$.

\end{rem}

\subsection{Generalizations}
One can  extend theorem \ref{thmcomparisononfiber} to the case where $\basepnt$ is a tangential basepoint  at one of the points $\sigma\in \Sigma$. As a result, a higher-dimensional version of  theorem    \ref{thmcomparisononfiber} can also be deduced from theorem \ref{thmBartensorstructure}, which states that the iterated integrals on  the configuration space $\En$
are  products of iterated integrals on the fibers of the map $\En\rightarrow \E^{(n-1)}$, which is  the one-dimensional case treated above. Therefore all iterated integrals on 
$\En$ can be obtained from our averaging procedure.
\\

\end{document}